\documentclass[10pt,reqno]{amsart}

 \usepackage{amsmath,amsfonts,amssymb,amscd,amsthm,amsbsy,bbm, epsf,calc,enumerate,color,graphicx,verbatim,cite,import}
\usepackage{color}
\usepackage{datetime,appendix}
\usepackage{epstopdf}

\usepackage{float}

\usepackage{ifpdf}
\ifpdf
    \usepackage[colorlinks,citecolor=black,linkcolor=black]{hyperref}
\fi

\theoremstyle{plain}

\newtheorem{thm}{Theorem}[section]

\newtheorem{mainthm}{Theorem}

\newtheorem{lem}[thm]{Lemma}
\newtheorem{lemma}[thm]{Lemma}
\newtheorem{cor}[thm]{Corollary}

\newtheorem{prop}[thm]{Proposition}

\theoremstyle{definition}

\newtheorem{DEF}[thm]{Definition}

\theoremstyle{remark}                  
\newtheorem{rem}[thm]{Remark}
\newtheorem{remark}[thm]{Remark}

\def\F{{\mathcal F}}

\def\R{{\mathbb R}}

\def\real{{\mathbb R}}

\def\indicator{\mathbf{1}}
\def\integer{{\mathbb Z}}

\def\e{\varepsilon}

\def\Per{\textnormal{Per}}

\def\grad{\nabla}

\def\div{\textnormal{div}}

\def\dside{\partial_{\textnormal{side}}}

\definecolor{darkgreen}{rgb}{0,0.4,0}

\DeclareMathOperator*{\argmin}{\arg\!\min}

\numberwithin{equation}{section}

\DeclareRobustCommand{\SkipTocEntry}[5]{}

\def\XXint#1#2#3{{\setbox0=\hbox{$#1{#2#3}{\int}$ }
\vcenter{\hbox{$#2#3$ }}\kern-.6\wd0}}

\begin{document}
\title{Liquid Drops on a Rough Surface}
\author{William M. Feldman}
\email{feldman@math.uchicago.edu \vspace{-.5\baselineskip}}
\thanks{W. M. Feldman was partially supported by NSF-RTG grant DMS-1246999.}
\address{Department of Mathematics, The University of Chicago, Chicago, IL 60637, USA \vspace{-.5\baselineskip}}
\author{Inwon C. Kim}
\email{ikim@math.ucla.edu \vspace{-.5\baselineskip}}
\address{Department of Mathematics, UCLA, Los Angeles, CA 90095, USA}
\thanks{I. C. Kim was partially supported by NSF grant DMS-1566578.}
\maketitle
\begin{abstract}
We consider a liquid drop sitting on a rough solid surface at equilibrium, a volume constrained minimizer of the total interfacial energy.  The large-scale shape of such a drop strongly depends on the micro-structure of the solid surface.  Surface roughness enhances hydrophilicity and hydrophobicity properties of the surface, altering the equilibrium contact angle between the drop and the surface. Our goal is to understand the shape of the drop with fixed small scale roughness. To achieve this, we develop a quantitative description of the drop and its contact line in the context of periodic homogenization theory, building on the qualitative theory of Alberti and DeSimone \cite{AlbertiDeSimone}. 
\end{abstract}

\section{Introduction}
We consider the interfacial energy of a configuration of solid, liquid and vapor occupying complementary regions $S$, $L$ and $V$ of $\real^{d+1}$. Points in $\real^{d+1}$ will be denoted by $(x,z) \in \real^d \times \real$.  We will take the solid region $S$ to to be an almost flat surface with small scale roughness,
 \begin{equation}\label{eqn: solid}
 S_\e = \{ (x,z) \in \real^d \times \real : z \leq \e\phi(\tfrac{x}{\e})\}.
 \end{equation}
Here $\phi : \real^{d} \to \real$ will be at least upper semi-continuous, bounded variation and $\integer^{d}$ periodic.  We take as a normalization $\max \phi = 0$.  Since the solid region $S$ is fixed by \eqref{eqn: solid}, the entire configuration is determined by either $L$ or $V$.  We will put the focus on the shape of the minimizing liquid drop $L$ with the volume constraint $|L| = \textup{Vol}$.  The total energy of a given configuration is,
\begin{equation}\label{main}
{{E}}_\e(L) := \sigma_{\textup{LV}}| \partial L \cap \partial V | + \sigma_{\textup{SV}}|\partial S_\e \cap \partial V| + \sigma_{\textup{SL}}|\partial S_\e \cap \partial L|,
\end{equation}
where $\sigma_{\textup{LV}}$, $\sigma_{\textup{SV}}$ and $\sigma_{\textup{SL}}$ are positive constants that represent, respectively, the liquid-vapor, solid-vapor and solid-liquid interfacial energies per unit surface area.   For now this energy should be interpreted formally, $|\cdot |$ refers to the surface measure, later we will be more rigorous and interpret this energy on the space of finite perimeter sets.  We also remark that, to be precise, $E_\e$ should really be evaluated only on compact (in $x$) subsets of $\real^{d+1}$.  See Section~\ref{sec: energy} for more details.

\medskip

  We refer to the three phase interface $\partial L \cap \partial S \cap \partial V$, which is formally a $d-1$ dimensional set, as the \emph{contact line}.  Smooth local minimizers of \eqref{main} satisfy a so-called contact angle condition along their contact line.  The prescribed angle between the liquid-vapor interface and the solid-liquid interface (see Figure~\ref{fig: apparent angle}) is referred to as the {\it Young contact angle} $\theta_Y$ associated with the energy and is given by,
  \begin{equation}\label{young}
\cos \theta_Y := \frac{\sigma_{\textup{SL}}-\sigma_{\textup{SV}}}{\sigma_{\textup{LV}}} \ \hbox{ with } \ \theta_Y \in [0,\pi].
\end{equation}
We say that the surface is \emph{hydrophobic} when $\sigma_{\textup{SL}} \geq \sigma_{\textup{SV}}$, i.e. when $\cos \theta_Y \geq 0$, and we say that the surface is \emph{hydrophilic} when $\sigma_{\textup{SL}} \leq  \sigma_{\textup{SV}}$, i.e. when $\cos \theta_Y \leq 0$.  We make note that our convention on the sign of $\cos\theta_Y$, i.e. the contact angle is measured from the outside of the liquid phase, is not standard.

\medskip

In the setting of completely flat surface, i.e. when $\phi \equiv 0$, the unique (modulo translation) global minimizers are spherical caps with contact angle $\theta_Y$ along the contact line. In the general rough surface setting the shape of the minimizer is more difficult to describe.  At the microscopic scale $\e$ the volume constrained global minimizers $L_\e$ of the interfacial energy $E_\e$ will satisfy the contact angle condition with angle $\theta_Y$ along their microscopic contact line. Our goal is to understand the global shape of $L_\e$ for a small, fixed, roughness scale $\e$. More precisely we are interested in (a) the appearance of a \emph{macroscopic} or \emph{apparent} contact angle measured at a large scale, and (b) the precise size of this larger scale or, similarly, what is the size of the {\it boundary layer} caused by the propagating influence of the small scale roughness.

 \begin{figure}[t]
  \centering
\begin{minipage}{.5\textwidth}
  \centering
 \def\svgwidth{3in}
  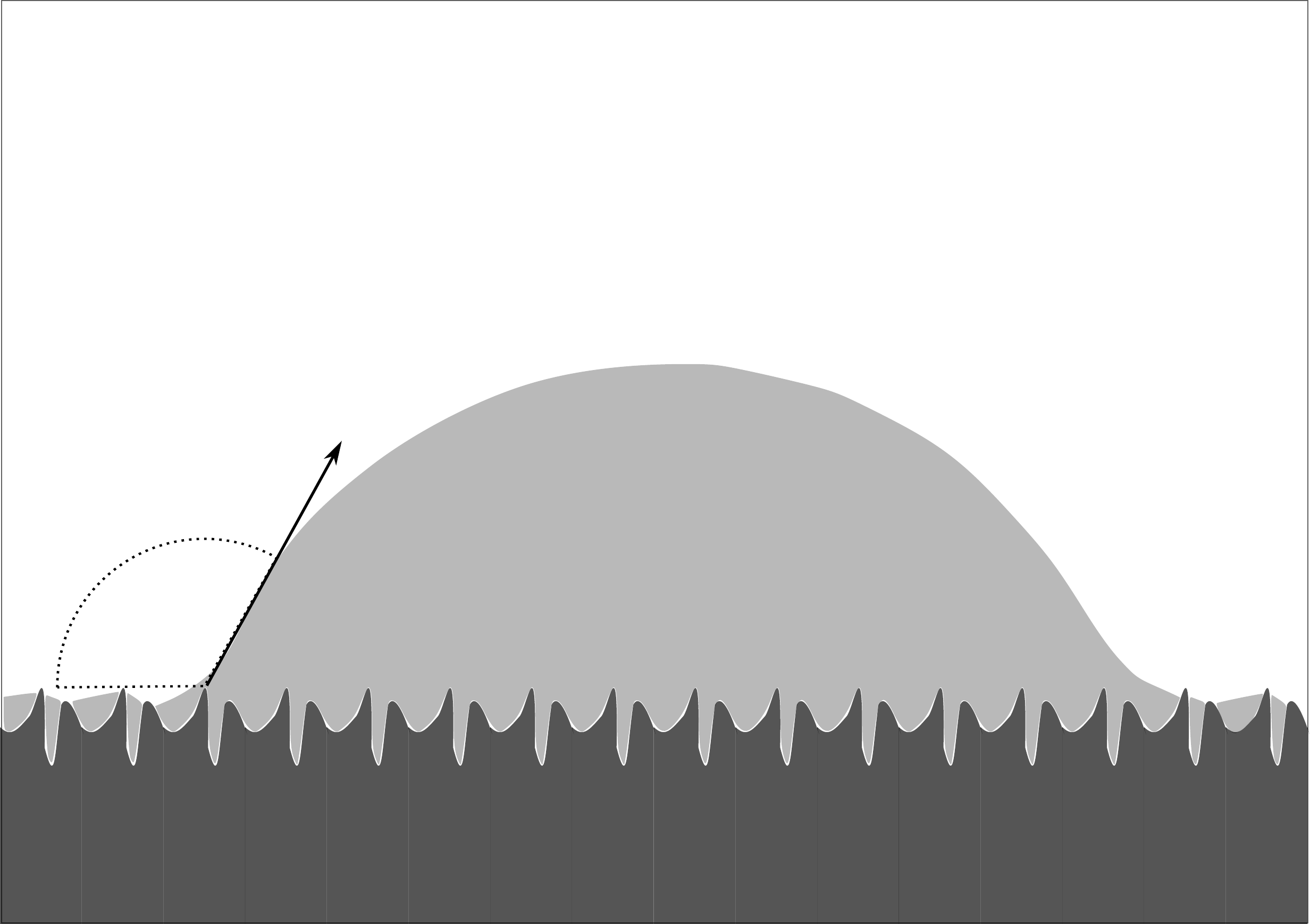
\end{minipage}%
\begin{minipage}{.5\textwidth}
\centering
 \def\svgwidth{3in}
  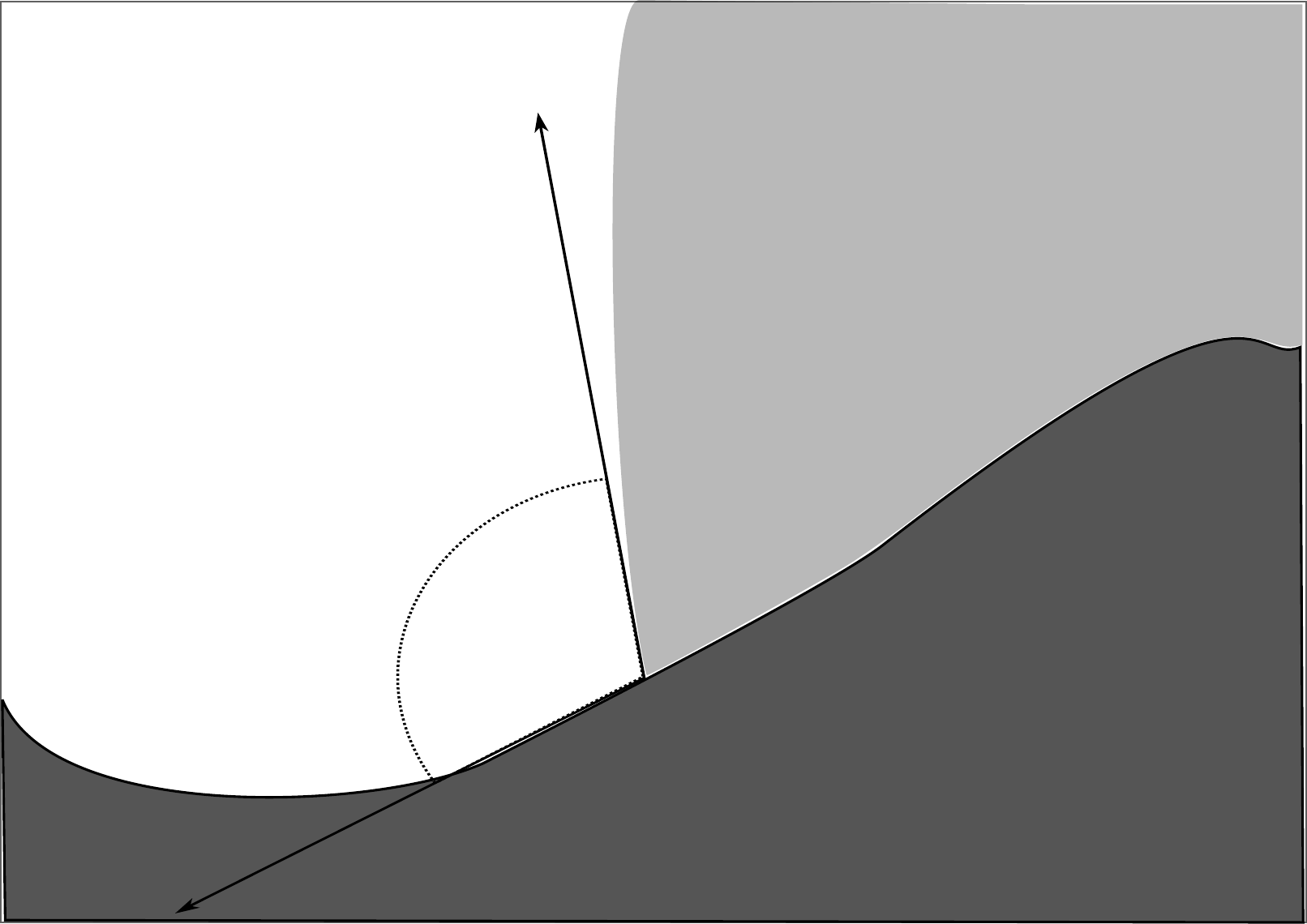
\end{minipage}
  \caption{Left: A liquid droplet on a rough hydrophilic surface, the effective/apparent contact angle is displayed. Right: The microscopic contact line and the Young contact angle are displayed.}
  \label{fig: apparent angle}
 \end{figure}

\medskip

For the global minimizer, the effective contact angle is uniquely determined by cell problems in the bulk of the wetted set and non-wetted sets. This connection between the homogenized contact angle for global minimizer and homogenization theory was first introduced by Alberti and DeSimone \cite{AlbertiDeSimone}.  
There are two cell problems, one for the effective solid-liquid interaction energy and one for the effective solid-vapor interaction energy.  The effective solid-vapor interaction energy $\overline{\sigma}_{\textup{SV}}$ is the minimal energy per unit area which can be achieved by interposing a liquid region in between the solid region $S$ and a vapor region $V \supset \{z \geq 0\}$.  Symmetrically, the effective solid-liquid interaction energy $\overline{\sigma}_{\textup{SL}}$ is the minimal energy per unit area which can be achieved by interposing a vapor region in between the solid region $S$ and a vapor region $L \supset \{z \geq 0\}$.  
See Section~\ref{sec: cell} for a detailed derivation and rigorous definition of the cell problems.  Given the effective interaction energies one can also define the effective contact angle,
\[ \cos \overline{\theta}_Y = \frac{\overline{\sigma}_{\textup{SL}}-\overline{\sigma}_{\textup{SV}}}{\sigma_{\textup{LV}}}.\]
 Different forms of the surface roughness lead to different optimal configurations for the two cell problems, which lead to different effective contact angles.  What one can show (see \cite{AlbertiDeSimone}), without any quantitative information about the volume constrained minimizers, is the $\Gamma$-convergence of $E_\e$ to the homogenized energy,
\begin{equation}\label{hom_energy}
\overline{{E}}(L):= \sigma_{\textup{LV}}|\partial L \cap\partial V| + \overline{\sigma}_{\textup{SL}}|\partial S_0\cap \partial L|+ \overline{\sigma}_{\textup{SV}}|\partial S_0\cap \partial V|,
\end{equation} 
 with the flat solid surface $S_0 = \{ z\leq0\}$.  However, exactly in what sense $\cos\overline{\theta}_Y$ relates to the apparent contact angle of the volume constrained minimizers $E_\e$ for a fixed positive $\e$ is not clear from the cell problems or from the qualitative homogenization result.  Clarifying this question is one of our main goals (see Theorem~\ref{thm: main}).

\medskip

We emphasize that the effective contact angle is not determined by a cell problem at the contact line, but rather in the bulk of the wetted set and non-wetted sets.  Identifying local minimizers or meta-stable states is more difficult, since then the cell problem is at the contact line and requires to actually solve the Euler-Lagrange equation which is a free boundary problem for the minimal surface equation.  For this article we restrict our attention to global minimizers.

 \medskip

 This problem is frequently discussed in the engineering literature, in particular in the context of a liquid droplet sitting on surface with periodic flat-topped pillars (see literature in Section~\ref{sec: lit} below).  Two, by now classical, ``models"  were developed to derive the value of the effective contact angle: the {\it Wenzel} model \cite{Wenzel1936} and the {\it Cassie-Baxter} model \cite{cassie1944}. We will describe the two models in the hydrophilic case, the hydrophobic case can be understood by symmetry -- switching the roles of $L$ and $V$.  In the Wenzel model one supposes that the liquid phase fills in the grooves of the solid completely and so the (approximate) energy per unit area under the wetted region is given by,
 \[ \cos \overline{\theta}_Y = \cos\theta_Y \int_{[0,1)^d} \sqrt{1+|D\phi|^2} \ dx.\]
 In the Cassie-Baxter model one supposes the the liquid spreads out filling in the grooves of the rough surface and creating a new ``effective" surface with a varying contact angle $\cos\theta_Y(x) = -1$ if $x \in \{ \phi <0\}$ (which is the liquid-liquid interaction energy) and on the maximum set of the solid surface $\{ \phi = 0\}$ the contact angle is unchanged $\cos\theta_Y(x) = \cos\theta_Y$.  Calling $f$ to be the area fraction of the unit cell taken up by $\{ \phi = 0\}$ this leads to an effective contact angle of,
 \[ \cos\overline{\theta}_Y = \cos\theta_Yf - (1-f).\]
 Empirically both of these two laws can be observed depending on the degree of roughness of the surface and on the wetting properties of the solid.  
  
 \medskip

In the context of the cell problems the Cassie-Baxter and Wenzel models should in fact be interpreted as test configurations for the two minimization problems for $\overline{\sigma}_{\textup{SL}}$ and $\overline{\sigma}_{\textup{SV}}$, see Section~\ref{sec: values} for details.  In fact the true minimizing states can be something much more complicated than the Wenzel or Cassie-Baxter states.  Nonetheless, again shown in \cite{AlbertiDeSimone}, when $\phi$ is Lipschitz continuous and $\cos\theta_Y$ is sufficiently small (depending on the Lipschitz constant of $\phi$) the Wenzel model is exactly correct.  We give a slight generalization of this fact below (see Lemma~\ref{lem: wenzel}) which covers more general surfaces including the case of periodically repeated flat-topped pillars. 
One may also wonder, for the particular type of surfaces where this statement makes sense, whether the Cassie-Baxter model is exact when $\cos\theta_Y$ is close to $1$.  In $d=1$ one can make explicit computations to find that this occurs \cite{AlbertiDeSimone}. We were able to show the analogous result in $d=2$ (see Theorem~\ref{thm:main0} below), which is the most physically relevant case.

\medskip

  A question which naturally occurs in our analysis is whether a hydrophobic surface can be made completely de-wetting, i.e. $\cos\theta_Y<1$ but $\cos \overline{\theta}_Y =1$, by surface roughness.  Symmetrically one can ask whether a hydrophilic surface can be made completely wetting, i.e. $\cos\theta_Y = -1$, by surface roughness.  
 The answer is no, this has been shown in \cite{AlbertiDeSimone} for a certain class of surfaces, when $| \partial S \cap \{ z= 0\}| >0$ in every unit cell.  We prove that $|\cos\overline{\theta}_Y|$ is bounded away from $1$ under a slightly relaxed assumption on the solid surface and we give a proof which allows for \emph{chemically textured} rough surfaces ($\sigma_{\textup{SL}}$ and $\sigma_{\textup{SV}}$ also depend on $x$).

 \medskip

We state as our first main result, a collection of results about the cell problems and the effective contact angle.

\begin{mainthm}[see Proposition~\ref{prop: contact angle quantity non-degen} and Section~\ref{sec: values}]\label{thm:main0}
Let $\cos\theta_Y$ be as given in \eqref{young} and let $\cos\bar{\theta}_Y$ be the homogenized contact angle defined via the cell problems in \eqref{hom_energy}. 
\begin{enumerate}[$(a)$]
\setlength\itemsep{.5em}
\item $|\cos \overline{\theta}_Y| <1-c_0$, where $c_0$ depends on $\cos \theta_Y$ and the regularity of $\phi$ near its maximum. 
\item There exist $\integer^d$-periodic sets $L_{\textup{SL}}$ and $L_{\textup{SV}}$ which have energy exactly $\overline{\sigma}_{\textup{SL}}$ and $\overline{\sigma}_{\textup{SV}}$ in every unit cell. 
\item Suppose that $d+1=2$ or $3$, and  $\phi(x) = M(1-\indicator_{P})$, with $P$ a smooth boundary periodic subset of $\real^d$,  i.e.  the surface is made of periodically arrayed flat-topped pillars.  Then the following holds:
\begin{enumerate}[$(i)$]
\item The function $\cos\theta_Y  \mapsto \cos\bar{\theta}_Y$ is a concave function, whose value and slope are bounded by Wenzel and Cassie-Baxter states.
\item The Cassie-Baxter bound is obtained if either $1-|\cos \theta_Y|$ is small or if $M$ is large. On the other hand the Wenzel bound is obtained if $|\cos\theta_Y|$ or $M$ is sufficiently small. 
\end{enumerate}
\end{enumerate}
\end{mainthm}

\medskip

The next step is to understand the behavior of the minimizers of the volume constrained rough surface problem.    Below we state our result on the quantitative homogenization, namely convergence rate of the volume constrained minimizers. This is the more significant part of our paper, both in analytical difficulty and in novelty.
 
  \medskip

In order to compare, quantitatively, the energies $E_\e$ and $\overline{E}$, we need a large-scale regularity theory. More precisely we want to show that the contact line, at least when viewed at a sufficiently large length scale, is $d-1$ dimensional. This would justify the formal arguments that we used to derive the cell problem.  At a very basic level we follow the strategy in Caffarelli and Mellet \cite{CM}, where the authors study capillary drops sitting on flat but chemically textured surface with $\sigma_{\textup{SL}}, \sigma_{\textup{SV}}$ being $\e$-periodic functions of $x$ and $\phi\equiv 0$.  The rough surface presents a significant challenge in the regularity analysis. 

\medskip

The second part of our regularity analysis is to prove the non-degeneracy of the contact angle, again when viewed at a sufficiently large length scale.  This estimate allows us to upgrade convergence in measure to uniform convergence outside of a certain boundary layer.  In particular we get a bound on the length scale one needs to zoom out to (starting from the microscopic length scale $\e$) in order to see the homogenized contact angle appear.  The precise size of this boundary layer is still out of reach to us, but we make the first progress in this direction.

\medskip

 The statement of the Theorem is somewhat imprecise, mainly because the energies $E$ and $\overline{E}$ really need to be evaluated on a compact (at least in $x$) region of $\real^{d+1}$, we gloss over the dependence on the domain.

\begin{mainthm}[see Theorem~\ref{thm: main periodic}]\label{thm: main} 
Let $\rho_0(\cos\overline{\theta}_Y,\textup{Vol}), z_0(\cos\overline{\theta}_Y,\textup{Vol})$ so that $L_0 :=B^+_{\rho_0}(0, z_0)$ is a volume constrained minimizer of \eqref{hom_energy} for every $x \in \real^d$. Let $L_\e$ with $|L_\e| = \textup{Vol}$ be the volume constrained minimizer associated with \eqref{main}, then the following holds:
\begin{enumerate}[(i)]
\item \textup{(Convergence of the energy)}
\[|E_\e(L_\e) - \overline{E}(L_0)| \leq C \e^{1-o(1)}.\]
\item \textup{(Convergence in measure)}  The sets $L_\e$ converge in measure, modulo translation, to the globally minimizing spherical cap,
\begin{equation}
\min_{x \in \real^d} |L_\e \Delta (L_0+x)| \leq C\e^{\alpha(1-o(1))},
\end{equation}
where $ 0 <\alpha \leq 1/2$ is the exponent from the $L^1$ stability estimate Theorem~\ref{stability}.
\item \textup{(The size of the boundary layer)} Call $\beta = \frac{2d}{(d+1)(d+2)}$ and $h_0(\e) = C\e^{\beta+\frac{(1-\beta)\alpha}{d+1}-o(1)}$ then,
 \[ \min_{x \in \real^d} \tfrac{1}{\textup{Vol}^{\frac{1}{d+1}}}d_H\big(L_\e \cap \{ z \geq h_0(\e)\},  (L_0+x) \cap \{ z \geq h_0(\e)\} \big) \leq C \e^{\alpha(1-o(1))}.   \] 
 In words, $L_\e$ converges {uniformly}, outside of a boundary layer of size $h_0(\e)$, to some global minimizer of the homogenized problem. In $d+1=2$ this is $h_0(\e) \sim \e^{(1+\alpha)/3-o(1)}$ and in $d+1 = 3$ this is $h_0(\e) \sim \e^{(3+2\alpha)/9 - o(1)}$.
 \end{enumerate}
\end{mainthm}

\medskip

We remark that the methods of this paper can be generalized to consider the problem with a random surface $S$ satisfying certain stationarity and ergodicity assumptions.  This problem is the topic of a forthcoming work by the authors.

\medskip

All of our arguments can be easily adapted to the case of $x$-dependent $\sigma$'s as long as $\sup_x |\cos \theta_Y|(x) <1$.

 \subsection{Literature}\label{sec: lit}

There is a vast amount of literature in the science and engineering community on wetting phenomena on textured surfaces since the contributions of Wenzel \cite{Wenzel1936} and of Cassie and Baxter \cite{cassie1944}: see \cite{Barthlott1997,NEINHUIS1997,bico1999pearl,patankar, mchale, bormashenko, murakami2014wetting} and see the book by de Gennes et al \cite{deGennes} for further discussion of various aspects of wetting and for more references.  

\medskip

There has also been a large amount of mathematical literature on energy minimizing capillary drops.  We only describe some of the most recent and most relevant to our current work.  Only one mathematical work, that we are aware of, has directly studied the rough surface homogenization problem.  That is the work of Alberti and DeSimone \cite{AlbertiDeSimone} which we have described in detail above.   For flat and chemically textured surfaces, parallel convergence results to our main theorem are shown by Caffarelli and Mellet \cite{CM} for the periodic setting, and by Mellet and Nolen \cite{MelletNolen} for the random case.  A second paper of Caffarelli and Mellet \cite{CMhysteresis} showed the existence of local minimizers which exhibit non-constant large-scale contact angle, this is a manifestation of hysteresis. 

\medskip

For the volume constrained global minimizer of the energy given by \eqref{main} with a smooth hypersurface $S$, De~Philippis and Maggi \cite{DePMag, DePMag14a} proved $C^{1,1/2}$ regularity of the contact line and validity of Young's law \eqref{young} away from a lower dimensional singular set.  

\medskip

As for dynamic description of evolving liquid drops, many different models are available: see for instance  \cite{desimone2006new, dohmen2008micro, alberti2011quasistatic, glasner, grunewald2011variational, kim2014liquid}. In general, regularity near the contact line, or even the topology of the drop, is largely unknown for drops that are not global minimizers except drops with strong geometric properties (for example see Feldman and Kim \cite{feldmankim}).  This is a major challenge for global in time analysis on sliding, spreading or retracting drops beyond one dimension.

\subsection{Outline of the paper} We begin with introducing some notions from geometric measure theory in Section~\ref{sec: notation}. Then Sections~\ref{sec: cell}-\ref{sec: values} analyzes minimizers of the cell problem, while Section~\ref{sec: vol}-\ref{sec: hom} discusses the volume constrained problem.

\medskip

In Section~\ref{sec: cell} we introduce the assumptions for the solid surface and prove well-posedness of the cell problems, existence of a periodic minimizer, and properties of the homogenized contact angle. In particular we show that while the effect of wetting and de-wetting is enhanced by homogenized contact angle, it does not change the problem drastically at least with our assumptions, in the sense that if one starts with hydrophobic or hydrophilic coefficients with $|\cos\theta_Y|<1$, then the homogenized problem stays in the same regime, and furthermore the homogenized contact angle satisfies $|\cos\overline{\theta}_Y|<1$, i.e. it stays away from total wetting/de-wetting (Lemma 3.20). 

\medskip

Section~\ref{sec: values}  discusses specific values of homogenized contact angle in physical (three) dimensions, in the case of a specific surface given by periodic dents. In particular we show that Cassie-Baxter and Wenzel states are achieved as the optimal states, when the roughness is large (Cassie-Baxter) or small (Wenzel).

\medskip

In Section~\ref{sec: vol} we introduce the volume constrained minimizer of the interfacial energy $E_\e$ given in \eqref{main}.  We state some interior regularity results, i.e. away from the solid surface, for (almost) minimal surfaces. Beyond just the basic density estimates, we also need a higher regularity result.  We choose to draw from the flatness results in Caffarelli and Cordoba \cite{caffarellicordoba} since the form of their result is convenient for our purposes.

\medskip

In Section~\ref{sec: reg} we investigate the large-scale regularity of volume-constrained minimizers near the contact line.   Our first estimate states that the (macroscopic) $d-1$ dimensional measure of the ``macroscopic contact line" is bounded down to scale $\e^{1-o(1)}$ (Proposition~\ref{prop: perim 0 t}). It should be pointed out our estimate is near-optimal, since possible vapor bubbles in the grooves of the rough surface or liquid precursor films force us to consider the macroscopic contact line at scales larger than the microscopic length scale $\e$. Our second regularity result is on the, again large-scale, nondegeneracy of the drop near the contact line (Proposition~\ref{prop: non-degen}), which allows us to obtain density estimates almost up to the solid surface (Corollary~\ref{cor: uniform non-degen}). 

\medskip

Finally in Section~\ref{sec: hom} addresses the rate of convergence for the minimizers as the roughness scale $\e$ goes to zero. The convergence of the energies $E_\e(L_\e)$ to $\overline{E}(L_0)$ at rate $\e^{1-o(1)}$ is a consequence of the perimeter estimate and is almost optimal. From the energy convergence we obtain convergence in measure of the $L_\e$ using the stability result of the spherical cap minimizers of the homogenized energy $\overline{E}$ on a flat surface.  Finally we use the up to the boundary density estimates of Section~\ref{sec: reg} to obtain convergence in Hausdorff distance, at least outside of a boundary layer (see Theorem~\ref{thm: main} for the size of the boundary layer).




\section{Notations and background material}\label{sec: notation}

In this section we will explain various notations that will be in place throughout the paper.  We will also give some background material about functions of bounded variation and sets of finite perimeter.  Such spaces are the natural setting for variational problems involving surface area optimization. Then we will discuss some standard normalizations of the energy, and give a precise meaning to the energy in the setting of finite perimeter sets.

\subsection{Constants and parameters}  We will use $C, c>0$ to refer to constants which can change from line to line, typically it will be the case that $C\geq1$ and $c \leq 1$.  We say that a constant $C$ is \emph{universal} if it depends only on the fixed (unitless) parameters of the problem which are $d$, $\cos\theta_Y$ and the function $\phi$ defining the solid surface $S$.  We will sometimes make clear the dependence on these constants anyway when it is important or interesting.  In particular universal constants will not depend on parameters with units like the volume of the liquid drop $\textup{Vol}$, the liquid-vapor energy per unit surface area $\sigma_{\textup{LV}}$, or the size (to be made precise) of the confining region $\Omega$.  Lastly we call a constant \emph{numerical} if it does not depend on any parameter of the problem.  Sometimes it will be preferable to hide the universal constants, for two quantities $A$ and $B$ we write
\[ A \lesssim B \ \hbox{ if there is a universal $C$ so that } \ A \leq CB\]
and we write $A \sim B$ if $A\lesssim B$ and $B \lesssim A$.

\subsection{Notation of basic sets}\label{subsec: notations}  The basic setting of our problem will be in $d+1$-dimensional Euclidean space $\real^{d+1}$.  The axis unit vectors are denoted by $e_1,\dots e_{d+1}$.  We will denote points of $\real^{d+1}$ by $(x,z) \in \real^d \times \real$, $x$ will always refer to a point of $\real^d$ and $z$ to the height in the $e_{d+1}$ direction.  The notation $| x |$ or $|(x,z)|$ will refer to the Euclidean distance on $\real^d$ or on $\real^{d+1}$ respectively.  

\medskip

For any subset $A \subseteq \real^{d+1}$ we call $A^+ := A \cap \{ z >0 \}$.  

\medskip

For $(x,z) \in \real^d \times \real$ we define {\it balls} of $\real^{d+1}$ and {\it disks} of $\real^d$,
\[ B_r(x,z) := (x,z)+\{(y,w) \in \real^{d+1}: \ |(y,w)| <r \} \ \hbox{ and } \ D_r(x) := x+\{ y \in \real^d : \ |y| < r\}.\]
    Similarly we define {\it cubes} of $\real^{d+1}$ and {\it squares} of $\real^d$,
\[ Q_r(x,z) := (x,z)+(-r/2,r/2)^{d+1} \ \hbox{ and } \ \Box_r(x) := x+(-r/2,r/2)^d \]
We may wish sometimes to view $D_r$ and $\Box_r$ as embedded into $\real^{d+1}$, for $t \in \real$ we define,
\[ D_r^t := D_r \times \{ z= t\} \ \hbox{ and } \ \Box_r^t := \Box_r \times \{ z= t\}.\]
Depending on the context, and only if there will be no confusion, the notations $D_r$ (resp. $\Box_r$) may refer either to the subsets of $\real^d$ defined above or to $D_r^0$ (resp. $\Box_r^0$) as a subset of $\real^{d+1}$.  

\medskip

We will also define the cylindrical sets based on $D_r$ and $\Box_r$, for $a <b$ in $\real$,
\[ D_r^{a,b}:= D_r \times (a,b) \ \hbox{ and } \ \Box_r^{a,b} := \Box_r \times (a,b).\]
Evidently the boundaries of $D_r^{a,b}$ and $\Box_r^{a,b}$ are made up by the disjoint unions,
\[ \partial D_r^{a,b} = D_r^a \cup D_r^b \cup \dside D_r^{a,b} \ \hbox{ and } \ \partial\Box_r^{a,b} =\Box_r^a \cup \Box_r^b \cup \dside\Box_r^{a,b},\]
where the the {\it lateral boundaries} are $\dside D_r^{a,b} := \partial D_r \times [a,b]$ and $\dside\Box_r^{a,b}:= \partial \Box_r \times [a,b]$.

\medskip

For balls and cubes we define the intersections with the upper half-space,
\[B_r^+(x,z) = B_r(x,z) \cap \{ z >0\} \ \hbox{ and } \ Q_r^+(x,z) = Q_r(x,z)  \cap \{ z >0\}.\]

\subsection{Notation of measures}  We will need to deal with measures of sets in dimensions $d-1$, $d$ and $d+1$. For a set $A$ of $\real^n$ we will write $\mathcal{H}^k(A)$ for the $k$-dimensional Hausdorff measure.  The $0$-dimensional Hausdorff measure (i.e. the counting measure) will be denoted $\#(A)$.  

\medskip

When it is manifestly clear that $A$ is a set of Hausdorff dimension $k$ will often abuse notation and write $|A|$ instead of $\mathcal{H}^k(A)$.  For example, take $D_r = D_r(0)$ is the disk in $\real^d$ of radius $r$ we will write both,
\[ \mathcal{H}^d(D_r) = |D_r| \ \hbox{ and } \ \mathcal{H}^{d-1}(\partial D_r) = |\partial D_r|,\]
since the meaning of the notation $| \cdot |$ should be unambiguous in this context.

\subsection{Sets of finite perimeter}  The following material can mostly be found in the books \cite{Giusti,MaggiBV}.  We recall that the $BV$-norm of a function $f : \real^{n} \to \real$ on an open set $U$ is defined by,
\[ \int_{U} |Df| \ dy := \sup \left\{ \int_{\real^n} f \div g \ dy :  \ g \in C_c^\infty(U;\real^n) \hbox{ with } |g| \leq 1\right\}.  \]
A function is {\it locally of bounded variation} if the above norm is finite for every bounded open set $U$.  A set $L$ is called a {\it Cacciopoli set} or a {\it set of finite perimeter} if $\indicator_L$ has finite $BV$-norm on $\real^n$.  Similarly one can defined sets of locally finite perimeter.  We define,
\[ \Per(L,U) := \int_{U} |D\indicator_L| \ dy\]
If $f : \real^n \to \real$ is locally of bounded variation then the set $L = \{ (y,w) \in \real^{n+1}: w \leq f(y)\}$ which is the sub-graph of the function $f$ is a set of locally finite perimeter.  Furthermore we can define the surface area of the graph of $f$ over an open region $U$ see \cite[Chapter 14]{Giusti},
\begin{equation}\label{area} \int_U \sqrt{1+|Df|^2} \ dy := \sup \left\{ \int_{\real^n} g_{n+1}+f \sum_{i=1}^n\frac{\partial g_i}{\partial y_i} \ dy :  \ g \in C_c^\infty(U;\real^{n+1}) \hbox{ with } |g| \leq 1\right\}, 
\end{equation}
and we have that,
\[ \Per(L,U \times \real ) =   \int_U \sqrt{1+|Df|^2} \ dy.\]
In what follows we will need several times the following refinement of this same idea,
\begin{lem}\label{lem: surface area to BV}
Let $f$ be a bounded variation function on a bounded open set $U \subset \real^n$, then
\[\int_U \sqrt{1+|Df|^2}-1 \ dy \geq a\int_{U} |Df| \ dy - a^2|U| \ \hbox{ for every } \ a \in [0,1].\]
\end{lem}
The quantity on the left hand side naturally arises when computing the change in surface area under replacing the subgraph $L$ by the flat subgraph $\{(y,w): w \leq 0\}$ over the set $U$.
\begin{proof}
Using \eqref{area}, we restrict the class of $g$ such that  $(g',g_{n+1})$ with $|g'| \leq a$ and $|g_{n+1}| \leq \sqrt{1-a^2}$ to get a lower bound on the supremum, 
\begin{equation*}
\begin{array}{lll}
\int_U \sqrt{1+|Df|^2}-1 \ dy & \geq & \sup \left\{ \int_{\real^n} f \grad \cdot g' \ dy :  \ g' \in C_c^\infty(U;\real^n) \hbox{ with } |g'| \leq a\right\} - \int_U (1-\sqrt{1-a^2}) \ dy \vspace{1.5mm}\\
& = & a\int_{U} |Df| \ dy - a^2|U|.
\end{array}
\end{equation*}
Here we used for the second inequality the definition of the $BV$-norm on $U$ and the fact that $a^2 \geq 1-\sqrt{1-a^2}$ for every $a \in [0,1]$.
\end{proof}

\medskip

For any set $L \subset \real^n$ and $t \in [0,1]$ we define the points of density $t$ for $L$,
\[ L^{(t)}:= \left\{ x \in \real^n: \ \lim_{r \to 0} \frac{|L \cap B_r(x)|}{|B_r|} = t\right\}.\]
The \emph{essential boundary} of $L$ is defined by $\partial_eL = \real^n \setminus (L^{(0)} \cup L^{(1)})$.  For the remainder of the paper, whenever we speak about a set of finite perimeter $L$ we will make the normalization (and abuse of notation) that $L = L^{(1)}$, $L^{C} = L^{(0)}$ and $\partial L =  \real^n \setminus (L^{(0)}\cup L^{(1)}) = \partial_eL$.

\medskip

We will also use the notion of the \emph{trace} of a finite perimeter set.  Let $U$ be a open set in $\real^n$ with Lipschitz boundary $\partial U$ and $f \in BV_{loc}(U)$.  Then there exists a trace $\varphi \in L^1_{loc}(\partial U)$ so that for every $g \in C^1_c(\real^n; \real^n)$,
\[ \int_\Omega f \div g \ dy = -\int_\Omega g \cdot Df \ dy + \int_{\partial \Omega} \varphi g \cdot \nu d\mathcal{H}^{n-1}(y)\]
where $\nu$ is the outer normal to $\Omega$, see \cite[Theorem 2.10]{Giusti}.  For locally finite perimeter sets $L \subset U$ we can take the trace of the indicator function on $\partial \Omega$ which we will refer to as $\varphi_L \in L^1_{loc}(\partial \Omega)$, $\varphi_L$ is itself an indicator function.  We note that if $L$ is not a subset of $U$ then actually it has two traces on $\partial U$, one is the trace of $L \cap U$ and the other is the trace of $L \cap U^C$.  We will sometimes refer to these as $\varphi_L^+$ (for the trace from the outside) and $\varphi_L^-$ (for the trace from the inside).

\medskip

Given the results of the previous paragraph about traces we can make a convenient definition of $\Per(L,\Omega)$ when $A$ is a \emph{closed} set of $\real^n$ with Lipschitz boundary.  In fact we will want to be even more general and taking $U \subset \real^n$ open with Lipschitz boundary and a relatively open subset $\Gamma \subset \partial U$ and define,
\[ \Per(L,U \cup \Gamma) = \Per(L,U) + \int_{\Gamma} \varphi_L(y) d\mathcal{H}^{n-1}(y).\]
We will very often use this notation so the reader should take careful note whether the set $\cdot$ in $\Per(L,\cdot)$ contains some part of its boundary.

\medskip

We will very often use the following standard Lemma, which is essentially just encoding the fact that the surface area of a graph of a function over a flat domain $U \subset \real^n$ is at least the area of the domain.  In particular it is quite close to the idea of Lemma~\ref{lem: surface area to BV}.
\begin{lem}\label{lem: lines through}
Let $U$ be an open bounded domain of $\real^n$ and $L \subset U \times \real$.  Suppose that for some $T>0$, 
\[U \times \{ z =0\}  \subset L^{(1)} \ \hbox{ and } \ U \times \{ z= T\} \subset L^{(0)}.\]
Then,
\[ \Per(L,U \times [0,T]) \geq |U|.\]
\end{lem}

The proof is by using the divergence theorem to integrate $0 = \int_{U \times \{ 0 < z < T\}} \textup{div}(e_{n+1})$ see \cite{MaggiBV}.

\medskip

\subsection{Energy normalization}  To emphasize the importance of the parameter $\cos\theta_Y$ we can make a standard normalization of the energy $E$.  First one can divide through by $\sigma_{\textup{LV}}$, the minima of the rescaled energy is of course unchanged. We could further normalize reducing to just one coefficient $\cos\theta_Y$ by rearranging,
\begin{equation}\notag
\frac{1}{\sigma_{\textup{LV}}}{{E}}(L) = | \partial L \cap \partial V|  + \frac{\sigma_{\textup{SL}}-\sigma_{\textup{SV}}}{\sigma_{\textup{LV}}}|\partial S \cap \partial L|+ \frac{\sigma_{\textup{SV}}}{\sigma_{\textup{LV}}}|\partial S|,
\end{equation}
where the last term is just a constant factor and can be ignored.  Thus we are able to define the normalized energy ${{E}}'$, which has the same constrained local minima as the original energy ${{E}}$,
\begin{equation}\label{eqn: normalized energy}
{{E}}'(L) :=  | \partial L \cap \partial V|  + \cos\theta_Y|\partial S \cap \partial L|.
\end{equation}
This is a convenient reduction, where the emphasis is placed on the single parameter $\cos \theta_Y$.  We will generally avoid making this reduction, except when it has an obvious advantage for computations, since reducing to the single parameter $\cos\theta_Y$ obscures the derivation of the cell problems.

\medskip

\subsection{The interfacial energy}\label{sec: energy}  In this section we give a precise meaning to the interfacial energy $E$ and it's normalized form $E'$ for a given configuration $S,L, V$ in some region of $\real^{d+1}$.  At minimum we will suppose that $S$ is a closed set with Lipschitz boundary, and that $L$ is a set of locally finite perimeter.  The set $V$ will be the complement of $S \cup L$ and therefore will also be a set of locally finite perimeter.  With $S$ fixed the configuration can be parametrized entirely by $L$.  It is most convenient to define the interfacial energy by its density with respect to the Hausdorff $d$-dimensional measure,
\begin{equation}\label{eqn: energy density}
dE(L,x,z) = \sigma_{\textup{LV}} \left. d\mathcal{H}^d\right|_{\partial_eL \cap \partial S} +\sigma_{\textup{SL}} \left. d\mathcal{H}^d\right|_{\partial_eL \cap \partial_e V}+\sigma_{\textup{SV}} \left. d\mathcal{H}^d\right|_{\partial_eV \cap \partial S} .
\end{equation}
Then for any region $\Omega \subset \real^{d+1}$ (a Borel set) we can define,
\begin{equation}\notag 
E(L,\Omega) := \int_{\Omega} dE(L,x,z),
\end{equation}
which is,
\[
E(L,\Omega) = \sigma_{\textup{LV}} \mathcal{H}^d({\partial_eL \cap \partial S} \cap \Omega )+\sigma_{\textup{SL}} \mathcal{H}^d({\partial_eL \cap \partial S} \cap \Omega)+\sigma_{\textup{SV}}  \mathcal{H}^d({\partial_eV \cap \partial S} \cap \Omega).
\]
We will usually find it most convenient to write the energy in terms of perimeter,
\[
E(L,\Omega) = \sigma_{\textup{LV}} \Per(L,\Omega \setminus S) + \sigma_{\textup{SL}} \mathcal{H}^d({\partial_eL \cap \partial S} \cap \Omega)+\sigma_{\textup{SV}}  \mathcal{H}^d({\partial_eV \cap \partial S} \cap \Omega),
\]
especially when we work with the normalized energy $E'$ since we can then remove all reference to $V$ in the definition,
\begin{equation}\label{eqn: norm energy defn}
E'(L,\Omega) := \Per(L,\Omega \setminus S) + \cos\theta_Y \mathcal{H}^d({\partial_eL \cap \partial S} \cap \Omega).
\end{equation}
We will usually be formal and write $\partial L$ instead of $\partial_eL$, referring to the essential boundary only when it is especially important.




\section{The Cell Problems}\label{sec: cell}

The determination of $\overline{\sigma}_{\textnormal{SV}}$ and $\overline{\sigma}_{\textnormal{SL}}$ can be reduced to two ``cell problems" without a volume constraint.  The cell problem for $\overline{\sigma}_{\textnormal{SL}}$ identifies the optimal configuration in the bulk of the wetted set, the homogenized coefficient $\overline{\sigma}_{\textnormal{SL}}$ is the total energy per unit cell of this optimal configuration.  The cell problem for $\overline{\sigma}_{\textnormal{SV}}$ identifies the optimal configuration in non-wetted regions, the homogenized coefficient $\overline{\sigma}_{\textnormal{SV}}$ is the total energy per unit cell of this optimal configuration.  We remark that when $\sigma_{\textup{SV}}$ and $\sigma_{\textup{SL}}$ are fixed constants, i.e. no chemical texturing, only one of these two cell problem solutions will be non-trivial (see Proposition~\ref{prop: contact angle quantity non-degen}).  We will introduce the cell problems and their properties in a general setting, which incorporates non-smooth, non-periodic surfaces $S$. Our arguments are easily adaptable to $x$-dependent $\sigma$'s.

\subsection{General assumptions on the solid surface.}  Let us take $S$ to be a closed set of $\real^d \times \real$ with Lipschitz boundary satisfying that for some $0\leq M < \infty$,
\begin{equation}\label{eqn: S basic assumption}
\real^d \times (-\infty,-M) \subseteq S  \subseteq \real^d \times (-\infty,0] \ \hbox{ and furthermore touches both the $-M$ and $0$ levels.}
\end{equation}
 We interpret $S$ as the region occupied by the solid surface when viewed at the microscopic scale where the inhomogeneities are of unit size occurring over a unit length scale.  We will furthermore assume that $S$ is the region below the graph of a function.  Let $\phi:\mathbb{R}^d\to \real$ be a upper-semicontinuous and bounded variation function with $-M = \inf \phi$ and $0 = \sup\phi$. We then consider $S$ of the form, 
\begin{equation}
S = \{(x,z) \in \real^d \times \real : -\infty < z \leq \phi(x) \},
\end{equation}
 where we will need to further impose that this set at least has locally Lipschitz boundary.  This graph property is not necessary for our result and could be replaced by more general conditions, which we will remark on below. 
 
 \medskip
 
To ensure the existence of minimizers for the rough surface problem we will need to assume that $\partial S$ is sufficiently smooth, in particular we will take
 \begin{equation}\label{eqn: smooth assumption}
  \phi \ \hbox{ is $C^{1,1}$.} 
  \end{equation}
  This assumption is used only to prove the lower-semicontinuity of the energy functional (see \cite[Propositions 19.1, 19.3]{MaggiBV}), though such a strong assumption is likely to not be necessary.  Assumption~\eqref{eqn: smooth assumption} will always be in force in Sections~\ref{sec: vol}-\ref{sec: hom}, in Sections~\ref{sec: cell}-\ref{sec: values} we will specify exactly when it is necessary.  In particular we will be interested in some examples of surfaces $S$ where the boundary is only Lipschitz, for example 
 \[ \phi(x) = M(1- \indicator_{B_{1/2}}(x)) \ \hbox{ and extend by $\integer^d$-periodicity to } \ \real^d,\]
 which models surfaces which are used in physical experiments \cite{bico1999pearl}.  We will consider such examples in Section~\ref{sec: values}, in those cases we will be able to show, in certain circumstances, the existence of energy minimizers by explicit construction. 

\medskip

  Lastly we put an assumption which guarantees that the homogenized contact angle is non-degenerate.  We assume that there is a continuous increasing modulus $\omega : [0,M] \to \real_+$ with $\omega(0) = 0$ so that for every $0 \leq \delta \leq M$ and every open $U \subseteq \real^d$,
\begin{equation}\label{eqn: condition at the max}
\Per (S,U \times (-\delta,\infty))\leq (1+\omega(\delta))|\{ \phi > -\delta\} \cap U|.
\end{equation}
This condition is guaranteed for example either when $|\{ \phi = 0 \}| >0$ or when $\phi$ is $C^1$ at its maximum and periodic.  In particular it follows from the assumption \eqref{eqn: smooth assumption} that $\phi$ is smooth, we state it separately since it plays a separate role.  

\medskip

Let us also mention at this stage an important quantity associated with the solid surface which is the {\it roughness}.  For every open set $U \subset\real^d$ we define,
\begin{equation}
\rho(U,S) := \frac{1}{|U|} \int_{U} \sqrt{1+|D\phi|^2} \ dx.
\end{equation}
If the limit exists and does not depend on $U$, which it will when $\phi$ is periodic (or random stationary ergodic), we can also define,
\begin{equation}
\overline{\rho}(S) := \lim_{t \to \infty} \rho(tU,S).
\end{equation}

\medskip

{ \bf $\circ$ Surfaces which are not sub-graphs.}   The assumptions above are tailored to the case of surfaces which are sub-graphs.  Here we explain some natural assumptions allowing for more general non-subgraph surfaces like some which are found in nature \cite{Barthlott1997, NEINHUIS1997}.  A natural set of assumptions generalizing the graph setting would be to take $S$ to be simply connected with smooth boundary and satisfying for some $M>0$,
\[ \{ z \leq -M\} \subset S \subset \{z \leq 0\}.\]

\subsection{The cell problems}  To motivate the form of the cell problems let us begin with a purely formal description of the problem described in the introduction, the minimization of the energy functional
\begin{equation}
{{E}}_\e(L,\Omega) = \sigma_{\textup{LV}} \mathcal{H}^d(\partial L \cap \partial V \cap \Omega) + \sigma_{\textup{SL}}\mathcal{H}^d(\partial S_\e \cap \partial L \cap \Omega)+\sigma_{\textup{SV}} \mathcal{H}^d(\partial S_\e \cap \partial V \cap \Omega),
\end{equation}
over configurations $(L,V,S_\e)$ in some finite region $\Omega$ with $|L| = \textup{Vol}$.  Let $L$ be the global minimizer associated with this problem.

\medskip

Consider computing the energy ${{E}}_\e(L)$ by discretizing at a certain scale $R\e$ with $1 \ll R \ll \e^{-1}$, dividing $\real^d$ up into disjoint squares $\Box_{R\e}$ of $\real^d$ centered at the lattice points of $R\e \integer^d$.    We divide up the squares $\Box_{R\e}$ depending on whether they are underneath the bulk of the liquid region, the bulk of the vapor region, or neither.  Squares underneath the liquid region we called {\it wetted},  underneath the vapor region we called {\it non-wetted} and otherwise they are {\it on the contact line}.  Slightly more precisely we mean to say that $\Box_{R\e}$ is wetted if $\Box_{R\e} + h\e e_{d+1}$ is contained in the liquid region $L$ for some $1 \ll h \ll \e^{-1}$, a similar definition applies to non-wetted squares.

\medskip

Assuming that the volume constrained minimizer $L$ has sufficient large scale regularity, i.e. that the contact line region has dimension lower than $d$ at the scale $R\e$, then the energy contribution of the contact line squares will be negligible.  Now in the wetted region, up to a small error coming from the boundary with the contact line, we can essentially think that the configuration will minimize the energy ${{E}}_\e$ under the constraint that $\Box_{R\e}+h \e e_{d+1} \subseteq L$ for every wetted square.  This motivates us to define the following cell problems which find the minimal energy effective solid-liquid (or solid-vapor) interface over a given wetted (or non-wetted) set.

\medskip

 \begin{figure}[t]
 \centering
 \def\svgwidth{3in}
  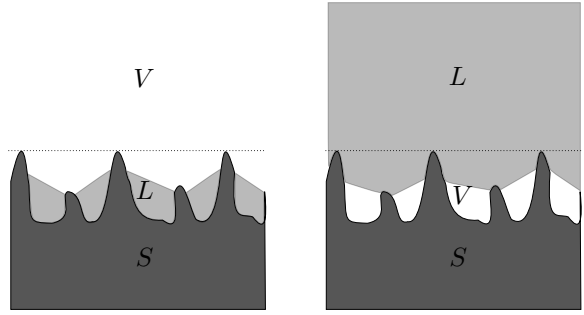
  \caption{On the left is solid-vapor cell problem solution for a hydrophilic solid, on the right is solid-liquid cell problem solution for a hydrophobic solid.}
 \end{figure}
 
 First of all we will zoom in on the microscopic structure of the minimizer inside the wetted and non-wetted sets.  Rescaling by $\frac{1}{\e}$ we will actually look at the energy $E = E_1$ with the solid surface $S = S_1$.
 
 \medskip

Let $U$ be an open region of $\real^d$.  We compute the minimal energy of a configuration which, macroscopically, looks like a solid-liquid interface above the region $U$:
\begin{equation}\label{eqn: SL quantity}
 \Sigma_{\textup{SL}}(U) := \inf_{L \in \mathcal{A}_{\textup{SL}}(U)}{{E}}(L,U \times \real)  \ \hbox{ with }  \mathcal{A}_{\textup{SL}}(U):= \bigg\{ L \in BV_{loc}(\real^{d+1} \setminus {S}) :     \{ z \geq 0\} \subseteq L \bigg\}.
\end{equation}
and the minimal energy of a configuration which, macroscopically, looks like a solid-vapor interface above the region $U$:
\begin{equation}\label{eqn: SV quantity}
 \Sigma_{\textup{SV}}(U) := \inf_{L \in \mathcal{A}_{\textup{SV}}(U)}{{E}}(L,U \times \real)  \ \hbox{ with }  \mathcal{A}_{\textup{SV}}(U):= \bigg\{ L \in BV_{loc}(\real^{d+1} \setminus {S}) :     L \subseteq \{z \leq 0\} \bigg\}.
\end{equation}
We make note that there is no energy term associated with surface area on $\partial U \times \real$.  It is also notationally convenient to define the macroscopic contact angle associated with the minimal energy configurations on $U$,
\begin{equation}\label{eqn: contact angle quantity}
 \cos \Theta_Y(U) = \frac{\Sigma_{\textup{SL}}(U)-\Sigma_{\textup{SV}}(U)}{\sigma_{\textup{LV}}|U|}
 \end{equation}
We denote $\Theta_Y = \pi$ if the right hand side above is $<-1$ and take $\Theta_Y = 0$ if the right hand side above is $>1$.

\medskip

The homogenized coefficients $\overline{\sigma}_{\textup{SL}}$ and $\overline{\sigma}_{\textup{SV}}$ can then be defined as the minimal energy per unit area at large scales.  We define,
\begin{equation}\label{eqn: hom energies gen}
\overline{\sigma}_{\textup{SL}} := \lim_{t \to \infty} \frac{1}{t^d}\Sigma_{\textup{SL}}(tU) \ \hbox{ and } \ \overline{\sigma}_{\textup{SV}} := \lim_{t \to \infty} \frac{1}{t^d}\Sigma_{\textup{SV}}(tU)
\end{equation}
these limits do not necessarily exist and even if they do they are not necessarily independent of $U$.  In the rest of this section we will justify that, in the periodic setting, these limits do in fact exist and satisfy
(Proposition~\ref{prop: contact angle quantity non-degen})
\begin{equation}\label{eqn: hom contact angle}
 -1 < \cos\overline{\theta}_Y:= \frac{\overline{\sigma}_{\textup{SL}} - \overline{\sigma}_{\textup{SV}}}{{\sigma}_{\textup{LV}}} <1.
 \end{equation}

\medskip

In order to better understand these definitions, and also because it will be useful later, we make a note of several equivalent minimization problems.  
\begin{lem}\label{reduction}
For every $t\geq 0$,
\begin{equation*}
\begin{array}{lll}
\Sigma_{\textup{SL}}(U) &=& \inf\bigg\{ {{E}}(L,U \times \real) : \ L \in BV_{loc}(\real^{d+1} \setminus {S}) \ \hbox{ and } \  \{ z \geq t \} \subseteq L \bigg\}  \vspace{1.5mm}\\
&=& \inf\bigg\{ {{E}}(L,U \times \real) : \ L \in BV_{loc}(\real^{d+1} \setminus {S}) \ \hbox{ and } \  \{ z = t \} \subseteq L \bigg\}  
\end{array}
\end{equation*}
The analogous equivalences hold for $\Sigma_{\textup{SV}}(U)$.
\end{lem}
The proof is a straightforward application of Lemma~\ref{lem: lines through}.

\medskip

{\bf $\circ$ Cell problem minimizers.}  It turns out that for any given constant values of $\sigma_{\textnormal{SL}}$ and $\sigma_{\textnormal{SV}}$ only one of the two cell problems \eqref{eqn: SL quantity} and \eqref{eqn: SV quantity} has a non-trivial minimizer.
\begin{lem}\label{lem: trivial minimizers}
Let $U$ be an open set with piecewise smooth boundary.  In the hydrophilic case when $-1  \leq \cos \theta_Y \leq 0$,
\[  \frac{1}{|U|}\Sigma_{\textnormal{SL}}(U) = {\sigma}_{\textnormal{SL}}\rho(U) = {\sigma}_{\textnormal{SL}}\frac{1}{|U|}\int_{U} \sqrt{1+|D\phi|^2} \ dx \ \hbox{ and } \ L_{\textnormal{SL}}(U) =\real^{d+1} \setminus S\]
and symmetrically in the hydrophobic case $0 \leq \cos\theta_Y \leq 1$,
\[ \frac{1}{|U|}\Sigma_{\textnormal{SV}}(U) = {\sigma}_{\textnormal{SV}}\rho(U) = {\sigma}_{\textnormal{SV}}\frac{1}{|U|}\int_{U} \sqrt{1+|D\phi|^2} \ dx \ \hbox{ and } \ L_{\textnormal{SV}}(U) = \emptyset.  \]
\end{lem}
\begin{rem}
This result would still be true for chemically textured rough surfaces, i.e. when ${\sigma}_{\textnormal{SL}}$ and ${\sigma}_{\textnormal{SV}}$ depending on $x$, as long as ${\sigma}_{\textnormal{SL}}(x) - {\sigma}_{\textnormal{SV}}(x)$ has a fixed sign.  If the sign of ${\sigma}_{\textnormal{SL}}(x) - {\sigma}_{\textnormal{SV}}(x)$ is allowed to change, then it is quite possible that \emph{both} the solid-liquid and the solid-vapor cell problems will have non-trivial solutions at the same time.  In order that our proofs can generalize to the setting of $x$-dependent surface energies as easily as possible we will avoid making use of Lemma~\ref{lem: trivial minimizers} what follows.
\end{rem}
Based on Lemma~\ref{lem: trivial minimizers} we can also make an alternative formulation of the cell problems in terms of the normalized energy ${{E}}'$ introduced in \eqref{eqn: normalized energy} and \eqref{eqn: norm energy defn}.  Recall, for an open set $U \subset \real^d$ with piecewise smooth boundary we have defined,
\[
{{E}}'(L,U \times \real) :=  \Per(L, U \times \real \setminus S)  + \cos\theta_Y \mathcal{H}^d( \partial L \cap \partial S \cap (U \times \real)).
\]
  Let us consider the hydrophobic case, $\cos\theta_Y >0$.  The hydrophilic case and be understood by symmetry.  Then, using Lemma~\ref{lem: trivial minimizers}, for any $U \subset \real^d$ open with piecewise smooth boundary,
\begin{equation}\label{eqn: normalized energy cell}
\cos \Theta_Y(U) = \inf_{L \in \mathcal{A}_{\textup{SL}}(U)} \frac{1}{|U|}{{E}}'(L,U \times \real ) \ \hbox{ and } \ \cos \overline{\theta}_Y = \lim_{t \to \infty}  \cos \Theta_Y(tU),
\end{equation}
if said limit exists independent of $U$.  

\begin{proof}[Proof of Lemma~\ref{lem: trivial minimizers}]
We will just consider a hydrophobic surface where $\sigma_{\textup{SL}} \geq \sigma_{\textup{SV}}$, the other case follows by symmetry. Let $L \in \mathcal{A}_{\textup{SL}}(U)$ be a finite perimeter, positive measure subset of $(U \times \real) \setminus S$ which lies below the $\{z = 0\}$ level.  It suffices to show that 
\[ {{E}}(L,U \times \real) > \sigma_{\textnormal{SV}} |U|\rho (U) = {{E}}(\emptyset,U \times \real)\]
 since this will make the empty set the unique minimizer in $\mathcal{A}_{\textup{SL}}(U)$.  {Since $|L|>0$ and $S$ is the region above a graph $\Per(L,(U \times \real) \setminus S)>0$} so
\[
 {{E}}(L, U \times \real) >   (\sigma_{\textnormal{SL}}-\sigma_{\textnormal{SV}})\mathcal{H}^d( \partial L \cap \partial S \cap U \times \real) + \sigma_{\textnormal{SV}} \mathcal{H}^d( \partial S \cap U \times \real) \geq   \sigma_{\textnormal{SV}} \mathcal{H}^d( \partial S \cap U \times \real),
\]
where we have used the hydrophobicity for the last inequality. 
\end{proof}

By the standard direct method of calculus of variations we can show, under assumption~\eqref{eqn: smooth assumption}, that minimizers exist for \eqref{eqn: SL quantity} and \eqref{eqn: SV quantity}. Despite the possible non-uniqueness of minimizers for the variational problems \eqref{eqn: SL quantity} and \eqref{eqn: SV quantity}, it turns out that the \emph{maximal} (and \emph{minimal}) minimizers are unique.  We say that $L$ satisfying
\begin{equation}\label{eqn: def maximal minimizer}
 {{E}}(L,U \times \real) = \min_{\Lambda \in \mathcal{A}} {{E}}(\Lambda,U \times \real) \ \hbox{is \emph{maximal} if for any $K \in \mathcal{A}$ with } \ {{E}}(K,U \times \real) = \min_{\Lambda \in \mathcal{A}} {{E}}(\Lambda,U \times \real), \ K \subseteq L.
 \end{equation}
The uniqueness of these minimizers is useful to us mainly to prove the existence of a $\integer^d$-periodic cell problem solution (see Lemma~\ref{maximal} below), which is useful (but not absolutely necessary) for the proof of homogenization in Section~\ref{sec: hom}.
\begin{lem}\label{lem: maximal minimal minimizers}
Suppose that $\phi$ is smooth.  For $U \subset \real^d$ open there exists unique maximal (resp. minimal) minimizers $L_{\textup{SL}}(U)$ and $L_{\textup{SV}}(U)$ so that,
\[ {{E}}(L_{\textup{SL}},U \times \real ) = \Sigma_{\textup{SL}}(U) \ \hbox{ and } \ {{E}}(L_{\textup{SV}},U\times \real ) = \Sigma_{\textup{SV}}(U).\]
\end{lem}
The proof of this result follows by the direct method of calculus of variations combined with the following union-intersection inequality for the uniqueness:
\begin{lem}\label{lem: union intersection}
For every $L,L' \in BV(\real^{d+1} \setminus S)$ and every open set $\Omega \subset \real^{d+1}$ which is bounded in the $x$ variable, the inequality holds,
\[ {{E}}(L \cup L',\Omega) + {{E}}(L \cap L',\Omega) \leq {{E}}(L,\Omega) + {{E}}(L',\Omega).\]
\end{lem}
\begin{proof}
Call $V$ and $V' $ to be the vapor regions corresponding to $L$ and $L'$ respectively. From the book \cite{MaggiBV} (see Lemma 12.22) we have,
\begin{equation}\label{eqn: per union ineq}
\Per(L \cup L',\Omega)+\Per(L \cap L',\Omega) \leq \Per(L ,\Omega)+\Per( L',\Omega).
\end{equation}
On the other hand, calling $\varphi_L$, $\varphi_{L'}$, $\varphi_V$ and $\varphi_{V'}$ to be the respective traces of $\indicator_L$, $\indicator_{L'}$, $\indicator_{V}$ and $\indicator_{V'}$ on $\partial S \cap \Omega$,
\begin{equation}\label{eqn: trace union L eq}
 \int_{\partial S \cap \Omega} \varphi_{L} \vee \varphi_{L'} \ d\mathcal{H}^d = \int_{\partial S \cap \Omega} \varphi_{L} + \varphi_{L'}  - \varphi_{L} \wedge \varphi_{L'} \ d\mathcal{H}^d,
\end{equation}
and similarly for the traces of $V,V'$,
\begin{equation}\label{eqn: trace union V eq}
 \int_{\partial S \cap \Omega} \varphi_{V} \vee \varphi_{V'} \ d\mathcal{H}^d = \int_{\partial S \cap \Omega} \varphi_{V} + \varphi_{V'}  - \varphi_{V} \wedge \varphi_{V'} \ d\mathcal{H}^d.
\end{equation}
Multiplying \eqref{eqn: trace union L eq} by $\sigma_{\textnormal{SL}}$, multiplying \eqref{eqn: trace union V eq} by $\sigma_{\textnormal{SV}}$ and summing both with \eqref{eqn: per union ineq} yields the desired inequality for ${{E}}(\cdot,\Omega)$. 
\end{proof}

{\bf $\circ$ Additivity properties of $\Sigma_{\textup{SV}}$ and $\Sigma_{\textup{SL}}$.} Before we consider the more specialized case of periodicity we make two more observations about the properties of $\Sigma_{\textup{SL}}$ and $\Sigma_{\textup{SV}}$.  We will show that $\Sigma_{\textup{SL}}$ and $\Sigma_{\textup{SV}}$ are almost additive, this is a very useful property for us in periodic media, and perhaps it is even more useful in random media.  First we claim that $\Sigma_{\textup{SL}}$ and $\Sigma_{\textup{SV}}$ are super-additive quantities:
\begin{lem}[Super-additivity of $\Sigma_{\textup{SL}}$/$\Sigma_{\textup{SV}}$] \label{lem: super-additivity}
Let $U,V$ open subsets of $\real^d$ and let $W$ be an open set with $U \cup V \subset W$, then
\[  \Sigma_{\textup{SL}}(W) \geq  \Sigma_{\textup{SL}}(U)+ \Sigma_{\textup{SL}}(V).\]
The same result holds for $\Sigma_{\textup{SV}}$.
\end{lem}
\begin{proof} Let $L$ be any set which is admissible for the minimization \eqref{eqn: SL quantity} associated with $W$, i.e. $L$ is a finite perimeter subset of $\real^d \setminus S$ with $\{ z \geq 0\} \subset L$, then  
\[ {{E}}(L, W \times \real) \geq    {{E}}(L, U \times \real) +  {{E}}(L,V \times \real) \geq \Sigma_{\textup{SL}}(U)+\Sigma_{\textup{SL}}(V)\]
since $L$ is admissible for the $\Sigma_{\textup{SL}}(U)$ and $\Sigma_{\textup{SL}}(V)$ minimizations as well. Taking the infimum over $L$ yields the result.
\end{proof}
Not only are $\Sigma_{\textup{SL}}$ and $\Sigma_{\textup{SV}}$ super-additive, but they are close to being sub-additive as well. 
\begin{lem}[Almost additivity of $\Sigma_{\textup{SL}}$/$\Sigma_{\textup{SV}}$] \label{lem: almost additivity}
Suppose that $U$ is the interior of a union of closed squares $ \overline{\Box}_i$ with disjoint interiors over $i$ in some index set $I$, then
$$ \sum_{i\in I} \Sigma_{\textup{SL}}(\Box_i) \leq \Sigma_{\textup{SL}}(U) \leq \sum_{i\in I} \Sigma_{\textup{SL}}(\Box_i) + \sigma_{\textup{LV}}2dM \sum_{i \in I} |\Box_i|^{\frac{d-1}{d}}.$$
\end{lem}

\begin{proof}
We only need to check the upper bound due to the previous lemma. To this end we take the minimizers $L_i$ for $\Box_i$ and take the union $L:= \cup_{i\in I} L_i$. Then we have
\[
\Sigma_{\textup{SL}}(U)\leq {{E}}(L, U\times \real) \leq \sum_{i\in I} \Sigma_{\textup{SL}}(\Box_i) + \sigma_{\textup{LV}}\sum_{i\in I}\mathcal{H}^d(\partial L  \cap \dside \Box_i^{-M,0}).
\]
Note that we are using that $S$ is the region under the graph of a smooth function so that 
\[\mathcal{H}^d(\partial S \cap \partial L  \cap \dside \Box_i^{-M,0}) \leq \mathcal{H}^d(\partial S \cap \dside \Box_i^{-M,0}) = 0.\]
  We conclude since 
\[\mathcal{H}^d(\partial L  \cap \dside \Box_i^{-M,0}) \leq \mathcal{H}^d( \dside \Box_i^{-M,0}) = 2dM|\Box_i|^{\frac{d-1}{d}}.\]
We will see this sort of estimate appear many times in the rest of the paper.
\end{proof}

\subsection{Cell problems: periodic media}

In order to justify the existence of the limits in \eqref{eqn: hom energies gen} and their independence on the set $U$ we need to put some kind of stationarity and ergodicity assumptions on the rough surface $S$.  The first such assumption that we will consider is periodicity.  We suppose that the function $\phi$, and therefore also the set $S$, is invariant under $\integer^d$-lattice translations,
\begin{equation}\label{eqn: periodicity assumption}
\phi(\cdot + k ) = \phi(\cdot) \ \hbox{ and hence } \ S + k = S \ \hbox{ for all } k \in \integer^d.
\end{equation}
In the periodic setting we can justify the existence of the limits in \eqref{eqn: hom energies gen} and prove that the homogenized contact angle is non-degenerate \eqref{eqn: hom contact angle}.  Actually we will be able to do even better and show the existence of $\integer^d$-periodic $L_{\textup{SL}}$ and $L_{\textup{SV}}$ which have energy exactly $\overline{\sigma}_{\textup{SL}}$ and $\overline{\sigma}_{\textup{SV}}$ respectively over a unit periodicity cell.

\medskip

First we will show that for periodic media the homogenized coefficients $\overline{\sigma}_{\textup{SL}}$ and $\overline{\sigma}_{\textup{SV}}$ exist in the sense of \eqref{eqn: hom energies gen} with an explicit rate of convergence.

\medskip

\begin{thm}\label{thm: cell gen periodic}
There are $\overline{\sigma}_{\textup{SL}}$ and $\overline{\sigma}_{\textup{SV}}$ so that for any open set $U \subset \real^d$ with smooth boundary, or a square $U = \Box$, there is an depending on the smoothness property of $\partial U$, invariant under translations of $U$, so that for all $r \geq r_0$,
\[ \left| \tfrac{1}{r^d| U|} \Sigma_{\textup{SL}}(r U) - \overline{\sigma}_{\textup{SL}}\right| \leq C(d)\sigma_{\textup{LV}}M\frac{|\partial U|}{|U|}\frac{\lceil\log \frac{r}{r_0}\rceil}{r} \ \hbox{ and } \ \left| \tfrac{1}{r^d| U|} \Sigma_{\textup{SV}}(r U) - \overline{\sigma}_{\textup{SV}}\right| \leq C(d)\sigma_{\textup{LV}}M\frac{|\partial U|}{|U|}\frac{\lceil\log \frac{r}{r_0}\rceil}{r}. \]
For squares $U = \Box$ the logarithmic term can be removed and $r_0(\Box) = |\Box|^{-1/d}$.
\end{thm}

The essential important property of the homogenized coefficients $\overline{\sigma}_{\textup{SL}}$ and $\overline{\sigma}_{\textup{SV}}$ is that they define a non-degenerate contact angle by \eqref{eqn: hom contact angle}. For our regularity theory we will need even better, that the same non-degeneracy property happens at finite scales in a quantifiable way.  This is the content of our next Lemma.

\begin{prop}\label{prop: contact angle quantity non-degen}
Let $U = \Box_r$ for $r\geq 1$ or $U = D_r$ for $r \geq \frac{\sqrt{d}}{2}$.  There exists $c_0$ depending on $1-|\cos\theta_Y|$ and the modulus $\omega$ from \eqref{eqn: condition at the max} so that,
\[ -1 + c_0 \leq \cos \Theta_Y(U) \leq 1 - c_0.
\]
In particular the same holds for $\bar{\theta}_Y$ due to Theorem~\ref{thm: cell gen periodic}.
 \end{prop}

\begin{rem}\label{rem: cond at max}
Proposition~\ref{prop: contact angle quantity non-degen} and/or its proof will be used several times throughout the paper, in particular in Section~\ref{sec: reg}.  These are the only (but very important) places where the assumption~\eqref{eqn: condition at the max} comes into play.  These are also the only places where having an $x$-dependent surface energy requires a non-trivial change in proof.

\medskip

In the proof below we will make use of the more general condition \eqref{eqn: condition at the max} to demonstrate how it is used, in the rest of the paper when the result is re-proven in some form we will use the condition that $|\{\phi = 0\} \cap \Box_1| >0$, which is simpler to work with.  Furthermore we give a slightly more complicated proof, which avoids using Lemma~\ref{lem: trivial minimizers}, so that the proof will generalize easily to the case of $x$-dependent surface energies.
\end{rem}

We now proceed with the proofs.  First of Theorem~\ref{thm: cell gen periodic}. 
\begin{proof}[Proof of Theorem~\ref{thm: cell gen periodic}]
We first prove the result for $\Box_1$ and dyadic $r$.  Then we prove the result for a general $U$ by a dyadic decomposition.  We just do the argument for the solid-liquid cell problem, then the solid-vapor case follows by symmetry.

\medskip

1. By the $\integer^d$-periodicity of $S$, for any $k \in \integer^d$ and any $\Box$,
\[ \Sigma_{\textup{SL}}(\Box+k) = \Sigma_{\textup{SL}}(\Box).\]
Let $ \ell >r \geq 1$ be dyadics, we view $\ell \Box_1$ as a union of $(\ell/r)^d$ (almost) disjoint $\integer^d$-translates of $ r\overline{\Box}_1$,
\[ \ell \Box_1 = \bigcup_{|k|_\infty \leq \ell/r} r(\overline{\Box}_1 + k).\]
Then by the almost additivity of $\Sigma_{\textup{SL}}$ Lemma~\ref{lem: almost additivity},
\[ |\Sigma_{\textup{SL}}(\ell\Box_1) - (\ell/r)^d\Sigma_{\textup{SL}}(r \Box_1)| \leq \sigma_{\textup{LV}}2dM(\ell/r)^dr^{(d-1)}.\]
Then dividing through by $\ell^d$,
\begin{equation}\label{eqn: dyadic conv}
 |\tfrac{1}{\ell^d}\Sigma_{\textup{SL}}(l\Box_1) - \tfrac{1}{r^d}\Sigma_{\textup{SL}}(r\Box_1)| \leq \frac{\sigma_{\textup{LV}}2dM}{ r}.
 \end{equation}
This establishes that $\tfrac{1}{r^d}\Sigma_{\textup{SL}}(r\Box_1)$ is a Cauchy sequence over dyadic $r \geq 1$,  call $\overline{\sigma}_{\textup{SL}}$ to be the limit.  Then sending $\ell \to \infty$ along dyadics in the previous inequality yields the rate of convergence.  

\medskip

2. Now for an arbitrary bounded $U$ with smooth boundary, we do a standard Whitney type decomposition writing $U$ as a union, over some countable index set $I$, of closed dyadic squares $\overline{\Box}_i$ with disjoint interiors.  These squares satisfy that,
\begin{equation}\label{eqn: whitney prop}
 \sqrt{d} |\Box_i|^{1/d}\leq \textup{dist}(\Box_i, U^C) \leq 4\sqrt{d} |\Box_i|^{1/d}  .
 \end{equation}
For a $\delta$ dyadic with $|U|^{1/d}\geq \delta >0$ call $I_\delta$ the subset of the $\Box_i$ that have side length $|\Box_i|^{1/d} = \delta$.  Let us count how many boxes of side length $\delta$ there can be.  

\medskip

Since $\partial U$ is assumed to be smooth,
\begin{equation}\label{eqn: t0 def}
\hbox{there exists $t_0>0$ so that for every $x \in \partial U$ it holds,} \ \frac{1}{4} \leq \frac{|D_{t_0}(x) \cap U| }{|D_{t_0}|}\leq \frac{3}{4}.
 \end{equation}
For $ t >0 $ call 
\[ U_{t}:= \{ x \in \real^d: \ d(x, U^C) >t\}.\] 
Since $U$ has smooth boundary, in particular it has finite perimeter and so, using the above density estimate, we have the estimate (see \cite{Kraft2015} Theorem 4 for a proof),
\begin{equation}\label{eqn: d func est}
 |U \setminus U_t| \leq C(d) |\partial U|t \ \hbox{ for all } \ t \in (0,t_0).
 \end{equation}
For the $\Box_i$ with $i \in I_\delta$ we have,
\[ \textup{dist}(\Box_i, U^C) \leq 4\sqrt{d} \delta  \ \hbox{ and so } \ \Box_i \subseteq U \setminus U_{5\sqrt{d}\delta}.\]
Thus for every $\delta \leq \frac{1}{5\sqrt{d}}t_0$
\[ \#(I_\delta) \delta^d = \sum_{i \in I_\delta} |\Box_i| \leq |U \setminus U_{5\sqrt{d}\delta}| \leq C(d)\mathcal{H}^{d-1}(\partial U)\delta.\]
Thus we have the bounds,
\begin{equation}\label{eqn: Idelta upper bounds}
\left\{
\begin{array}{lll}
\#(I_\delta) \leq C|\partial U|\delta^{1-d} & \hbox{for}  &\delta \leq \frac{1}{5\sqrt{d}}t_0 \vspace{1.5mm}\\
\#(I_\delta) \leq \lfloor \delta^{-d}|U| \rfloor & \hbox{for}  & \delta \geq \frac{1}{5\sqrt{d}}t_0  .
\end{array}\right.
\end{equation}
\medskip

Now we define $r_0 := 5\sqrt{d}/t_0$ so that for every $r \geq r_0$,
\begin{equation*}
\begin{array}{lll}
 \left|\Sigma_{\textup{SL}}(rU) -\displaystyle\sum_{r|\Box_i|^{1/d} \geq 1} \Sigma_{\textup{SL}}(r\Box_i)\right| &\leq& \Sigma_{\textup{SL}}\left(rU \setminus \displaystyle\bigcup_{r|\Box_i|^{1/d} \geq 1} r\overline\Box_i\right)+ CM\sigma_{\textup{LV}}\displaystyle\sum_{r|\Box_i|^{1/d} \geq 1} r^{d-1}|\Box_i|^{\frac{d-1}{d}} \vspace{1.5mm}\\
 & \leq & \sigma_{\textup{LV}}\left|r\left(U \setminus \displaystyle\bigcup_{r|\Box_i|^{1/d} \geq 1} \overline\Box_i\right)\right| + CM\sigma_{\textup{LV}}\displaystyle\sum_{\substack{\delta \textup{ dyadic,}\\ \delta r \geq 1}} \sum_{i \in I_{\delta}}r^{d-1}|\Box_i|^{\frac{d-1}{d}}  \vspace{1.5mm}\\
 &\leq& C\sigma_{\textup{LV}}|\partial U| r^{d-1}+ CM\sigma_{\textup{LV}}\displaystyle\sum_{\substack{\delta \textup{ dyadic,}\\ \delta r \geq 1}}\#(I_\delta) r^{d-1}\delta^{d-1}  \vspace{1.5mm}\\
 &\leq& C\sigma_{\textup{LV}}|\partial U|(1 + M \log_2 \frac{r}{r_0} + t_0^{-1})r^{d-1}
 \end{array}
 \end{equation*}
  where we have used for the last inequality the upper bounds of \eqref{eqn: Idelta upper bounds}. Then using the convergence rate for the dyadic squares \eqref{eqn: dyadic conv},
\begin{equation*}
\begin{array}{lll}
 \left|\displaystyle\sum_{r|\Box_i|^{1/d} \geq 1} \Sigma_{\textup{SL}}(r\Box_i) - \overline{\sigma}_{\textup{SL}}|rU|\right| &\leq& C\sigma_{\textup{LV}}|\partial U| r^{d-1} + CM\sigma_{\textup{LV}}\sum_{r|\Box_i|^{1/d} \geq 1}\frac{1}{r|\Box_i|^{1/d}}r^d|\Box_i| 
 \end{array}
 \end{equation*}
 which can be estimated in exactly the same way as before.  Thus, putting the previous two estimates together, we get the desired estimate on $|\Sigma_{\textup{SL}}(rU) - \overline{\sigma}_{\textup{SL}}|rU||$.
\end{proof}

\begin{proof}[Proof of Proposition~\ref{prop: contact angle quantity non-degen}]

The following proof can be slightly simplified by using the assumption of constant coefficients via Lemma~\ref{lem: trivial minimizers}.  We avoid using Lemma~\ref{lem: trivial minimizers} so that our proof could be generalized to the case of $x$-dependent surface energies.

\medskip

By symmetry, we can just prove the upper bound. We will proceed under the assumption that $|\{ \phi = 0\} \cap \Box_1| = 0$.  The argument is even easier when $|\{ \phi = 0\} \cap \Box_1|  >0$.

\medskip

The proof is essentially to construct a Cassie-Baxter-like state for comparison.

\medskip

First we fix $0 < \delta \leq M$ so that,
\[ \cos\theta_Y(1+\omega(\delta)) = \min \left\{ \frac{1}{2}(1+\cos \theta_Y) ,\cos\theta_Y(1+\omega(M))\right\} <1,\]
where the modulus $\omega$ is from assumption~\eqref{eqn: condition at the max}.  The reason for this choice will be made clear in the proof.

\medskip

1. The requirement $r \geq 1$ for $U = \Box_r$ (or $r \geq \frac{\sqrt{d}}{2}$ for $U = D_r$) is just to guarantee that a constant fraction of the area of $U$ is taken up by $\{ \phi \geq -\delta\}$.  We first consider the case of $\Box_r$.  Let  integer $n\geq 0$ such that $2^n \leq r < 2^{n+1}$,
\[ |\Box_r \cap \{ \phi \geq -\delta\}| \geq |\Box_{2^n} \cap \{ \phi \geq -\delta\}|= 2^{nd}|\Box_1 \cap\{ \phi \geq -\delta\}|\geq 2^{-d}|\Box_1 \cap\{ \phi \geq -\delta\}||\Box_r|. \]
Note that by assumption~\eqref{eqn: S basic assumption} $|\Box_1 \cap\{ \phi \geq -\delta\}|>0$. 

\medskip

Now we do the case of $D_r$.  Let integer $n \geq 0$ such that $\frac{\sqrt{d}}{2}2^n \leq r \leq \frac{\sqrt{d}}{2}2^{n+1}$,
\[ |D_r \cap \{ \phi \geq -\delta\}| \geq |\Box_{2^n} \cap \{ \phi \geq -\delta\}|= 2^{nd}|\Box_1 \cap\{ \phi \geq -\delta\}|\geq d^{-d/2}|\Box_1 \cap\{ \phi \geq -\delta\}||D_r|. \]
This estimate will be sufficient for later.

\medskip

2. The following part is not necessary for the constant coefficient problem we are studying, because we know that $L_{\textup{SV}} = \emptyset$.  Now we to show that for $\delta$ chosen as above,
\[ \Sigma_{\textup{SV}}(U) = \inf \bigg\{ {{E}}(L,U \times \real) : L \in \mathcal{A}_{\textup{SV}}(U) \ \hbox{ and } \ L \subseteq \{ z \leq -\delta\} \bigg\}.\]
For an arbitrary $L \in \mathcal{A}_{\textup{SV}}(U)$ consider the replacement $L' = L \cap \{ z \leq - \delta\}$, the change in energy is given by,
\[  {{E}}(L',U \times \real) - {{E}}(L,U \times \real) \leq \sigma_{\textup{LV}}[1-\cos\theta_Y(1+\omega(\delta))]|U| <0. \]

\medskip

3. Let $\Lambda_{\textup{SV}} \in \mathcal{A}_{\textup{SV}}(U)$ arbitrary with $\Lambda_{\textup{SV}} \subseteq \{ z < -\delta\}$,  call $\Lambda_\delta = \{ z \geq (-\delta) \vee  \phi(x)\}$.  We will take $\Lambda_{\textup{SL}}:= \Lambda_{\textup{SV}} \cup \Lambda_\delta$ as a test set for the solid-liquid problem. Since $\Lambda_{\textup{SV}} \cap \{ z < -\delta\} = \emptyset$ and $\Lambda_\delta \cap \{ z \geq \delta\} = \emptyset$ we can compute,
\begin{equation*}
\begin{array}{lll}
{{E}}(\Lambda_{\textup{SL}},U \times \real) & \leq& {{E}}(\Lambda_{\textup{SV}},U \times (-\infty,-\delta]) + {{E}}(\Lambda_\delta,U \times [-\delta,\infty))\vspace{1.5mm}\\
&=& {{E}}(\Lambda_{\textup{SV}},U \times \real) + \sigma_{\textup{LV}}|\{\phi(x) <-\delta\} \cap U| +\int_{U} (\sigma_{\textup{SL}} - \sigma_{\textup{SV}})\indicator_{\phi \geq -\delta}  \sqrt{1+|D\phi|^2} \ dx \vspace{1.5mm}\\
&\leq& {{E}}(\Lambda_{\textup{SV}},U \times \real)+\sigma_{\textup{LV}}|\{ \phi < - \delta\} \cap U|+\sigma_{\textup{LV}}\cos\theta_Y(1+\omega(\delta))|\{\phi \geq -\delta\} \cap U|  \vspace{1.5mm}\\
& = & {{E}}(\Lambda_{\textup{SV}},U \times \real)+ \left(1-\big[1-\cos\theta_Y(1+\omega(\delta))\big]\frac{|\{\phi \geq -\delta\} \cap U|}{|U|}\right)\sigma_{\textup{LV}}|U| \vspace{1.5mm}\\
& \leq & {{E}}(\Lambda_{\textup{SV}},U \times \real)+(1-c_0)\sigma_{\textup{LV}}|U|
\end{array}
\end{equation*}
where $c_0 = \frac{1}{2d^{d/2}}(1-\cos\theta_Y)|\Box_1 \cap\{ \phi \geq -\delta\}|$.  Taking the infimum over the allowed $\Lambda_{\textup{SL}}$ and $\Lambda_{\textup{SV}}$ yields the result.

\end{proof}

{\bf $\circ$ Existence of $\integer^d$-periodic minimizers.} In order to establish the existence of $\integer^d$-periodic cell problem solutions we will define a slightly different minimization problem on an appropriate class of periodic sets.  A-priori we need to consider the possibility that the optimal configuration may depend on the length scale.  For this reason we will start by looking at the class of $r\integer^d$-periodic minimizers for integer values of $r \geq 1$. 

\medskip

For the solid-liquid cell problem we define the admissibility class of $r\integer^d$-periodic sets which correspond to a macroscopic solid-liquid interface,
\begin{equation}
\mathcal{A}_{\textup{SL}}^{r\textup{-}per} = \bigg\{ L  \in  BV_{loc}(\real^{d+1} \setminus S) :  L \hbox{ is $r\integer^d$-periodic and} \  \{ z\geq 0\} \subseteq L \bigg\}.
\end{equation}
Analogously, for the solid-vapor cell problem we define for every integer $r \geq 1$ the class of $r\integer^d$-periodic sets which correspond to a macroscopic solid-vapor interface,

\begin{equation}
\mathcal{A}_{\textup{SV}}^{r\textup{-}per} = \bigg\{ L  \in  BV_{loc}(\real^{d+1} \setminus S) :  L \hbox{ is $r\integer^d$-periodic and} \  L \subseteq \{ z< 0\} \bigg\}.
\end{equation}
We point out that when $L \in \mathcal{A}_{\textup{SL}}^{r\textup{-}per}$ is admissible for the solid-liquid cell problem the region occupied by the vapor $V(S,L) = \real^{d+1} \setminus (L \cup S)$ is an admissible region for the solid-vapor cell problem $V \in \mathcal{A}_{\textup{SV}}^{r\textup{-}per}$.

\medskip

The total energy per $r\integer^d$ cell associated with a given $r\integer^d$-periodic solid-liquid-vapor configuration is given by,
\begin{equation}
 {{E}}_{r\textup{-}per}(L) := \sigma_{\textup{LV}} \mathcal{H}^d_{r\textup{-}per}(\partial L \cap \partial V) + \sigma_{\textnormal{SL}} \mathcal{H}^d_{r\textup{-}per}( \partial L \cap \partial S) + \sigma_{\textnormal{SV}} \mathcal{H}^d_{r\textup{-}per}( \partial V \cap \partial S).
\end{equation}

\begin{prop}\label{maximal}
There exist unique $1$-periodic sets $L_{\textnormal{SL}} \in \mathcal{A}_{\textup{SL}}^{1\textup{-}per}$ and $L_{\textnormal{SV}} \in \mathcal{A}_{\textup{SL}}^{1\textup{-}per}$ so that for every integer $r \geq 1$,
\[ {{E}}_{r\textup{-}per}(L_{\textnormal{SL}}) = \inf_{ L \in  \mathcal{A}_{\textup{SL}}^{r\textup{-}per}}{{E}}_{r\textup{-}per}(L) \ \hbox{ and } \ {{E}}_{r\textup{-}per}(L_{\textnormal{SV}}) = \inf_{ L \in  \mathcal{A}_{\textup{SV}}^{r\textup{-}per}}{{E}}_{r\textup{-}per}(L).\]
Furthermore,
\[ {{E}}_{1\textup{-}per}(L_{\textnormal{SL}}) = \overline{\sigma}_{\textup{SL}} \ \hbox{ and } \ {{E}}_{1\textup{-}per}(L_{\textnormal{SV}}) = \overline{\sigma}_{\textup{SV}}.\]
\end{prop}

The main tool in the proof is the union-intersection inequality of Lemma~\ref{lem: union intersection} for the periodic energies:
\begin{lem}\label{lem: S union ineq}
For every $L,L' \in \mathcal{A}_{\textup{SL}}^{r\textup{-}per}$ (or alternatively in $\mathcal{A}_{\textup{SV}}^{r\textup{-}per}$) the inequality holds,
\[ {{E}}_{r\textup{-}per}(L \cup L') + {{E}}_{r\textup{-}per}(L \cap L') \leq {{E}}_{r\textup{-}per}(L) + {{E}}_{r\textup{-}per}(L').\]
\end{lem}
The proof is almost identical to that of Lemma~\ref{lem: union intersection} and is omitted.  We will also need the following result about the lower semi-continuity of ${{E}}$ with respect to $L^1$ convergence of indicator functions.
\begin{lem}\label{lem: lsc}
Let $r\geq 1$ an integer and assume that $\phi$ is smooth.  Suppose that $L_n$ is a sequence of sets of $\mathcal{A}_{\textup{SV}}^{r\textup{-}per}$ so that the indicator functions converge in $L^1$ to the indicator function of a set $L$, then it holds,
\[{{E}}_{r\textup{-}per}(L) \leq \liminf_{n \to \infty} {{E}}_{r\textup{-}per}(L_n).\]
\end{lem}
The proof of this Lemma can be found in the book \cite{MaggiBV} Propositions 19.1 and 19.3.

\begin{proof}[Proof of Proposition~\ref{maximal}]
The existence of a minimizer over the class $\mathcal{A}_{\textup{SV}}^{r\textup{-}per}$ is by the standard direct method of calculus of variations using the lower semi-continuity of the energy functional from Lemma~\ref{lem: lsc}, see for example \cite{CM} or \cite{MaggiBV} chapter 19.

\medskip

Let us show that if $L$ and $K$ are both minimizers of ${{E}}_{r\textup{-}per}$ over $\mathcal{A}_{\textup{SV}}^{r\textup{-}per}$ then $L \cup K$ is a minimizer as well.  This follows in a straightforward way from Lemma~\ref{lem: S union ineq},
$${{E}}_{r\textup{-}per}(L\cup K) - {{E}}_{r\textup{-}per}(L)  \leq  {{E}}_{r\textup{-}per}(K) - {{E}}_{r\textup{-}per}(L \cap K)\leq 0.$$
Similarly one could show that $L \cap K$ is a minimizer.    

\medskip

Call $\mathcal{M}$ to be the set of $L \in \mathcal{A}_{\textup{SV}}^{r\textup{-}per}$ which minimize ${{E}}_{r\textup{-}per}$.  We would like to define $L_{\textnormal{SV}}$ to be the union of all the $F \in \mathcal{M}$, but this union may very well be uncountable.  Let $L_n \in \mathcal{M}$ be a sequence such that,
$$|L_n| \nearrow \sup_{ L \in \mathcal{M}} | L| \ \hbox{ as } \ n \to \infty.$$
  Without loss we can assume that $L_n$ are increasing since we can always replace by $\cup_{1}^n L_j$ which are in $\mathcal{M}$ by the first part of this proof.  As an increasing sequence the indicator functions of $L_n$ converge in $L^1(\real^{d+1})$ to another indicator function of a set $L_{\textnormal{SV}}$. Since the $BV$ semi-norm is lower semicontinuous with respect to $L^1$ convergence of indicators of finite perimeter sets, we know that $L_{\textnormal{SV}} \in \mathcal{A}_{\textup{SV}}^{r\textup{-}per}$.   By Lemma~\ref{lem: lsc}, ${{E}}_{r\textup{-}per}$ is lower semi-continuous with respect to $L^1$ convergence so we obtain,
\[ {{E}}_{r\textup{-}per}(L_{\textnormal{SV}}) \leq \liminf_{n \to \infty} {{E}}_{r\textup{-}per}(L_n) = \inf_{L \in \mathcal{A}_{\textup{SV}}^{r\textup{-}per}} {{E}}_{r\textup{-}per}(L) \leq {{E}}_{r\textup{-}per}(L_{\textnormal{SV}}).\]
Thus $L_{\textnormal{SV}}$ is a minimizer of ${{E}}_{r\textup{-}per}$ over $\mathcal{A}_{\textup{SV}}^{r\textup{-}per}$.  

\medskip

Now we show that $L_{\textnormal{SV}}$ is maximal.  Suppose that $K \in \mathcal{M}$ is another minimizer, then $K \cup L_{\textnormal{SV}}$ is a minimizer as well.  If $| K\cup L_{\textnormal{SV}}| > | L_{\textnormal{SV}}| = \sup_{ L \in \mathcal{M}} |L|$ we obtain a contradiction so it must hold that $|K\setminus L_{\textnormal{SV}}| = 0$.  This is the desired result.

\medskip

Now since the maximal $r\integer^d$-periodic minimizer $L_{\textnormal{SV}}$ is unique, it must share the periodicity lattice of $S$ and be $\integer^d$-periodic.  Thus we can define a single $L_{\textnormal{SV}}$, independent on $r$, which is the $\integer^d$ periodic minimizer of ${{E}}_{r\textup{-}per}$ over $\mathcal{A}_{\textup{SV}}^{r\textup{-}per})$ for every integer $r \geq 1$.

\medskip

Lastly we explain why $E_{1\textup{-}per}(L_{\textup{SL}}) = \overline{\sigma}_{\textup{SL}}$ where $\overline{\sigma}_{\textup{SL}}$ was defined in Theorem~\ref{thm: cell gen periodic}.  From Theorem~\ref{thm: cell gen periodic},
\[ \overline{\sigma}_{\textup{SL}} = \lim_{r \to \infty} \frac{1}{r^d} \Sigma_{\textup{SL}}(\Box_r).\]
Let $r \in \mathbb{N}$, then $L_{\textup{SL}}$ is admissible for the minimization problem for $\Sigma_{\textup{SL}}(\Box_r)$ and so,
\[ \Sigma_{\textup{SL}}(\Box_r) \leq E(L_{\textup{SL}},\Box_r) \leq r^dE_{1\textup{-}per}(L_{\textup{SL}}).\]
sending $r \to \infty$ yields $E_{1\textup{-}per}(L_{\textup{SV}}) \geq \overline{\sigma}_{\textup{SL}}$.  For the other direction, take any $L\in \mathcal{A}_{\textup{SL}}(\Box_r)$, it can be extended in the natural way to an $r\integer^d$-periodic set locally finite perimeter set on $\real^{d+1}$ which is in $\mathcal{A}^{r\textup{-}per}_{\textup{SL}}$ and,
\[ r^dE_{1\textup{-}per}(L_{\textup{SV}}) \leq E_{r\textup{-}per}(L) \leq E(L,\Box_r) + \sigma_{\textup{LV}}2dMr^{d-1}.\]
Taking the infimum over $L \in \mathcal{A}_{\textup{SL}}(\Box_r)$ admissible and sending $r \to \infty$ yields the other inequality $E_{1\textup{-}per}(L_{\textup{SV}}) \leq \overline{\sigma}_{\textup{SL}}$.

\end{proof}




\section{The Values of the Homogenized Surface Energies}\label{sec: values}

In this section we will consider solid surfaces of a special form, with $\phi$ given by 
\[ \phi(x) := M(\indicator_P(x)-1) \ \hbox{ with } \ \{\phi = 0 \} = P.\]
 Here $P\subseteq \real^d$ is a set of locally finite perimeter, which is invariant under $\integer^d$ translations.

\medskip

Geometrically $S$ is made up of periodically repeating flat-topped pillars with shape $P$ and height $M>0$.  We call
\[ f := |P  \cap \Box_1| \ \hbox{ which is the area fraction of $P$ in a unit periodicity cell.}\]
Two classical models have been used by physicists and engineers to describe the macroscopic contact angle of droplets sitting on a rough surface $S_\e = \e S$ of the form described here.  The first is the Wenzel law \cite{Wenzel1936}
\[ \cos \theta_W = \cos\theta_Y\int_{\mathbb{T}^d} \sqrt{1+|D\phi|^2} \ dx, \]
and the second is the Cassie-Baxter law \cite{cassie1944}
\[\cos \theta_{CB} = \cos\theta_Y f + (1-f).\]
As described in \cite{AlbertiDeSimone}, in the context of the homogenization theory, these models can be understood simply as the energy per unit area of possible test minimizers for the solid-vapor and solid-liquid cell problems.  The {\it Wenzel Model} corresponds, for hydrophobic solids, to testing 
\begin{equation}\label{eqn: W state}
\Lambda_{\textup{SV}} = \emptyset\quad \hbox{ and } \ \Lambda_{\textup{SL}} = \R^{n+1}\setminus S =: L_{\textup{W}}.
\end{equation}

  The {\it Cassie-Baxter Model} corresponds, for hydrophobic solids, to testing 
  \begin{equation}\label{eqn: CB state}
  \Lambda_{\textup{SV}} =\emptyset\quad \hbox{ and } \ \Lambda_{\textup{SL}} =\{x_{d+1}\geq 0\} =: L_{\textup{CB}}.\end{equation}
    In this way we see that, always, the Wenzel and Cassie-Baxter Models serve as upper bounds for the true macroscopic contact angle of the global minimizer, 
\[ |\cos \overline{\theta}_Y| \leq \min\{|\cos \theta_{CB}|, |\cos \theta_W|\}  . \]

\medskip

\begin{figure}[t]
 \centering
 \def\svgwidth{3in}
  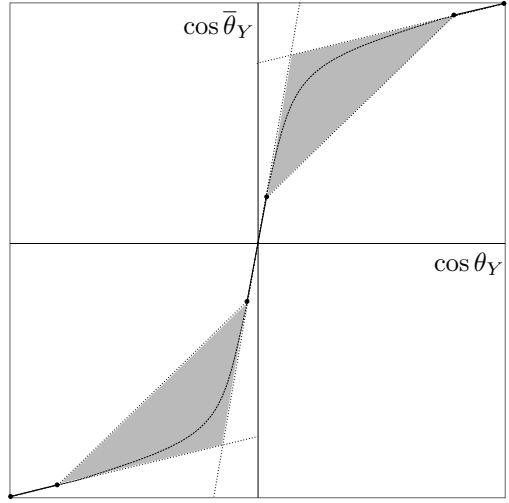
  \caption{An artistic rendering of the graph of $\cos\overline{\theta}_Y$ as a function of $\cos\theta_Y$, it is symmetric with respect to the $\cos \theta_Y = 0$ axis and is a concave function in $0 \leq \cos\theta_Y \leq 1$ which must lie inside the shaded region. }
  \label{fig: cos theta graph}
\end{figure}

Let us consider the graph of $\cos\overline{\theta}_Y$ as a function of $\cos\theta_Y$, see Figure~\ref{fig: cos theta graph}.  Of course it is symmetric with respect to sending $\cos \theta_Y \mapsto -\cos\theta_Y$ by the symmetry discussed previously so we can restrict to discussing the hydrophobic setting.  Recall from \eqref{eqn: normalized energy cell},
\[ 
\begin{array}{lll}
\cos \overline{\theta}_Y &=& \lim_{|\Box| \to \infty}\inf_{L \in \mathcal{A}_{\textup{SL}}(\Box)} \frac{1}{|\Box|}{{E}}'(L,\Box) \vspace{1.5mm}\\
&=&  \lim_{|\Box| \to \infty}\inf_{L \in \mathcal{A}_{\textup{SL}}(\Box)} \frac{1}{|\Box|}\left[ \Per(L,\Box \times \real \setminus S) + \cos\theta_Y \mathcal{H}^d(\partial L \cap \partial S  \cap (\Box \times \real))\right] 
\end{array}
\]
we see that $\cos\overline{\theta}_Y$ is a concave function of $\cos\theta_Y$.  The infimum of linear functions is concave, and the pointwise limit of concave functions is concave.  From this we can guarantee that the graph of $\cos\overline{\theta}_Y$ lies above its secant lines, this is reason for the lower boundary of the shaded region in Figure~\ref{fig: cos theta graph}.

\medskip

Let us call $s(\cos\theta_Y)$ to be the slope of this graph, then we claim,
\begin{equation}\label{eqn: slope bounds}
   f \leq s(\cos\theta_Y) \leq \int_{\mathbb{T}^d} \sqrt{1+|D\phi|^2} \ dx,
   \end{equation}
i.e. the maximal slope is the slope of the Wenzel state, and the minimal slope is that of the Cassie-Baxter state.  If the graph is not smooth, which is quite possible, then consider \eqref{eqn: slope bounds} to be a claim about the slopes of supporting lines at $\cos\theta_Y$.  The inequality \eqref{eqn: slope bounds} follows from the concavity of the graph $\cos \theta_Y \mapsto \cos\overline{\theta}_Y$ and the fact that the Wenzel state $\cos\theta_Y \mapsto \cos\theta_Y\int_{\mathbb{T}^d} \sqrt{1+|D\phi|^2} \ dx$ is a supporting line to the graph at $\cos\theta_Y = 0$ and the Cassie-Baxter state is a supporting line to graph at $\cos\theta_Y = 1$.

\medskip

It was also proven in \cite{AlbertiDeSimone} that when $\phi$ is Lipschitz, which is not the case for our flat topped pillars, and $|\cos\theta_Y|$ is sufficiently small depending on $\|D\phi\|_\infty$, then actually $\theta_Y=\theta_W$.  In this section we make a slight extension of that argument to obtain the similar result about the exactness of the Wenzel for the important case of flat-topped pillars.  We also prove a new result, which is that when $d+1 = 3$ for $|\cos\theta_Y|$ sufficiently close to $1$, the Cassie-Baxter model is exact.

\subsection{Exact solutions in one dimension}  When $d=1$ the problem can, in many situations, be solved explicitly.  This is due to the fact that the liquid-vapor boundary of the cell problem solution will just be a union of line segments.  Such situations are considered in \cite{AlbertiDeSimone}, wherein they point out simple situations for which the cell problem solution is neither Cassie-Baxter nor Wenzel.  On the other hand, one can also compute that for $|\cos\theta_Y|$ small enough (depending on the shape of the surface) the Wenzel state solves the cell problem, while for $1-|\cos\theta_Y|$ small the Cassie-Baxter state solves the cell problem. One can also check, in certain special situations, that Wenzel and Cassie-Baxter are the only possible situations.

\subsection{The Wenzel model is exact for $|\cos\theta_Y|$ close to $0$}  

 We will need to make some assumptions about the regularity of the pillar shapes in order to show that the cell problem minimizer is the Wenzel state that fills in all the holes.  Such assumptions are necessary as for an arbitrary bounded variation set $P$ there may be arbitrarily thin holes which would not be filled in by the droplet even for very small $\cos \theta_Y>0$.  

\medskip

We suppose that one can cover $\partial S$ in the unit periodicity cell by $N$ cylinders with different axes and, in addition, in each cylinder $\Gamma_i$, $S$ is a Lipschitz graph with Lipschitz constant $K_i$ with respect to the axis of the cylinder.  We denote $e_{\Gamma_i}$ to be the axis of the $i$th cylinder and $\Pi_{i}$ to be the projection onto the orthogonal complement of $e_{\Gamma_i}$. The following Lemma says that Wenzel state is achieved if $\cos \theta_Y$ is sufficiently small.

\begin{lem} \label{lem: wenzel}
With the above assumption, suppose that,
 \[
 |\cos\theta_Y| \leq \dfrac{1}{\sum_{i=1}^N\sqrt{1+K_i^2}},
 \]
 then $\cos\bar{\theta}_Y = \cos\theta_W$.
\end{lem}

\begin{proof}
The proof is a slight extension of Section 4.5 of \cite{AlbertiDeSimone} which considers the case when $N=1$ and the cylinder has axis $e_{d+1}$. Following their proof we discuss the hydrophilic case, with $L$ an admissible set for the solid-vapor cell problem. 

\medskip

It is sufficient to show that $\cos\theta_Y| \partial S\cap \partial L | \geq -|\partial L \cap \partial V|$. Let $\Gamma_i$ be one of the $N$ cylinders describe above. Due to our assumption $S$ is a Lipschitz graph $x+f_i(x) e_{\Gamma_i}$, with $f$ having Lipschitz constant at most $K_i$, over $x$ in the base of the cylinder. Then we have the following:
\[
\begin{array}{lll}
|\partial S \cap \partial L\cap \Gamma_i| &\leq & \int_{\Pi_i(\partial S \cap \partial L\cap \Gamma_i)} \sqrt{1+|Df_i|^2}\vspace{1.5mm} \\
&\leq& \sqrt{1+K_i^2} |\Pi_i(\partial S \cap \partial L\cap \Gamma_i)| \vspace{1.5mm} \\
&\leq &\sqrt{1+K_i^2}  |\Pi_i(\partial V \cap \partial L\cap \Gamma_i)|  \vspace{1.5mm}\\
&\leq& \sqrt{1+K_i^2}|\partial V \cap \partial L \cap \Gamma_i| \end{array}
\]
Here in the third inequality we use the fact that at each point of $\partial S \cap \partial L$ in $\Gamma$ one can move above the surface along the axis direction of the cylinder to reach a point in $\partial V \cap \partial L \in \Gamma$. Hence we conclude that $\cos\theta_Y |\partial S \cap \partial L| \geq -|\partial V \cap \partial L|$ if 
\[
|\cos\theta_Y| \leq \frac{1}{\sum_{i=1}^N\sqrt{1+K_i^2}}.
\]
\end{proof}

\subsection{The Cassie-Baxter model is exact for $|\cos\theta_Y|$ close to $1$ in $d+1 = 3$} 

We next show that when $\cos\theta_Y$ is sufficiently close to $1$ the minimal energy state for the solid-liquid cell problem is Cassie-Baxter, i.e. vapor fills in the entire region $\{ \phi(x) < z< 0\}$.  By symmetry this also means that when $\cos\theta_Y$ is sufficiently close to $-1$, the minimal energy state for the solid-vapor cell problem is Cassie-Baxter, liquid fills in the entire region $\{ \phi(x) < z< 0\}$.  The argument given here only works when $d+1 = 2,3$, but these are the dimensions which are physically relevant.  We will restrict to $d+1 = 3$ since in the $d+1=2$ case the cell problem solutions can be computed explicitly making this slightly more complicated argument unnecessary.

\medskip

Because the surface $S = \{ z \leq \phi(x)\}$ is not a set with smooth boundary the existence result of the previous section for the solid-liquid cell problem does not apply. Nonetheless we will simply argue that for any admissible set $L$ the energy is larger than the energy of the Cassie-Baxter state $L_{CB} := \{ z \geq 0\}$.  This will imply simultaneously that a minimizer exists and that said minimizer is $L_{CB}$.

 \begin{lem}
 Let $d=2$ and $\phi(x) = M(\indicator_P(x)-1)$ where $P$ is a non-trivial, smooth boundary, $\integer^d$-periodic set which is the location of the pillars of height $M$ and we call $f := |P \cap \Box_1|$.  There exists a constant $c_0(f) = c (1-(1-f)^{1/2})^2$ (with $c$ a constant independent of $f$) so that if 
 \begin{equation}
 1-\cos\theta_Y \leq  \frac{c_0M^2 \wedge M}{1+c_0M^2 \wedge M}
 \end{equation}
 then the solid-liquid cell problem solution exists and is the Cassie-Baxter state $L_{CB} = \{ z \geq 0\}$ and so
   \begin{equation}
\cos\overline{\theta}_Y = \cos\theta_Y f + (1-f).
 \end{equation}
 \end{lem}
In the proof we will maintain an arbitrary dimension $d \geq 1$ until the point where it is necessary to require $d = 2$.
\begin{proof}
Let $\Box\subset \real^d$ be an open square of side length at least $2$ and let $L$ admissible for the solid-liquid cell problem, i.e.
\[ \Box \times \{ z >0 \} \subseteq L.\]
We will show that ${{E}}(L,\Box) \geq {{E}}(L_{CB},\Box)$ as long as $1-\cos\theta_Y \leq \delta(|P|,M)$, in particular $\delta$ does not depend on the choice of $\Box$.  It is useful to note that for any square $\Box$ with side length $\geq 1$ it holds that $|\Box \cap P| \geq c(d)f|\Box|$ where we recall that $f = |P \cap \Box_1|$ is the area fraction of $P$ in a unit period cell.

\medskip

For this proof it will be simpler to work with the normalized energy,
\[ {{E}}'(L,\Box) = \Per(L, \real^{d+1} \setminus S) + \cos \theta_Y \mathcal{H}^d(\partial L \cap \partial S).\]
As previously described in the introduction the minimizers of ${{E}}'$ are the same as the minimizers of ${{E}}$.
\medskip

It is useful to rearrange $L$ to be the super-graph of a bounded variation function on $\real^d$.  This type of rearrangement will come up several times throughout the article.  For each $x \in \Box$ define,
 \begin{equation}\label{eqn: w defn sec 4}
 w(x) := \int_{-M}^0 \varphi_{L}(x,t) \ dt.
 \end{equation}
Let $T := \{ w = M\}$, this is exactly the subset of $\Box$ where the liquid fills all the way from $z = -M$ to $z=0$.   This function $w$ is in $BV(\Box)$ and furthermore satisfies that, 
  \begin{equation}\label{eqn: w prop sec 4}
 P \cap \Box \subseteq \{w(x) = 0 \} \ \hbox{ and } \ \int_{\Box} \sqrt{1+|Dw|^2} \ dx \leq \Per(L,\Box \times \real).
 \end{equation}
 
This is the result of \cite[Lemma 14.7]{Giusti}.  Let us keep in mind that the measure $Dw$ may have a singular part so the integral needs to be defined in the sense of \eqref{area}. We do not claim that $L$ is exactly the region above the graph of $-w$ in general, it will not be necessary to prove that for our purposes.   Now consider the change in energy by removing the part of the droplet below the level $z=0$, i.e. we replace $L$ by $L_{CB}$, 
\begin{equation}\label{eqn: 1 energy in sigma to 1}
\begin{array}{lll}
 {{E}}'(L_{CB},\Box) - {{E}}'(L,\Box) &=&  \cos\theta_Y|P| + |\Box \setminus P| -\cos\theta_Y(|P|+|T| + \mathcal{H}^d(\partial L \cap (\partial P \times (-M,0)))  \vspace{1.5mm}\\
 & &  \quad - \Per(L, \real^{d+1} \setminus S). \vspace{1.5mm}\\
 & \leq & \left[(1-\cos\theta_Y) |T| - \cos \theta_Y\int_{\Box } \big(\sqrt{1+|Dw|^2}-1\big) \  dx\right]
 \end{array}
 \end{equation}
Here we have used \eqref{eqn: w prop sec 4} and that $1 \geq \cos \theta _Y$.  Thus our proof amounts to being able to show that 
\[
\int_{\Box } (\sqrt{1+|Dw|^2}-1) \  dx \geq c(d,f,M)|T|.
\] 

\medskip

 Essentially we would like to use some kind of Sobolev type inequality to bound from below 
 \[\int_{\Box} (\sqrt{1+|Dw|^2}-1) \  dx \geq c\int_{\Box} |w| \ dx, \]
  and then use that $w=M$ on $T$ to lower bound that $L^1$ norm. 
 
 \medskip
 
 We start by dividing up $\Box$ into disjoint squares $(\Box_i)_{i \in I}$ for some index set $I$ of equal side length $r \in [1,2]$ so that 
 \[ \Box = \cup_{i \in I} \Box_i.\]
 These squares will be of the form $\Box_i = \Pi_{n=1}^d I_n$ where the intervals $I_n$ are either of the half open half closed form $[a_n,b_n)$ or open $(a_n,b_n)$ depending whether the square is at the boundary of $\Box$.  Since $|\{ w = 0 \} \cap \Box_i| \geq | P \cap \Box_i| \geq c(d)f|\Box|$ the following Poincare-type inequality holds (see for instance \cite{EvansGariepy}):
\begin{equation}
( \int_{\Box_i} |w|^{\frac{d}{d-1}} \ dx )^{\frac{d-1}{d}} \leq C \int_{\Box_i} |Dw| \ dx,
\end{equation}
with constant $C$ depending on $d$ and the lower bound for $| \{ w = 0 \}|/|\Box_i|$ (i.e. $c(d)f$ from above).  We argue separately in each $\Box_i$ and split the integral based on a parameter $a \in (0,1]$ to be chosen depending on $i$, applying Lemma~\ref{lem: surface area to BV},
\begin{align*}
 \int_{\Box_i } (\sqrt{1+|Dw|^2} - 1) \ dx & \geq a\int_{\Box_i}|Dw| \ dx - a^2|\Box_i|  \\
 & \geq C^{-1}a\left(\int_{\Box_i} |w|^{\frac{d}{d-1}} \ dx\right)^{\frac{d-1}{d}}-2^da^2 \\
 & \geq C^{-1}Ma|T \cap \Box_i|^{\frac{d-1}{d}} - 2^da^2
 \end{align*}
  in the second inequality we used the Moser-type inequality and that $|\Box_i| \leq 2^d$ and in the third we used that $w=M$ on the set $T$.  Let us give the name  $c_1:=\frac{1}{2}C^{-1}$ to keep track of that constant.

\medskip

   When $d=2$ we can choose $a = \frac{1}{5}c_1(M \wedge 1)|T \cap \Box_i|^{\frac{1}{2}}$ so that 
 \[ \int_{\Box_i } (\sqrt{1+|Dw|^2} - 1) \ dx \geq \tfrac{1}{25}c_1^2 (M^2 \wedge M)|T \cap \Box_i| .\]
 Note that for $d>2$ there is no choice of $0<a\leq 1$ to get above linear lower bound in terms of $|T|$. Now summing this estimate over $\Box_i$ we obtain,
 \[ \int_\Box (\sqrt{1+|Dw|^2} - 1) \ dx \geq \tfrac{1}{25}c_1^2 (M^2 \wedge M)\sum_{i \in I}|T \cap \Box_i| =  \tfrac{1}{25}c_1^2 (M^2 \wedge M)|T|.\]
Now finally we can plug this estimate back into \eqref{eqn: 1 energy in sigma to 1} to obtain,
\[ {{E}}'(L_{CB},\Box) - {{E}}'(L,\Box) \leq \left[(1-\cos\theta_Y) - \tfrac{1}{25}c_1^2 (M^2 \wedge M)\cos\theta_Y\right]|T| \]
which is non-positive for $1-\cos\theta_Y \leq \frac{ \tfrac{1}{25}c_1^2 (M^2 \wedge M)}{1+ \tfrac{1}{25}c_1^2 (M^2 \wedge M)}$.
\end{proof}

%
%
%
%
%
%




\section{Volume Constrained Minimizers and their Interior Regularity}\label{sec: vol}
From this section we go back to the original volume constrained minimization problem for ${{E}}_\e$ described in the introduction.  Fix a positive volume $\textup{Vol}>0$ and bounded closed set $\Omega \subset \real^{d+1}$ with smooth (or piecewise smooth) boundary and $\Omega^+$ contains a ball of volume at least $\textup{Vol}$. For concreteness, and because it makes the estimates more clear, take $\Omega = \overline{U} \times [-T,T]$ for some $T>0$ and a bounded domain $U$ with piecewise smooth boundary.  Recall the energy ${{E}}_\e(L, \Omega)$ given \eqref{main}. We consider the problem of finding ${L} \subset \Omega $ the volume constrained global minimizer of the capillary energy $E_\e$,
\begin{equation}\label{eqn: vol prob}
L  = \argmin \bigg\{ {{E}}_\e(\Lambda,\Omega): \Lambda \in \mathcal{A}(\textup{Vol},\Omega) \bigg\} 
\end{equation}
\[
\hbox{ in the admissible class } \ \mathcal{A}(\textup{Vol},\Omega) = \bigg \{ \Lambda \in BV(\Omega \setminus S_\e) \ \hbox{ with } \ |\Lambda| = \textup{Vol} \bigg\}
\]
and where $S_\e = \{ (x,z) : z \leq \e \phi(\tfrac{x}{\e})\}$. 

\medskip

We will present in this section some preliminary results about existence, stability and regularity properties of the minimizer. The results of this section, as they are not much related to the nature of the solid surface, can mostly be found in other works, in particular \cite{CM, caffarellicordoba}.  Some small modifications are necessary and we explain them here.

\medskip

Note that we consider ${{E}}_\e(\cdot, \Omega)$ which includes the area $\partial L \cap \partial \Omega$ (as a liquid-solid or liquid-vapor interface) in the energy.  If we were to ignore this part of the energy, for example by considering ${{E}}_\e(\cdot, \Omega)$, the cost of surface area on $\partial \Omega$ would be zero, and the minimizing droplet (at least in the hydrophobic case) would stick to the walls of $\Omega$.  This is not desirable as $\Omega$ is not intended to represent a physical surface, rather it is a constraint which is motivated by mathematical considerations, in particular it is necessary for the initial compactness argument establishing the existence of an energy minimizer. When the meaning is clear from context we will write $E(\cdot) = E_\e(\cdot,\Omega)$ for brevity.

\medskip

In Section~\ref{sec: height constraint} we will show that one can replace $\Omega$ bounded by $\Omega = \overline{U} \times \real$ which is unbounded in the $z$-direction, but still bounded in the $x$ directions.  The constraint in the $x$-directions cannot in general be removed for two reasons.  The first is that the energy would always be infinite if the domain $U$ was taken to be $\real^d$.  The second problem is that in the hydrophilic case the solid-vapor cell problem solution may be non-trivial, in that case we generally expect that the volume constrained minimizer $L$ will fill in the region $\{\e \phi(x/\e) \leq z \leq 0\}$ (in the form of the solid-vapor cell problem solution).  If $U$ is taken to tend to $\real^d$ for a fixed $\e$ then all of the volume of $L$ will want to lie below the $z=0$ level.  We consider this behavior to be non-physical and thus the spatial constraint in the $x$-directions.  The first problem can be fixed (somewhat) by considering the normalized energy ${{E}}'$ but the second cannot. 

\subsection{Existence of constrained minimizers}  We first prove the existence of minimizers of our constrained variational problem \eqref{eqn: vol prob}.  We follow the usual direct method of calculus of variations. We remark that the results of this section do not require $|\cos\theta_Y|<1$.

\medskip

First we consider the lower semi-continuity of the energy functional.  This is where we use the assumption that $\phi$ is smooth (see \cite[Proposition 19.1, 19.3]{MaggiBV})
\begin{lem}\label{lem: lsc omega}
Assume that $\phi$ is at least $C^{1,1}$. Suppose that $L_n$ is a sequence of sets of $BV(\Omega)$ so that the indicator functions converge in $L^1$ to the indicator function of a set $L$. Then it follows
\[{{E}}(L,\Omega) \leq \liminf_{n \to \infty} {{E}}(L_n,\Omega).\]
\end{lem}

From this result, using the simple a-priori bound of the perimeter,
\begin{equation}\label{eqn: perim a priori}
 \sigma_{\textup{LV}}\textup{Per}(L) \leq E(L,\Omega) + \sigma_{\textup{LV}}|\cos\theta_Y| \mathcal{H}^d(\partial S \cap \Omega)
 \end{equation}
with the usual compactness result for a sequence of sets with uniformly bounded perimeter one gets the existence of a minimizer (see \cite[Theorem 19.5]{MaggiBV}).
\begin{prop}[]\label{prop: existence}
There exists a minimizer $L \in \mathcal{A}(\textup{Vol},\Omega)$ of the problem \eqref{eqn: vol prob}.
\end{prop}
Note we are using also the assumption that $\Omega \setminus S$ contains some ball of volume $\textup{Vol}$ which also gives the a-priori bound on the energy minimum,
\begin{equation}\label{eqn: energy a-priori}
 E(L,\Omega) \leq C_d\sigma_{\textup{LV}}\textup{Vol}^{\frac{d}{d+1}}.
 \end{equation}
 We also remark that, arguing similarly to \eqref{eqn: perim a priori}, we get a useful a-priori bound on the perimeter of $L$ in $z>0$ of the form,
 \begin{equation}\label{eqn: perim a-priori 2}
 \Per(L,\{ z>0\}) \leq \sigma_{\textup{LV}}^{-1}E(L,\Omega) \leq C_d\textup{Vol}^{\frac{d}{d+1}}.
 \end{equation}

 \subsection{Removing the constraint on the droplet height}\label{sec: height constraint} Now we would like to remove the constraint that the confining region $\Omega$ is bounded in the $z$ direction, to consider regions of the form $\overline{U} \times \real$.  Such domains are particularly natural since the total energy $E$ of the minimizer as well as the volume of the water droplet below the $z=0$ level naturally scale with $|U|$.  The result is from Caffarelli and Mellet \cite{CM} based on an argument by Barozzi \cite{Barozzi}.  The arguments involve only modifications of the minimizer away from the $z=0$ level and therefore are not affected by the rough surface.  
 \begin{lem}[Proposition 3 of \cite{CM}]
 Let $\Omega_T = \overline{U} \times [-T,T]$, there exists $T_1$, a universal constant multiplied by $\textup{Vol}^{\frac{1}{d+1}}$, so that for $T \geq T_1$ there exists $L \in \mathcal{A}(\textup{Vol},\Omega_{T_1})$ satisfying,
 \[ E(L,\Omega_{T}) = \min_{\Lambda \in \mathcal{A}(\textup{Vol},\Omega_{T})} E(\Lambda,\Omega_T).\]
 \end{lem}
 From this it is easy to see that this $L$ is a minimizer of $E$ over the admissible class $\mathcal{A}(\textup{Vol},\overline{U} \times \real)$.  From now on we can and will take $\Omega = \overline{U} \times \real$.

\subsection{Perturbing the minimizer}  

For our perturbation arguments later, it is very useful for this reason to compare the minimal energies corresponding to volume constrained minimization with slightly perturbed constraint.  

\begin{lem}\label{lem: volume change}
Let $\Omega:= \overline{U}\times\real$. There is a dimensional constant $C(d)$ so that for every $\delta>0$, as long as $\e < \textup{Vol}/(2M|U|)$,
\begin{equation}\label{eqn: volume change}
 \min_{\Lambda \in \mathcal{A}(\textup{Vol},\Omega)} {{E}}(\Lambda,\Omega)\leq  \min_{ \Lambda \in \mathcal{A}(\textup{Vol}+\delta,\Omega)} {{E}}(\Lambda,\Omega) \leq \min_{ \Lambda \in \mathcal{A}(\textup{Vol},\Omega)} {{E}}(\Lambda,\Omega)+C\sigma_{\textup{LV}}\textup{Vol}^{-\frac{1}{d+1}}\delta .
 \end{equation}
 \end{lem}
 The proof needs to be a bit different from \cite{CM} due to the rough surface.  From now on we will always assume that $\e <\textup{Vol}/(2M|U|)$ so that the monotonicity result of Lemma~\ref{lem: volume change} will always apply.  We remark that the same result  holds for the normalized energy $E'$ without the factor of $\sigma_{\textup{LV}}$ appearing on the right.
 
 \begin{proof}[Proof of Lemma~\ref{lem: volume change}]
\medskip

First let us prove the inequality to the left. Let $L_\delta$ be the minimizer of $E$ over the admissible class $\mathcal{A}(\textup{Vol}+\delta,\Omega)$.  We modify $L$ to have volume $\textup{Vol}$.  Due to the assumption that $\e < \textup{Vol}/(2M|U|)$ we have,
\[ |L_\delta^+| \geq \textup{Vol}+\delta - |(\Omega \setminus S) \cap \{ z \leq 0\}| \geq \textup{Vol}+\delta - M\e|U| > \delta\]
Then by continuity there exists a $t> 0$ so that,
\[ L_\delta' := L_\delta \cap \{ z \leq t\} \ \hbox{ has } \ |L_\delta'| = \textup{Vol}.\]
Then $L_\delta'$ is in the admissible class $\mathcal{A}(\textup{Vol},\Omega)$ and so, by the minimization property,
\[ \min_{\Lambda \in \mathcal{A}(\textup{Vol},\Omega)}E(\Lambda,\Omega) \leq E(L_\delta',\Omega) \leq E(L_\delta,\Omega),\]
where for the last inequality we have used that our modification $L_\delta \mapsto L_\delta'$ does not increase the energy, see Lemma~\ref{lem: lines through}.

\medskip

For the second inequality, let $L$ be a minimizer of $E$ over the admissible class $\mathcal{A}(\textup{Vol},\Omega)$.  We modify $L$ to have volume $\textup{Vol}+\delta$.  Note that from the assumption $\e < \textup{Vol}/(2M|U|)$ we have,
\[ |L^+| \geq \textup{Vol} - |(\Omega \setminus S) \cap \{ z \leq 0\}| \geq \textup{Vol} - M\e|U| \geq \tfrac{1}{2}\textup{Vol}. \]
With $t = \delta/|L^+|$ we dilate the part of $L$ above the level $z=0$,
\[ L_t = ((1+t)L^+) \cup L^- \ \hbox{ which has } \ |L_t| = \textup{Vol}+\delta.\]
This modification increases the area of the liquid vapor interface (above level $z=0$) but the surface interfaces stay the same leading to,
\[E(L_t,\Omega) \leq E(L,\Omega) + \sigma_{\textup{LV}}t\Per(L,\{z>0\}) \leq E(L,\Omega) + C_d\sigma_{\textup{LV}}t \textup{Vol}^{\frac{d}{d+1}}\]
where we used the a-priori bound \eqref{eqn: energy a-priori} for the second inequality.  Using that $t \leq 2\delta/\textup{Vol}$ we conclude. 
 \end{proof}

\subsection{Non-degeneracy of the free surface}  Next we state the measure theoretic regularity of the liquid boundary which we refer to as \emph{interior density estimates} since it requires to be away from the solid surface.
\begin{lem}[Interior density estimates]\label{lem: meas reg}
There exist universal constants $C,c>0$ so that for any $(x,z) \in \partial_e L$ and every $r <z$,
\begin{equation}
 c\leq \frac{|L \cap B_r(x,z)|}{|B_r|} \leq 1-c \ \hbox{ and } \ C^{-1}r^d\leq  \Per(L,B_r(x,z)) \leq  Cr^d.
 \end{equation}
 
\end{lem}
The proof is based only on modifications of the minimizer localized to $B_r(x,z)$ which does not involve the solid surface and so it is almost the same as \cite[Lemma 7]{CM}.  There is a minor technical difference which has to do with the confining region $\Omega$ and so we provide in the Appendix.

\subsection{Higher regularity of the free surface}
  In the next section we make use of the interior regularity result for minimal surfaces in $\real^n$, in particular the almost-flatness statement in \cite{caffarellicordoba}. While stronger regularity results should apply for our setting, we find that the argument presented in \cite{caffarellicordoba} is relatively straightforward to modify for our situation. First of all we define an almost minimizer of the perimeter functional.  
\begin{DEF}
We say $\Lambda \subset \real^n$ is an \emph{almost minimizer} of the perimeter functional in an open ball $B \subset \real^n$ if, for some $h>0$ and every $K$ with $\Lambda \Delta K \subset \subset B$,
  \[ \Per(\Lambda,B) \leq \Per(K,B) + h|\Lambda \Delta K|.\]
  \end{DEF}
\begin{DEF}\label{flat}
We say that $\left.\partial \Lambda \right|_{B_\lambda(0)}$ is $\delta$-\emph{flat} if in some coordinate system we have,
\[ \partial \Lambda \cap B_\lambda(0) \subset \{ |x_{n}| \leq \delta \lambda \}.\]
\end{DEF}
We make note that, by Lemma~\ref{lem: volume change}, our volume constrained minimizer $L \subset \real^{d+1}$ is an almost minimizer of the perimeter function in $B_r(x,z)$ for any $r<z$ with constant $h = C_d\textup{Vol}^{-\frac{1}{d+1}}$.

\medskip

\begin{lem}\label{lem: minimal surface interior}
Suppose that $\Lambda \subset \real^n$ is an almost minimizer of the perimeter in $B_1(0)$ with constant $h>0$.  There exists $C(n)>0$ so that for all $\delta>0$, $\frac{1}{4}>\lambda>0$ and any covering of 
\[\{x \in B_{1/2}(0) : \left.\partial \Lambda \right|_{B_\lambda(x)}  \hbox{ is not $\delta$-flat} \}\]
 by $N$ balls of radius $\lambda$ with finite overlapping,
\[ N \leq C(1+h) \delta^{-1} \lambda^{2-n}.\]
\end{lem}

{\it Sketch of the proof.} We will only outline the proof of above lemma, since the arguments are parallel to that of  Lemma 5 in \cite{caffarellicordoba}. Their proof is based on the usage of distance function $d$ the (almost) minimal surface $B \cap \partial \Lambda $. More precisely the main step in the proof of Lemma 5 in \cite{caffarellicordoba} is to show that the distance function is superharmonic away from $\partial \Lambda$. In our case the statement would be that the distance function $d(x) := \textup{dist}(x,B \cap \partial \Lambda)$ satisfies 
\begin{equation}\label{superharmonic}
\Delta d \leq h\hbox{ in  }  B \setminus \partial \Lambda \hbox{ in the sense of viscosity solutions,}
\end{equation}
 where $h$ is the constant from the almost minimizer condition.

\medskip

Let $K$ be a perturbation of $\Lambda$ with $\Lambda \Delta K \subset \subset B$.  When $\Lambda$ and $K$ are smooth then, calling $A:= \Lambda\triangle K$, the distance function $d=d(x,\partial L^+)$ satisfies 

\[
 \int_A \Delta d \,dx = \int_{\partial A} \nu_A \cdot Dd\, dS_A \leq \Per(K,B) - \Per ( L,B) \leq h|A|,
\]
where $\nu$ denotes the outward normal and we have used the fact that $\nu_A \cdot Dd = -|Dd|=-1$ on $\partial L  \cap \partial A$.

\medskip

Using this fact, one can still prove Lemma 5 in \cite{caffarellicordoba}.




\section{Macroscopic regularity of the contact line} \label{sec: reg}

This section addresses the large-scale regularity of minimizing drops near the solid-liquid vapor contact line. We present two regularity results which, as explained below, describes different aspects of the drop surface near the surface. Parallel statements have been established \cite{CM} in the case of a flat chemically textured surface. In our case we face the possible irregularities of the drop boundary, such as air pockets, close to the rough surface.  This unknown near-surface geometry of the minimizer generates considerable challenges in the analysis.  

\medskip

We first proceed to show Proposition~\ref{prop: perim 0 t} which bounds the $(d-1)$ dimension Hausdorff measure of the ``contact line", at least when measured above a certain scale $r_0(\e) \gg \e$.  This result states that one can ignore the behavior of the minimizer near the contact line to measure its energy as $\e\to 0$. Indeed we observe that this estimate alone is sufficient to prove the convergence of the energies as $\e \to 0$ (to the energy of the global minimizer of the homogenized problem): see  Section~\ref{sec: hom} (Theorem~\ref{thm: main periodic}). 

\medskip

Our second result is Proposition~\ref{prop: non-degen}, which can be viewed either as a large scale estimate of the non-degeneracy of the contact angle or as a density estimate (like Lemma~\ref{lem: meas reg}) which holds (almost) up to the solid surface.  The (almost) boundary density estimate in turn yields the convergence estimate in Hausdorff distance of the liquid boundary as $\e\to 0$, at least down to a certain length scale $z \sim h_0(\e)$ (see Theorem~\ref{thm: main periodic}(ii)). This estimate puts a bound on the size of the ``boundary layer" near the solid surface where the influence of the small scale irregularities is felt.  At the moment it is not clear to the authors whether this estimate is optimal in terms of the size of $h_0(\e)$.  We believe, however, that our proof provides a useful framework which could potentially be improved to achieve a sharp result.

\medskip

\subsection{Upper bound of the total perimeter near the surface}
Our goal in this section is to obtain a bound of the following form
\begin{equation}\label{eqn: perimeter bound}
 `` \quad \Per(  L , \{ 0 < z< t \} ) \leq C t \ \hbox{ for all } \ t >0. \quad "
 \end{equation}
 Such result does hold for the flat boundary with $x$-dependent surface energies (see \cite{CM}), but we need to be a bit more careful with the rough boundary.  We do expect, at least when the cell problem solutions are not the Wenzel state, that there is a significant amount of free surface area (i.e. liquid-vapor interface), of order the size of the macroscopic wetted set, near the level $z=O(\e)$.  Therefore, we do not rule out the possibility that $\mathcal{H}^d( \partial L \cap \{ 0 < z< t \})$ is non-vanishing as $t \to 0$.  Nonetheless the following estimate, which stays away from $z\leq O(\e)$, turns out to be sufficient for the homogenization result in section 7.

 \begin{prop}\label{prop: perim 0 t}
There exist constants $C_0 ,C_1, R_0 \geq 1$ universal so that,
 \begin{equation}\notag\label{eqn: perimeter bound delta t} 
 \Per\big( L , \{ \tfrac{1}{2} t < z< \tfrac{3}{2}t \} \big) \leq C_0(\textup{Vol}^{-\frac{1}{d+1}}|U| +|\partial U|) t \quad \hbox{for all} \quad t \geq  r_0(\e)
 \end{equation}
 where the scale $r_0(\e)$ is defined,
 \[ r_0(\e):=R_0\e\exp(C_1 |\log (\textup{Vol}^{-\frac{1}{d+1}}\e)|^{1/2}).\]
 \end{prop}

 We remark that our methods would also allow for a correctly scaling localized version of this estimate.  See \cite[Lemma 2.10]{DePMag} where a similar result is proved in the case of a smoothly varying boundary with smoothly varying interfacial energies.

\medskip

Let us briefly discuss the outline of the proof for Proposition~\ref{prop: perim 0 t}.  For the flat surface case which was done in \cite{CM}, one achieves \eqref{eqn: perimeter bound} by comparing the energy of $L$ with the set $L$ shifted down by t and cut off at level $z = 0$.  In this process, since the surface is flat, we replace any regions of solid-liquid-vapor layers (ordered by increasing height) by solid-vapor interfaces at level $z=0$ and any regions of solid-vapor-liquid layers by solid-liquid interfaces at level $z=0$.   This decreases the energy, proportional to the perimetor of $L$ that was removed (since, respectively, $\sigma_{\textup{LV}}+\sigma_{\textup{SL}} > \sigma_{\textup{SV}}$ and $\sigma_{\textup{LV}}+\sigma_{\textup{SV}} > \sigma_{\textup{SL}}$). The volume also changes, but if the energy change is large enough i.e. the perimeter of $L$ up to $t$-level is too large, then we can contradict the monotonicity formula \eqref{eqn: volume change}.   
 
 \medskip
 
 In our case, if we try naively to make the same argument shifting down $L$ and cutting off at level $z=0$, it is not clear whether the energy decreases proportionally to the total perimeter of $L$ that was removed.

\medskip

  Intuitively, if the drop boundary was sufficiently regular (i.e. at a scale $\gg \e$) in the region $z \sim t$, then $(1)$ the area of the ``vertical" parts of $\partial L$ would be eliminated by shifting down the drop resulting in a strict decrease of the energy while $(2)$ for the ``horizontal" parts of the liquid boundary we could replace by an appropriate cell problem solution at the $z=0$ level and use that $|\cos\overline{\theta}_Y|<1$ to get a strict decrease of the energy.  The actual computation is rather delicate and involves covering arguments by eccentric cylinders (which are useful to distinguish between vertical and horizonta parts of the liquid boundary) due to the fact that the liquid boundary is only ``mostly flat" from Lemma~\ref{lem: minimal surface interior}. 
 
 \medskip
 
The above simple idea, unfortunately, translates into rather lengthy proof as we will see below. To carry out the argument we need to first localize at a larger scale to apply Lemma~\ref{lem: minimal surface interior} and deal with the boundary $\partial U$ and then at a smaller scale to use the result of Lemma~\ref{lem: minimal surface interior}.  All of the covering arguments are just a (necessary) precursor to delicate energy arguments which deal separately with the vertical and horizontal parts of the liquid boundary.  Finally, to get the (almost) optimal result, there is a bootstrap argument which comes up (which is difficult to describe outside of the context of the proof).

\begin{proof}[Proof of Proposition~\ref{prop: perim 0 t}]
For the purposes of the proof we can assume that $\textup{Vol} = 1$.  The result of the Proposition can be obtained by scaling.

\medskip

  Let us define the cylinder centered at a point $(x_0,z_0)$ with radius $r$ and height $h$ by,
 \[ \Gamma_{r,h}(x_0,z_0) := \{ (x,z) \in \real^{d} \times \real : |x-x_0| \leq r \ \hbox{ and } \ |z-z_0| \leq h\}.\]
 For $\lambda >0$ and a cylinder $\Gamma$ in the class above we call $\lambda \Gamma$ to be the cylinder with the same center with height and radius scaled by $\lambda$.  Let $\gamma > 0$ and consider the class of cylinders with eccentricity $\gamma$, $\Gamma_{r,\gamma r}(x,z)$ for $r>0$ and $(x,z) \in \real^d \times \real$.  We will need to use the Vitali covering lemma with this class, for any finite collection $\{\Gamma_j\}_{1 \leq j \leq N}$ of cylinders with eccentricity $\gamma$, there is a disjoint subcollection $J \subset \{1, \dots ,N\}$ so that 
 \[\cup_{1\leq j \leq n} \Gamma_j \subset \cup_{j\in J} 3 \Gamma_j.\]
 We only consider $t \leq \frac{1}{36}$ -- when $t \geq \frac{1}{36}$ the result of the Proposition is immediate from the bound of the total perimeter of the volume constrained energy minimizer \eqref{eqn: perim a-priori 2}. 
 
 \medskip
 
 As described above, a sub-optimal bound on the number of bad cylinders where $\partial L$ is not $\delta$-flat can be used in a bootstrap type argument to obtain a better bound on the number of said bad cylinders.  This requires a slightly larger window of $z$ at each iteration which is why we will start by considering bounding the area of $\partial L$ in,
 \[ W_k := \{ 2^{-k} t < z < 2(1-2^{-k})t\} \ \hbox{ for } \ k = 1,2,3 \dots\]
 We will need to be careful about the dependence of the estimates on $k$, for now consider $k$ to be fixed.  
 
 \medskip
 
 We cover $\partial_e L \cap W_k$ by cylinders twice, the first covering is mostly just a localization so that we can apply the minimal surface regularity result Lemma~\ref{lem: minimal surface interior}.  By compactness we can cover $\partial_eL \cap W_k$ by cylinders $G_i$, for $1 \leq i \leq N_1$, of height $\tfrac{1}{6}2^{-k}t$ and radius $\tfrac{1}{6}2^{-k}t$ centered at points of $\partial_eL \cap W_k$. We call $s:= \tfrac{1}{6}2^{-k}t$ to be the radius/height of the cylinders $G_i$.  By Vitali's covering lemma there is a subcollection subcollection $I \subset \{1,\dots,N_1\}$ so that the $\{G_i\}_{i \in I}$ are disjoint and so that $\{3G_i \}_{i \in I}$ covers $\partial L \cap W_k$.

 \medskip

 Next we cover by a collection of smaller cylinders (with an eccentricity) which will be used in the energy argument.  As above we can cover $\partial_eL \cap W_k$ by cylinders $\Gamma_j$, for $1 \leq j \leq N_2$, centered at points $(x_j,z_j) \in \partial_e L \cap W_k$ of height $ \lambda 2^{-k} t$ and radius $\gamma^{-1}\lambda 2^{-k} t$. Here the parameter $0 < \lambda \leq \frac{1}{36}$ will be chosen later (depending on $t$), and we will eventually choose $\gamma >4 $ depending at most on the dimension $d$.  For simplicity we call 
 \[ r := \gamma^{-1} \lambda 2^{-k} t \ \hbox{ the {\it radius} of the cylinders $\Gamma_j$.}\]
 By Vitali's covering lemma there is a subcollection $J \subset \{1,\dots,N_2\}$ so that the $\{3\Gamma_j\}_{j \in J}$ are disjoint and so that $\{9\Gamma_j \}_{j \in J}$ covers $\partial L \cap W_k$.  We also denote the `disk' $D_j\subset \real^d$ which is the projection of $\Gamma_j$ onto $\{ z= 0\}$,
 \[ D_j := \{ x \in \real^d : |x-x_j| \leq r\}.\]
 We remark since it will be relevant later that although the $3\Gamma_j$ are disjoint, the $3D_j$ may overlap.
 
 \medskip
 
 Now we make our first of several subdivisions of $I$ and $J$, separating out the large scale cylinders which near the boundary $\partial U \times \real$.  We call,
 \[ I_{\partial U} = \{ i \in I : 6G_i \cap (\partial U \times \real) \neq \emptyset \}.\]
 Then we also separate out the small scale cylinders $\Gamma_j$ which are too close to the boundary,
 \[ J_{\partial U} = \{ j \in J : 9\Gamma_j \cap 3G_i = \emptyset \ \hbox{ for all } \ i \in I \setminus I_{\partial U}\}.\]
 Since the $\{3G_i\}_{i \in I}$ cover $\partial_e L \cap W_k$ every $9\Gamma_j$ intersects at least one $G_i$.  The purpose of our definition is so that every $j \in J \setminus J_{\partial U}$ the corresponding $9\Gamma_j$ intersects at least one of the ``interior" $G_i$ i.e. with $i \in I \setminus I_{\partial U}$, and also $\{9\Gamma_j\}_{j \in J \setminus J_{\partial U}}$ covers $\partial_eL \cap G_i$ for every $i \in I \setminus I_{\partial U}$.  In particular we get,
 \[ \partial_e L \cap W_k \subset (\bigcup_{i \in I_{\partial U}} 3G_i ) \cup (\bigcup_{j \in J \setminus J_{\partial U}} 9\Gamma_j ).\]
 Based on the above set up, which in particular guarantees $9\Gamma_j, 3G_i \subseteq \{ z>0\}$, Lemma~\ref{lem: meas reg} yields
 \begin{equation}\label{eqn: covering per ineq}
  \begin{array}{ll}
  \Per( L , W_k ) &\leq \sum_{i \in I_{\partial U}}\Per(L, 3G_i) + \sum_{j \in J \setminus J_{\partial U}}\Per(L, 9\Gamma_j) \vspace{1.5mm}\\
  &\lesssim |I_{\partial U}|s^d + |J \setminus J_{\partial U}|r^d \vspace{1.5mm}\\
  &\lesssim \sum_{i \in I_{\partial U}}\Per(L, G_i) + \sum_{j \in J \setminus J_{\partial U}}\Per(L, \Gamma_j) \vspace{1.5mm}\\
  & = \Per(L, \cup_{i \in I_{\partial U}}G_i) + \Per(L, \cup_{j \in J \setminus J_{\partial U}} \Gamma_j)
  \end{array}
  \end{equation}
 Thus, besides dealing with the boundary term $\Per(L, \cup_{i \in I_{\partial U}}G_i)$, to prove \eqref{eqn: perimeter bound delta t} it suffices to obtain the desired bound for the total perimeter of $L$ in the union of the $\Gamma_j$ for $j \in J\setminus J_{\partial U}$.  
 
 \medskip
 
 Notice, making almost the same argument as in \eqref{eqn: covering per ineq} and using that $G_i \subset W_{k+1}$, that in particular we have,
 \begin{equation}\label{eqn: number count}
  |I| \lesssim s^{-d} \Per(L, \cup_{i \in I} G_i)  \leq s^{-d} \Per(L, W_{k+1})  \sim (2^{-k} t)^{-d} \Per(L, W_{k+1}).
  \end{equation}
  This, in some sense, is the reason for looking the windows $W_k$, it is not quite true that $|I|s^d \sim \Per(L,W_k)$, we need to put $\Per(L,W_k) \lesssim |I|s^d \lesssim \Per(L,W_{k+1})$.

\medskip
  
  {\bf $\circ$ The perimeter bound near $\partial \Omega$.}  Our first order of business is to bound the perimeter of $L$ near the boundary of the confining domain $U \times \real$. This will follow from the interior density estimates combined with the fact that $\partial U$ has dimension $d-1$.  Using that the $G_i$ are disjoint and have radius $s$, for $i \in I_{\partial U}$ we have $G_i \subset (\partial U + 2D_{s})\times\{0 < z<3t\}$ and so,
  \[  |( \partial U + 2D_{s}) \times \{0 < z<3t\}| \geq \sum_{i \in I_{\partial U}} |G_i| \gtrsim   |I_{\partial U}| s^{d+1}.\]
On the other hand, from the smoothness property of $\partial U$, we have $|\partial U + 2D_s| \lesssim |\partial U| s$ and so,
\[ |I_{\partial U}| \lesssim s^{-d}|\partial U|t.\]
Finally using the upper density estimate of the perimeter from Lemma~\ref{lem: meas reg} (which does hold up to $\partial U$),
\begin{equation}\label{eqn: boundary terms}
 \sum_{i \in I_{\partial U}}\Per(L,G_i) \lesssim s^d | I_{\partial U}| \lesssim |\partial U|t.
 \end{equation}
  This is the desired bound for the perimeter in the cylinders of $I_{\partial U}$, from now one we can work with cylinders which are away from $\partial U$.

  \medskip

{\bf $\circ$ The cylinders where $\partial L$ is not flat.}  Let $\delta > 0 $, which is a dimensional constant to be chosen later, and call,
 \[ J_{bad} = \{ j \in J \setminus J_{\partial U} : \left. \partial L\right|_{3\Gamma_j} \hbox{ is not $\delta$-flat}\}.\]
Here $\delta$-flat is as defined in Definition~\ref{flat}. We will apply Lemma~\ref{lem: minimal surface interior} in each $2G_i$ for $i \in I \setminus I_{\partial U}$ to count $J_{bad}$.  Exactly as above in \eqref{eqn: number count} we can bound the total number of cylinders of $I$ by,
\begin{equation}\label{eqn: P count}
 |I \setminus I_{\partial U}|  \leq |I| \lesssim (2^{-k} t)^{-d} \Per(L, W_{k+1}).
 \end{equation}
Now since $\{3G_i\}_{i \in I}$ cover $\partial L \cap W_k$ we can guarantee that every cylinder $9\Gamma_j$ intersects at least one of the $3G_i$.  Since $j \in J_{bad}$ implies $j \not\in J_{\partial U}$ this means that $9\Gamma_j$ must intersect at least one $G_i$ with $i \in I \setminus I_{\partial U}$, i.e. we know,
\begin{equation}\label{eqn: jbad count}
 |J_{bad}| \leq |I \setminus I_{\partial U}| \max_{ i \in I \setminus I_{\partial U}} |\{ j \in J_{bad}: 9\Gamma_j \cap 3G_i \neq \emptyset\}|. 
 \end{equation}
Now $|\{ j \in J_{bad}: 9\Gamma_j \cap 3G_i \neq \emptyset\}|$ is exactly suited to be counted by the interior (partial) regularity Lemma~\ref{lem: minimal surface interior} applied in $6G_i$ (scaled properly).  For all $i \in I$ we obtain,
 \[ |\{ j \in J_{bad}: 9\Gamma_j \cap 3G_i \neq \emptyset\}| \lesssim (1+s)\delta^{-1}\lambda^{1-d} \lesssim \delta^{-1}\lambda^{1-d}.\]
 We combine this with \eqref{eqn: P count} and \eqref{eqn: jbad count} to obtain,
 \[ |J_{bad}| \lesssim \delta^{-1}\lambda^{1-d}(2^{-k} t)^{-d} \Per(L, W_{k+1}).\]
 Due to Lemma~\ref{lem: meas reg} it follows that
 \begin{equation}\label{eqn: main est nonflat}
 \sum_{j \in J_{bad}} \Per(L,\Gamma_j) \lesssim |J_{bad}| (\lambda2^{-k} t)^d \lesssim \delta^{-1} \lambda \Per(L, W_{k+1}).
 \end{equation}
 As $\delta$ will be a fixed dimensional constant, this is a bound which is good for small $\lambda$.

 \medskip

 {\bf $\circ$ The cylinders where $\partial L$ is $\delta$-flat.}  What remains after removing the boundary cylinders and the cylinders where $\partial L$ is not $\delta$-flat we will call,
 \[ J_{good} =   J \setminus (J_{\partial U} \cup J_{bad}) .\]
In particular $\partial L$ is $\delta$-flat in each cylinder $3 \Gamma_j$ for $j \in J_{good}$ and also $3\Gamma_j \subset U \times \real$. We will prove the following lemma for the parts of the drop boundary which are flat:
 
 \begin{lem}\label{middle_lemma}
 There exist constants $C$ and $R_0$ depending on $1-|\cos\overline{\theta}_Y|$ and $\|\phi\|_\infty$ so that
 \begin{equation}\label{eqn: main estimate flat}
 \sum_{j \in J_{good}} \Per(L,\Gamma_j) \leq C|U|t \ \hbox{ as long as } \  \gamma r= \lambda 2^{-k} t \geq R_0\e.
 \end{equation}
 \end{lem}

 Although we have suppressed the dependence on $k$ the collection of $\Gamma_j$, $J_{good}$ and $r$ depend on $k$, but, importantly, the constants $C$ and $R_0$ in the estimate \emph{do not} depend on $k$.   Notice that, although we would like to choose $\lambda  \to 0$ to improve the estimate of \eqref{eqn: main est nonflat}, the condition of Lemma~\ref{middle_lemma} gives a limit on how small $\lambda$ can be.
 
 \medskip

 If we combine the result of Lemma~\ref{middle_lemma} with \eqref{eqn: boundary terms} and \eqref{eqn: main est nonflat} and we plug into \eqref{eqn: covering per ineq} we will obtain,
 \begin{equation}\label{eqn: main iterate}
 \Per(L,W_k) \lesssim (|U|+|\partial U|)t + \lambda \Per(L,W_{k+1}) \  \hbox{ as long as } \  \lambda \geq\tfrac{2^kR_0\e}{t}.
 \end{equation}
 This bound can be iterated with a good choice of $\lambda$ to obtain the result of the Proposition.  This will be carried out below.  First we need to prove Lemma~\ref{middle_lemma}.
 
 \medskip
 
 To prove Lemma~\ref{middle_lemma} we will make a further classification of the cylinders depending on whether $\partial L$ is ``horizontal" or ``vertical" in $3\Gamma_j$.
 
 \medskip

For each $j \in J\setminus J'$ we know that $\partial L$ is $\delta$-flat in $3\Gamma_j$, that is there exists a unit vector $n_j \in S^d$ so that,
 \begin{equation}\label{eqn: eta flat}
  \bigg\{[(x,z)-(x_j,z_j)] \cdot n_j \leq -3\delta r\bigg\} \subseteq L \cap 3\Gamma_j \subseteq \bigg\{[(x,z)-(x_j,z_j)] \cdot n_j \leq 3\delta r\bigg\}.
  \end{equation}
   We say that (see Figure~\ref{fig: vertical horizontal}):
 \[ \partial L \hbox{ is \emph{horizontal} in $\Gamma_j$ if } \ |n_j \cdot e_{d+1}| \geq 2\gamma^{-1}; \]
 \[ \hbox{$\partial L$ is \emph{vertical} in $\Gamma_j$ if }  \ |n_j \cdot e_{d+1}| < 2\gamma^{-1}.\]
 Let us denote
 \[ J_h := \{ j \in J_{good}: \hbox{ $\partial L$ is horizontal in $\Gamma_j$}\} \ \hbox{ and } \ J_v := \{ j \in J_{good}: \hbox{ $\partial L$ is vertical in $\Gamma_j$}\}.\]

 \begin{figure}[t]
 \centering
 \def\svgwidth{3in}
 \begingroup%
  \makeatletter%
  \providecommand\color[2][]{%
    \errmessage{(Inkscape) Color is used for the text in Inkscape, but the package 'color.sty' is not loaded}%
    \renewcommand\color[2][]{}%
  }%
  \providecommand\transparent[1]{%
    \errmessage{(Inkscape) Transparency is used (non-zero) for the text in Inkscape, but the package 'transparent.sty' is not loaded}%
    \renewcommand\transparent[1]{}%
  }%
  \providecommand\rotatebox[2]{#2}%
  \ifx\svgwidth\undefined%
    \setlength{\unitlength}{595.27558594bp}%
    \ifx\svgscale\undefined%
      \relax%
    \else%
      \setlength{\unitlength}{\unitlength * \real{\svgscale}}%
    \fi%
  \else%
    \setlength{\unitlength}{\svgwidth}%
  \fi%
  \global\let\svgwidth\undefined%
  \global\let\svgscale\undefined%
  \makeatother%
  \begin{picture}(1,.7)%
    \put(0,0){\includegraphics[clip, trim=0cm 0in 0cm 0in, width=\unitlength]{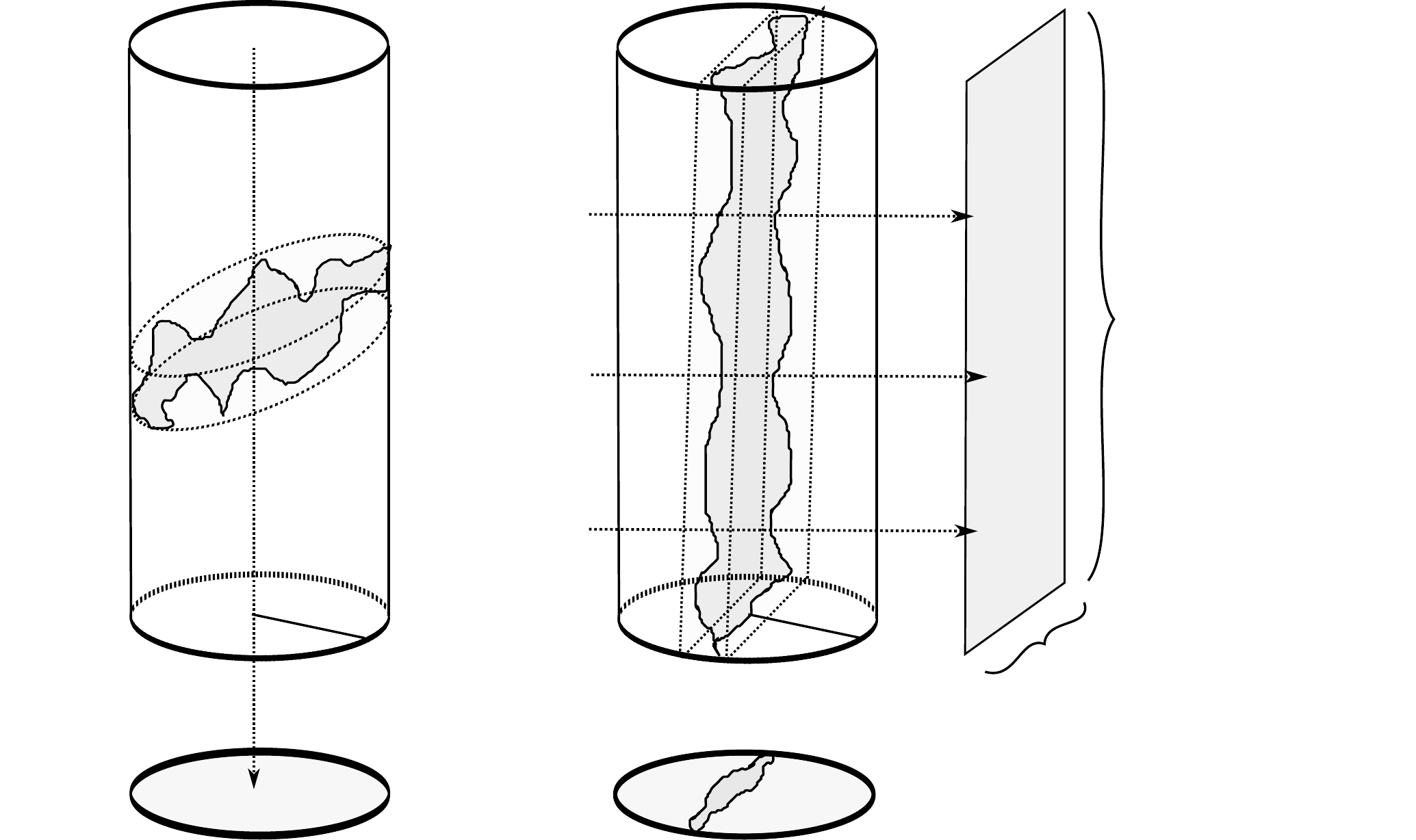}}%
    \put(-0.04,.025){\color[rgb]{0,0,0}\makebox(0,0)[lb]{\smash{$z=0$}}}%
    \put(0.75,.1){\color[rgb]{0,0,0}\makebox(0,0)[lb]{\smash{$r$}}}
    \put(0.8,.357){\color[rgb]{0,0,0}\makebox(0,0)[lb]{\smash{$\gamma r$}}}%
    \put(0.2,.095){\color[rgb]{0,0,0}\makebox(0,0)[lb]{\smash{$r$}}}%
  \end{picture}%
\endgroup%
\caption{Cylinders where the liquid boundary is flat and horizontal or flat and vertical respectively.}\label{fig: vertical horizontal}
 \end{figure}
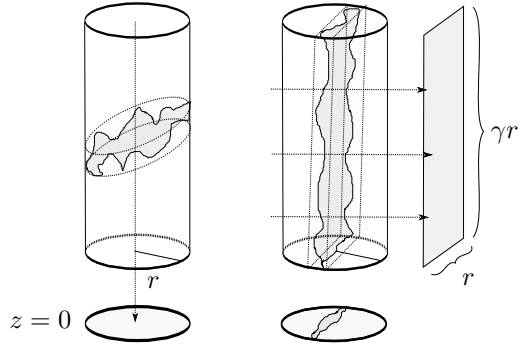
 
 \medskip
 
 \begin{lem}\label{lem: vertical horizontal}
 There exists a dimensional constant $c_0>0$ such that if $\delta \leq \frac{2}{3}\gamma^{-1}\wedge c_0$, then the following holds.
 \begin{itemize}
 \item[(a)] If $\partial L$ is vertical in $\Gamma_j$ then 
 \[ \Per( L , \Gamma_j) \geq c_0 \gamma r^{d}.\]
\item[(b)]$\partial L$ is horizontal in $\Gamma_j$ then $\partial L$ does not touch the top or bottom of $\Gamma_j$, i.e., 
\[  D_j \times \{z = z_j - \gamma r\} \subset L \ \hbox{ and } \  D_j \times \{z = z_j + \gamma r\} \subset L^C \ \hbox{ or vice versa.} \]
The same holds for the top and bottom caps of the cylinder $3\Gamma_j$.
\end{itemize}
 \end{lem}
 For both parts the proof is essentially contained in the picture Figure~\ref{fig: vertical horizontal}, we will write out the details below in the Appendix on page~\pageref{proof: vertical horizontal}.

\medskip

Given the above set-up we are now ready to apply energy arguments to show \eqref{eqn: main estimate flat}.  In order to elucidate the arguments we will make several separate energy arguments, putting the emphasis on one part of the energy comparison at a time.  We will use the normalized energy $E'$ (see \eqref{eqn: normalized energy}) to reduce the notation, recall that $L$ is a constrained minimizer for $E$ if and only if it is a constrained minimizer for the normalized energy $E'$.  

\medskip

   {\bf $\circ$ Perimeter bound for the vertical cylinders.}  First we make the perimeter bound for the cylinders $\{\Gamma_j\}_{j \in J_v}$ where $\partial L$ is vertical.  We argue based on an energy comparison with the modified minimizer,
  
  \begin{equation}\label{alternative}
  L_t = L \setminus \{ (x,z) : \ 0< z < 2t\} - 2t e_{d+1}.
   \end{equation}
We shift down by $2t$ just to be sure that all of that cylinders $\Gamma_j$ are completely contained in $0 < z < 2t$.  Now, since $L$ was a minimizer with volume $1$ and $|L_t| \geq |L| - C|U|t$, Lemma~\ref{lem: volume change} yields that  \begin{equation}\label{eqn: volume change Ft 1}
    {{E}}'(L) \leq {{E}}'(L_t)+C|U|t.
    \end{equation}
We now show a quantitative decrease of the energy in the vertical cylinders $\{ \Gamma_j\}_{j \in J_v}$ using the eccentricity of the cylinder. The idea is that, via Lemma~\ref{lem: vertical horizontal}, the removed perimeter is at least $\gamma r \omega_{d-1}r^{d-1}$ while the added perimeter at the $z=0$ level is at most $ \omega_{d} r^d$ where $r$ is the radius of the cylinder, see Figure~\ref{fig: vertical horizontal} for a depiction.  This is a strict decrease in the total energy when $\gamma$ is chosen large enough. More precisely
   
      \begin{equation}\label{eqn: energy comp vertical}
   \begin{array}{lll}
   {{E}}'(L) - {{E}}'(L_t) &\geq& \Per(L,\{0 < z< 2t\} \setminus \cup_{j \in J_v} \Gamma_j) +  \sum_{j \in J_v} \Per(L,\Gamma_j) \vspace{1.5mm}\\
   & & - \sum_{j \in J_v}|D_j| - \int_{\real^d \setminus \cup_{j \in J_v}D_j}|\varphi_L^+(x,2t) - \varphi_L^+(x,0)| \ dx \vspace{1.5mm}\\
   &\geq& \sum_{j \in J_v}\left( c_0\gamma r^{d} - (\omega_{d} +c_1(d))r^d\right) \vspace{1.5mm}\\
   &\geq& \sum_{j \in J_v} r^d \geq c\sum_{j \in J_v} \Per(L,\Gamma_j)
   \end{array}
   \end{equation}
   
   Here we have used in the first inequality that where $|\varphi_L^+(x,0) - \varphi_L^+(x,2t)| = 1$ we have $(x,s)\in \partial L$ for some $0<s<2t$, and then we use Lemma~\ref{lem: meas reg} . We also used  Lemma~\ref{lem: vertical horizontal} to bound from below $\Per(L,\Gamma_j)$ for $j \in J_v$. For the third inequality we used the choice of $\gamma = 4 \vee (1+\omega_d)/c_0 $, for the last inequality we used the interior measure theoretic regularity of $\partial L$ from Lemma~\ref{lem: meas reg}.  Now combining the result of \eqref{eqn: energy comp vertical} with \eqref{eqn: volume change Ft 1} we obtain \eqref{eqn: main estimate flat} for the vertical part of $\partial L$.

   \medskip

 {\bf $\circ$ Making the projections disjoint.} We will now show how to remove the overlapping of the $\{D_j\}_{j \in J_h}$.  
 The purpose of doing this is to simplify energy computations for horizintal cylinders.  By Vitali covering lemma we can find a subcollection $J_h' \subset J_h$ so that $\{D_j\}_{j \in J_h'}$ are mutually disjoint and $\cup_{J_h} D_j \subseteq \cup_{j \in J_h'} 3D_j$. We claim that,
 \begin{equation}\label{eqn: overlapping bound}
  \sum_{ J_h \setminus J_h'} \Per(L, \Gamma_j) \leq Ct.
  \end{equation}
  
   This is proved by a relatively simple energy computation. The idea is that any part of the area of $\partial L$ which is overlapping with another portion of $\partial L$ directly above or below it can be removed by shifting down without having to balance against any added perimeter. We will make use of the following fact, which is the reason for our set-up,
  \begin{equation}\label{eqn: h property}
  \sum_{j \in J_h'} \Per(L,3\Gamma_j) \geq  \sum_{j \in J_h'} \mathcal{H}^d( 3D_j) \geq  \mathcal{H}^d(\cup_{j \in J_h} D_j).
  \end{equation}
  The second inequality is immediate.  Since $\partial L$ is horizontal in $3 \Gamma_j$, from Lemma~\ref{lem: vertical horizontal}, for any $x\in D_j$ there is some $s$ such that $(x,s)\in 3\Gamma_j \cap \partial L$. Thus  $\Per(L,3\Gamma_j) \geq \mathcal{H}^d( 3D_j), $
 which yields \eqref{eqn: h property}. 
 
 \medskip
 
  Now we proceed with the energy argument to prove \eqref{eqn: overlapping bound}. As before we consider $L_t$ given by \eqref{alternative}.    On the other hand, using that the $\{3\Gamma_j\}_{j \in J_h}$ are disjoint, we have also decreased the energy by,
  
   \begin{equation}\label{eqn: energy comp overlap}
   \begin{array}{lll}
   {{E}}'(L) - {{E}}'(L_t) &\geq&   \Per(L,\{0 < z< 2t\})  - \int_{\real^{d}} |\varphi_L^+(x,0) - \varphi_L^+(x,2t)| \ dx \vspace{1.5mm}\\
   &\geq& \sum_{j \in J_h} \Per(L,3\Gamma_j) - \mathcal{H}^d(\cup_{j \in J_h} D_j) \vspace{1.5mm}\\
   &\geq& \sum_{j \in J_h \setminus J_h'} \Per(L,3\Gamma_j) +\sum_{j \in J_h'} \Per(L,3\Gamma_j)-\mathcal{H}^d(\cup_{j \in J_h} D_j) \vspace{1.5mm} \\
   &\geq& \sum_{j \in J_h \setminus J_h'} \Per(L,3\Gamma_j)
   \end{array}
   \end{equation}
   Here we have used in the second inequality that where $|\varphi_L^+(x,0) - \varphi_L^+(x,2t)| = 1$  we have $(x,s)\in \partial L$ for some $0<s<2t$ and thus $x\in D_j$ for some $j\in J$. For the last inequality we used \eqref{eqn: h property}.  Now combining the result of \eqref{eqn: energy comp overlap} with \eqref{eqn: volume change Ft 1} we obtain \eqref{eqn: overlapping bound}.

\medskip

  {\bf $\circ$ Perimeter bound for the horizontal cylinders.} Finally we consider the cylinders where $\partial L$ is horizontal, now having also reduced to a situation where the projections $\{D_j\}_{j \in J_h'}$ of the $\{ \Gamma_j\}_{j \in J_h'}$ onto the $z=0$ level are mutually disjoint.  This is the most involved part of the argument, we use the cell problem solutions $L_{\textup{SV}}, L_{\textup{SL}}$ in our energy comparison because it is convenient, actually using the Cassie-Baxter state \eqref{eqn: CB state} as we did in Proposition~\ref{prop: contact angle quantity non-degen} would also work.
  
  \medskip
  
  We make a further division of ${J_h'}$ depending whether $L$ is above $\partial L$ in $\Gamma_j$, heuristically this should generally be the case when the solid surface is hydrophobic, or $L$ is below $\partial L$ in $\Gamma_j$, which is natural with hydrophilic surface.  Let 
\[ J_{h,\textup{SL}}' := \{j \in {J_h'} : n_j \cdot e_{d+1} < 0 \} \ \hbox{ and } \ J_{h,\textup{SV}}' := \{j \in {J_h'} : n_j \cdot e_{d+1} > 0 \}.\]
For $j \in J_{h,\textup{SL}}'$, if we look in the extended cylinder $D_j \times \real$ moving from $z=-\infty$ upward we will see a macroscopic solid-vapor interface followed by the vapor-liquid interface in $\Gamma_j$.  When we shift the droplet down bringing $\Gamma_j$ down to the $\{z = 0 \}$ level we will replace these two interfaces by a single macroscopic solid-liquid interface. 
(possibly interposed microscopically by vapor via the solid-liquid cell problem solution).  
 
 \medskip
 
  We present only the argument for $J_{h,\textup{SL}}'$,  since symmetric arguments applies for $J_{h,\textup{SV}}'$.  We remove $L \cap \{0< z < 2t\}$ from the optimal droplet and shift down by $2t$ while also replacing, below each $\{D_j\}_{j \in J_{h,\textup{SL}}'}$, by the solid-liquid cell problem solution, 
 \begin{equation}\label{eqn: shift down}
 L_t = \left[L - 2t e_{d+1}\right]^+\cup \left[L_- \setminus  \bigcup_{j \in J_{h,\textup{SL}}'} D_j \times \real \right] \cup \left[\bigcup_{j \in J_{h,\textup{SL}}'} \e L_{\textup{SL}}(\tfrac{1}{\e}D_j) \cap (D_j\times(-\infty,0])\right] .
 \end{equation}
 Recall that $L_{\textup{SL}}(\frac{1}{\e}D_j)$ an optimizer for the solid-liquid cell problem in $\tfrac{1}{\e}D_j$ from Lemma~\ref{lem: maximal minimal minimizers}.  We remark that we have chosen to use $L_{\textup{SL}}(\frac{1}{\e}D_j)$ instead of the $\integer^d$-periodic cell problem solution to emphasize how the argument will generalize to the random setting.  Again from the monotonicity formula \eqref{eqn: volume change} for the volume change,
   \begin{equation}\label{eqn: volume change Et}
    {{E}}'(L) \leq {{E}}'(L_t)+Ct.
    \end{equation}
   Using that the $\{D_j\}_{j \in J_h'}$ are mutually disjoint we claim that the energy change is given by,   
   \begin{equation}\label{eqn: energy SVL}
   \begin{array}{lll}
   {{E}}'(L) - {{E}}'(L_t) &\geq& \Per(L,\{0 < z<2t\} \setminus \cup_{j \in J_{h,\textup{SL}}'} \Gamma_j) \vspace{1.5mm}\\
   &+ & \sum_{j \in J_{h,\textup{SL}}'} \Per(L,\Gamma_j) + \sum_{j \in J_{h,\textup{SL}}'} |D_j|\cos\Theta_Y(\tfrac{1}{\e}D_j)\vspace{1.5mm}\\
   & &   - M\e\sum_{j \in J_h}|\partial D_j|- \int_{\real^d \setminus \bigcup_{j \in J_{h,\textup{SL}}'}D_j}|\varphi_L^+(x,2t) - \varphi_L^+(x,0)| \ dx.
   \end{array}
   \end{equation}

    For now we take \eqref{eqn: energy SVL} for granted and complete the perimeter estimate.
  
  \medskip
  
  As we did in the previous segments of the argument we can bound,
  \[ \Per(L,\{0 < z<2t\} \setminus \cup_{j \in J_{h,\textup{SL}}'} \Gamma_j) - \int_{\real^d \setminus \cup_{j \in J_{h,\textup{SL}}'}D_j}|\varphi_L^+(x,2t) - \varphi_L^+(x,0)| \ dx \geq 0.\]
  Using Lemma~\ref{lem: vertical horizontal} and the definition of $J_{h,\textup{SL}}'$ we know that the bottom boundary caps of the cylinders $\{\Gamma_j\}_{j \in J_{h,\textup{SL}}'}$ are contained in $L^C$ and the top boundary caps are contained in $L$, and thus (by Lemma~\ref{lem: lines through})
  \[ \Per(L,\Gamma_j) \geq |D_j| \ \hbox{ for all } \ j \in  J_{h,\textup{SL}}'.\]
  On the other hand we know from Theorem \ref{thm: cell gen periodic}, taking the difference of the two estimates there, 
    \[|D_j|\cos\Theta_Y(\tfrac{1}{\e}D_j) \geq (\cos\overline{\theta}_Y-C(d)M\frac{\e}{r}\log\frac{r}{\e})|D_j|. \]
  
  Using the previous three inequalities and plugging into \eqref{eqn: energy SVL} we obtain, as long as $r \geq C\e$ for some constant $C$ depending on $M$ and $1-|\cos\overline{\theta}_Y|$,
     \begin{equation}\label{eqn: energy change SVL}
   \begin{array}{lll}
   {{E}}'(L) - {{E}}'(L_t) &\geq& \sum_{j \in J_{h,\textup{SL}}'} (1 - \cos\overline{\theta}_Y -C\tfrac{\e}{r}\log\tfrac{r}{\e})|D_j| \vspace{1.5mm}\\
   &\geq& c\sum_{j \in J_{h,\textup{SL}}'} |D_j| \geq c\sum_{j \in J_{h,\textup{SL}}'} \Per(L,\Gamma_j),
   \end{array}
   \end{equation}
 Lemma~\ref{lem: meas reg} is used for the final inequality.  This completes the proof of \eqref{eqn: main estimate flat}.

   \medskip

 It remains to carefully derive \eqref{eqn: energy SVL}.  We compute the energy change separately in each of the disjoint infinite cylinders $D_j \times \real$,
     \begin{equation}\label{eqn: energy change SVL1}
   \begin{array}{lll}
   {{E}}'(L) - {{E}}'(L_t) &=& {{E}}'(L, \{ z \leq 2t\}) - {{E}}'(L_t, \{z \leq 0\}) \vspace{1.5mm}\\
   &=& \bigg[{{E}}'(L, \{ z \leq 2t\} \setminus \bigcup_{j \in J_{h,\textup{SL}}'}D_j \times \real) -{{E}}'(L_t, \{ z \leq 0\} \setminus \bigcup_{j \in J_{h,\textup{SL}}'}D_j \times \real)\bigg] \vspace{1.5mm}\\
   & &  \quad + \sum_{j \in J_{h,\textup{SL}}'}\bigg[ {{E}}'(L, \{ z \leq 2t\}  \cap (D_j \times \real))- {{E}}'(L_t, \{ z \leq 0\}  \cap (D_j \times \real))\bigg] \vspace{1.5mm}\\
 &&\quad  - M\e\sum_{j \in J_h}|\partial D_j| \vspace{1.5mm}\\
   &=: &  [I] + \sum_{j \in J_{h,\textup{SL}}'} [II]_j - M\e\sum_{j \in J_h}|\partial D_j| 
   \end{array}
   \end{equation}
This gives us naturally two separate types of energy difference -- the computation of $[I]$ has come up before,
 \begin{equation}\label{eqn: energy change SVL2}
   \begin{array}{lll}
    [I]  & & \geq \Per(L, \{ z \leq2 t\} \setminus \bigcup_{j \in J_{h,\textup{SL}}'}D_j \times \real) - \int_{ \real^d \setminus \cup_{j \in J_{h,\textup{SL}}'} D_j} |\varphi_{L}^-(x,0)-\varphi_L^+(x,2t) | \ dx.
      \end{array}
   \end{equation}
   This leaves us to compute $\sum_j [II]_j$. Due to the definition of $J_{h,\textup{SL}}'$ and Lemma~\ref{lem: vertical horizontal}, for every $j \in J_{h,\textup{SL}}'$, the top cap of $\Gamma_j$,  which is $D_j \times \{ z = z_j + \frac{1}{2}\gamma r\}$, is contained in $L$. 
  Hence if $x \in D_j$ and $\varphi^+_{L^C}(x,2t) = 1$, then $(x,s)\in\partial L$ for some $s\in ( z_j + \frac{1}{2}\gamma r,2t)$. It follows that 
   \begin{equation}\label{eqn: top to 2t}
    \Per(L,D_j \times\{ z_j+\tfrac{1}{2}\gamma r \leq z \leq 2t\}) \geq \int_{D_j} \varphi_{L^C}^+(x,2t) \ dx. 
   \end{equation}
   
Now we have
   \begin{equation*}
   \begin{array}{lll}
    {{E}}'(L,D_j \times \{  z\leq 2t\}) &= & \Per(L,\Gamma_j) +\Per(L,D_j \times\{ z_j+\tfrac{1}{2}\gamma r \leq z \leq 2t\})+ {{E}}'(L, D_j \times\{  z \leq z_j-\tfrac{1}{2}\gamma r\})  \vspace{1.5mm} \\
     &\geq & \Per(L,\Gamma_j) +\int_{D_j} \varphi_{L^C}^+(x,2t) \ dx+ \e^d\Sigma_{\textup{SV}}'(\tfrac{1}{\e}D_j)
    \end{array}
    \end{equation*}
   where we have used \eqref{eqn: top to 2t} and Lemma~\ref{reduction}. Lastly, one can easily check that 
      \[ {{E}}'(L_t,D_j \times \{ z \leq 0\})  = \e^d\Sigma_{\textup{SL}}'(\tfrac{1}{\e}D_j) +  \int_{D_j} \varphi_{L^C}^+(x,2t) \ dx.\]
   
    After subtracting the previous two equations we get the desired result for $[II]_j$ to show \eqref{eqn: energy SVL}.

   \medskip

 \medskip

 {\bf $\circ$ Iteration argument to prove Proposition~\ref{prop: perim 0 t}.}  Now that we have proven Lemma~\ref{middle_lemma} we can pick up from \eqref{eqn: main iterate}:
 \begin{equation}\label{eqn: iteration per}
 \begin{array}{lll}
 \Per(L,W_k) \leq  C_1 t+C_2\lambda \Per(L,W_{k+1})   \ \hbox{ for any } \  \lambda 2^{-k}t \geq R_0\e ,
 \end{array}
 \end{equation}
 with $C_2$ universal and $C_1$ a universal constant times $|U| + |\partial U|$. This is the form of the bootstrapping, by using an inferior estimate for $\Per(L,W_{k})$ and then iterating with \eqref{eqn: iteration per} we will obtain an (almost) optimal estimate for $\Per(L,W_1)$. Suppose that $t \geq 8C_0 R_0\e$ and let
 \begin{equation}\label{eqn: k lambda}
  k := [ \tfrac{1}{2} \log_2\tfrac{t}{(2C_2)^{1/2} R_0\e}] \ \hbox{ and } \ \lambda := 2^k \tfrac{R_0\e}{t}.
  \end{equation}
  Note that we have  $k \geq 1$ and $\lambda \leq \tfrac{1}{2C_2}(\tfrac{R_0\e}{t})^{1/2}$.
   Applying the estimate \eqref{eqn: iteration per} repeatedly we obtain,
 \[ \Per(L,W_1) \leq C_1(1+C_2\lambda + \cdots (C_2\lambda)^{k-1})t+(C_2\lambda)^{k-1}\Per(L,W_{k}).\]
 Then using \eqref{eqn: iteration per} one last time to bound $\Per(L,W_{k})$, and using the simple bound (from \eqref{eqn: perim a-priori 2}),
 \[ \Per(L,W_{k+1}) \leq \Per(L, \{ z>0\}) \lesssim 1,\]
 yields,
 \[ \Per(L,W_1) \leq 2C_1t +C(C_2\lambda)^k \leq 2C_1t +C(\tfrac{R_0\e}{t})^{k/2} \leq 2C_1t+C\exp(-c(\log\tfrac{t}{R_0\e})^2) .\]
 Now supposing that, moreover, $t \geq R_0\e \exp(\frac{1}{c^{1/2}}(\log\frac{1}{\e})^{1/2})$ we obtain that,
 \[ \exp(-c(\log\tfrac{t}{R_0\e})^2) \leq \exp(-\log\tfrac{1}{\e}) \leq \e \leq t,\]
 so that we finally have the desired result,
 \[ \Per(L,W_1) \leq C (|U| + |\partial U|) t \ \hbox{ for any } \ t \geq R_0\e \exp(\frac{1}{c^{1/2}}(\log\frac{1}{\e})^{1/2}).\]
 This completes the proof of Proposition~\ref{prop: perim 0 t}.

\end{proof}

\subsection{Non-degeneracy of the large-scale contact angle} The goal of this section is to show, quantitatively, that the large-scale  contact angle of the volume constrained minimizer along its contact line is not too close to $0$ or $\pi$.

\medskip

Recall that  $Q_r$ refers to a $d+1$ dimensional cube of side length $r$, $\Box_r^t$ refers to a $d$ dimension cube of side length $r$ at height $t$,
\[Q_r(x,z) = (-r/2,r/2)^{d+1}+(x,z) \ \hbox{ and } \ \Box_r^t(x) = (-r/2,r/2)^{d} \times \{ z = 0\}+(x,t).\]
We will write $\Box_r$ to refer both to $\Box_r^0$ and to the corresponding subset of $\real^d$, and
\[\Box_r^{s,t}(x) := \Box_r \times \{s < z< t\}.\]
The dependence on the center point will be omitted when the center is not relevant.

  \begin{prop}[Contact angle non-degeneracy]\label{prop: non-degen} 
Let $Q_r^+ \subset \Omega$ be the upper half of a cube centered on the $z=0$ level.  There are universal $\delta_0>0$ and $C\geq 1$, depending only on $1-|\cos \theta_Y|$ and $\phi$, so that, 
\begin{equation}\label{eqn: non-degen SL}
 Q_r^+ \cap\{  z = \delta_0 r\} \subset L \ \hbox{ implies } \  \frac{|Q_{r/2}^+ \cap L^C|}{|Q_{r/2}^+|} \leq C \left(\frac{\e}{r}\right)^{\frac{2d}{d+2}} ,
 \end{equation}
 and for $r \leq \textup{Vol}^{\frac{1}{d+1}}$,
 \begin{equation}\label{eqn: non-degen SV}
 Q_r^+ \cap\{  z = \delta_0 r\} \subset L^C \ \hbox{ implies } \  \frac{|Q_{r/2}^+ \cap L|}{|Q_{r/2}^+|} \leq C \left(\frac{\e}{r}\right)^{\frac{2d}{d+2}} .
 \end{equation}
\end{prop}

  There is an asymmetry between the two parts of the result due to the volume constraint being on the liquid region $L$ and not on the vapor region $V$. We will only show the first part of Proposition~\ref{prop: non-degen} which says essentially that the macroscopic contact angle is bounded away from $0$.  This direction is slightly simpler since it does not need the requirement $r \leq \textup{Vol}^{-\frac{1}{d+1}}$. We will point out where this asymmetry comes up in the proof. 

\medskip

 Let us first state some simple consequences.  The first follows from Lemma~\ref{lem: meas reg} combined with Proposition~\ref{prop: non-degen}.
\begin{cor}\label{cor: uniform non-degen}
 Let $Q_r^+$ be the upper half of a cube centered on the $z=0$ level.  There are universal $\delta_0> 0$ and $C\geq1$ depending in particular on $1-|\cos \theta_Y|$ and $\phi$,  so that, calling $\beta = \frac{2d}{(d+1)(d+2)}$, 
 \[  Q_r^+ \cap\{  z = \delta_0 r\} \subset L\ \hbox{ implies } \ Q_{r/2}^+ \cap \left\{ \frac{z}{r} \geq C\left(\frac{\e}{r}\right)^{\beta}\right\} \subset L,\]
 and for $r \leq \textup{Vol}^{\frac{1}{d+1}}$,
 \begin{equation}\notag
 Q_r^+ \cap\{  z = \delta_0 r\} \subset L^C \ \hbox{ implies } \  Q_{r/2}^+ \cap \left\{ \frac{z}{r} \geq C\left(\frac{\e}{r}\right)^{\beta}\right\} \subset L^C.
 \end{equation}
\end{cor}

From the above Corollary in combination with Lemma~\ref{lem: meas reg}, we obtain \emph{near boundary density estimates} for the constrained minimizer $L$.  This will be precisely the result that we use in the proof of Theorem~\ref{thm: main periodic} below to measure the size of the boundary layer around $z = 0$ outside of which uniform convergence holds.

\begin{prop}[Near boundary density estimates]\label{prop: boundary density estimates}
 Let $x \in \real^d$ and $r \leq \textup{Vol}^{\frac{1}{d+1}}$.  There are $C \geq 1$ and $c>0$ universal so that, with $\beta := \frac{2d}{(d+1)(d+2)}$, 
 \begin{equation*}
\begin{array}{llll}
(i) & x+ t e_{d+1} \in L^{(1)} \cup \partial_eL & \hbox{ implies } & |L \cap Q_r^+(x)| \geq cr^{d+1}  \ \hbox{ for all } \ t \geq C \e^{\beta}r^{1-\beta}\vspace{1.5mm} \\
(ii) &  x+ t e_{d+1} \in L^{(0)} \cup \partial_eL & \hbox{ implies } & |L^C \cap Q_r^+(x)| \geq cr^{d+1}  \ \hbox{ for all } \ t \geq C \e^{\beta}r^{1-\beta} \vspace{1.5mm} \\ 
(iii) & x+ t e_{d+1} \in \partial_eL & \hbox{ implies } & \Per(L,Q_r^+(x)) \geq cr^{d}  \ \hbox{ for all } \ t \geq C \e^{\beta}r^{1-\beta}
\end{array} 
\end{equation*}
\end{prop}
The proof is straightforward and is provided in the Appendix.

\medskip

We remark that one can use the result of Proposition~\ref{prop: perim 0 t} to prove a version of Proposition~\ref{prop: non-degen}, but the estimate of the boundary layer achieved by that proof is strictly worse. In particular using Proposition~\ref{prop: perim 0 t} requires using an interior regularity result, which will only apply away from the solid surface, whereas we believe that a more precise argument needs to deal directly (as we do in our proof of Proposition~\ref{prop: non-degen}) with a situation where the liquid boundary is very close to the solid surface touching the solid in every $\e\integer^d$-periodicity cell.

\medskip

Now we return to proving Proposition~\ref{prop: non-degen}.  Let us begin with the outline of the proof, by which we hope to illustrate the challenges coming from the non-flat solid surface $S$.

\medskip

By scaling it suffices to prove the result when $r =1$ and $Q_r^+ = Q_r^+(0,0)$. For simplicity let us assume that  
 \begin{equation}\label{flat_top}
 |\{ \phi = 0\}| = \alpha> 0.
 \end{equation}
 This assumption is not required for the proof, for example assumption \eqref{eqn: S basic assumption} would be sufficient.  Assumptions \eqref{flat_top} and \eqref{eqn: S basic assumption} are both essentially designed to prove the non-degeneracy of the homogenized contact angle Proposition~\ref{prop: contact angle quantity non-degen}.  Proposition~\ref{prop: contact angle quantity non-degen} plays an important role in the current proof, but for technical reasons we cannot just quote it, we need to use the proof.  To see how to use the weaker assumption \eqref{eqn: S basic assumption} in the current proof one just needs to look carefully at the proof of Proposition~\ref{prop: contact angle quantity non-degen}.

\medskip

Suppose that $Q_1^+ \cap \{ z= \delta\} \subset L$.  We focus on the region $\Box_r^{0,\delta}$ with $1/2 \leq r <1$. The boundary of $\Box_r^{0,\delta}$ is naturally divided into three parts, the top, the sides and the bottom, we will refer to these by,
\begin{equation}
\hbox{the top is } \Box_r^\delta, \ \hbox{ the bottom is } \Box_r^0, \hbox{ and the sides we call } \  \partial _{\textup{side}}\Box_r^{0,\delta} := \partial\Box_r \times \{ 0<z < \delta\}.
\end{equation}

\medskip

$\circ$ {\bf Heuristics} The most basic idea of the proof is the same as Lemma 8 in \cite{CM}, which is based on a de Giorgi type iteration. Let us define
\[ V({r}) :=  L^C \cap \Box_{{r}}^{0,\delta} \ \hbox{ for } \ {r} \in (1/2,1),\]
then
\[\tfrac{d}{dr}|V(r)| = \mathcal{H}^d(L^C \cap \partial_{\textup{side}}\Box_r^{0,\delta}).\]
 The idea is to derive a differential inequality for $V$ of the form,
\begin{equation}\label{eqn: de giorgi ode}
  |V(r)| \leq (\tfrac{d}{dr}|V(r)|)^{1+1/d}\ \hbox{ which will imply } \ |V(\tfrac{1}{2})| = 0 \ \hbox{ for } \ |V(1)| <<1.
  \end{equation}
  
  Clearly, given the statement of Proposition~\ref{prop: non-degen}, we will not be able to prove that $|V(\frac{1}{2})| = 0$. 
The inequality we actually obtain in the proof is a modified version, which, in terms of ODEs, reads as 
\begin{equation}\label{discrete_continuum}
|V(r)| \leq C F(\frac{d}{dr}|V(r)|) \quad \hbox{ with } \ F(s) := [\e s^{1/2} + s^{1+1/d}].
\end{equation}

Now we explain how we will establish \eqref{discrete_continuum}. From isoperimetric inequality 
$$
 |V(r)| \leq C_d\mathcal{H}^d(\partial V(r))^{\frac{d+1}{d}}
 $$
where we can decompose the boundary of $V(r)$ into its component parts
\begin{equation}\label{eqn: vapor region decomp}
 \mathcal{H}^d(\partial V(r)) = \mathcal{H}^d(L^C \cap \dside \Box_r^{0,\delta})+ \Per(L^C, \Box_r^{0,\delta}) + \mathcal{H}^d(L^C \cap \Box_r^0).
 \end{equation}

\medskip

Thus we obtain \eqref{discrete_continuum} with $\e\to 0$ if we can bound the last two terms by the first term.  Since $\partial L^C \cap \Box_r^{0,\delta}$ contains the graph of a function over the set $L^C \cap \Box_r$, Lemma~\ref{lem: surface area to BV} yields 
\begin{equation}\label{eqn: ordering of areas by f}
\mathcal{H}^d(L^C \cap \Box_r^0) \leq \Per(L^C, \Box_r^{0,\delta}).
 \end{equation}

\medskip

Thus the main bulk of the proof is spent to obtain a bound similar to 
\begin{equation}\label{eqn: est ideal}
\Per(L^C, \Box_r^{0,\delta}) \leq C\mathcal{H}^d(L^C \cap \dside \Box_r^{0,\delta}) 
\end{equation}

To attempt to show \eqref{eqn: est ideal} we begin the same as in \cite{CM}, i.e. by considering the test set $L \cup \Box_r^{0,\delta}$ and computing the energy difference. In the flat boundary case studied in \cite{CM} this computation yields \eqref{eqn: est ideal} in a straightforward way, using again \eqref{eqn: ordering of areas by f} and that $|\cos \theta_Y| <1$.  The rough boundary case faces significant additional difficulties.  The essential problem is similar to what we faced in the proof of Proposition~\ref{prop: perim 0 t}.  In order to use the fact that $|\cos\overline{\theta}_Y| < 1$ one needs a certain amount of regularity at the $\e$-scale.

\medskip

To deal with this difficulty, we divide up the bottom boundary $\Box_{r}^0$ into $\e$-size squares and then we further divide those squares into {\it contact cells} where the liquid drop $L$ touches the $\{z=0\}$ level in at least a small portion of the cell, and {\it non-contact cells} where the liquid drop $L$ only touches the $\{z=0\}$ level in at most a very small portion of the cell.   We will show that \eqref{eqn: est ideal} can be still obtained in non-contact cells, meanwhile over the contact cells we will be able to use a Poincare-type inequality to bound the total volume of $L^C$. Arguing in this way we will naturally have to deal as well with the (unknown) boundary between the contact and non-contact regions.

\medskip

{\bf $\circ$ Setting of the problem.} In the proof we will work with the discretized version of \eqref{discrete_continuum}.  Define $r_k := \frac{1}{2} + 2^{-k}$ and $ V_k := | V(r_k)|$.  The discrete version of \eqref{discrete_continuum} then becomes 
 \begin{equation}\label{discrete}
  V_{k+1} \leq C(d,\alpha,\cos\theta_Y)F(2^kV_k).
   \end{equation}
Note that $V_k \to |V(\frac{1}{2})|$ as $k\to\infty$.  
 
 \medskip
Note that  by the area formula
\[
 V_k \geq \int_{r_{k+1}}^{r_k} \mathcal{H}^d( L^C \cap \dside \Box_r^{0,\delta})  \ dr,
 \]
so there is some $r_{k+1} \leq r_k^* \leq r_k$ so that
\begin{equation}\label{eqn: 2k vk}
 \mathcal{H}^d(L^C \cap \dside \Box_{r_k^*}^{0,\delta}) \leq 2^kV_k.
 \end{equation}
We define
 \[
 \Gamma_k:=\Box_{r_k^*}^{0,\delta}\hbox{ and } \Lambda:= L\cup \Gamma_k.
 \]  
 Our goal is to estimate the energy difference $E'(L) - E'(\Lambda)$ to achieve \eqref{discrete}.  
 
 \medskip
 
 In the proof below, we will often replace $L^C \cap \Gamma_k$ by a subgraph in $\Gamma_k$ with the same area and decreased perimeter as follows.  A full account of the following construction and resulting properties can be found in the book \cite[Lemma 14.7 and Theorem 14.8]{Giusti}.  For every $x \in \Box_{r_k^*}$ we define,
 \begin{equation}\label{eqn: w defn}
 w(x) := \int_0^\delta \indicator_{L^C}(x,t) \ dt.
 \end{equation}
 This function $w$ is in $BV(\Box_{r_k^*})$ and furthermore satisfies that, for any square $\Box \subseteq \Box_{r_k^*}$,
 \begin{equation}\label{eqn: w prop}
 \int_{\Box} w(x) \ dx = |L^C \cap \Box^{0,\delta}| \ \hbox{ and } \ \int_{\Box} \sqrt{1+|Dw|^2} \ dx \leq \Per(L^C,\Box^{0,\delta}).
 \end{equation}

\medskip
 
 Recall that $L$ is a volume constrained minimizer with volume $\textup{Vol}$. Since $|\Lambda|\geq |L|$, by the minimization property of $L$ and the monotonicity formula for comparing minima at different volumes Lemma~\ref{lem: volume change} we can check that 
\begin{equation}\label{energy}
{{E}}'(L) \leq {{E}}'(\Lambda).
\end{equation}
We remark that this is the place where the proof of \eqref{eqn: non-degen SV} is not exactly symmetric, since the set modification will decrease the volume in that case there would be an additional term $r\textup{Vol}^{-\frac{1}{d+1}}$ on the right hand side above coming from Lemma~\ref{lem: volume change}.  Under the assumption that $r\textup{Vol}^{-\frac{1}{d+1}} \leq 1$ this will not affect the rest of the proof significantly.

\medskip
Computing the change in energy under the perturbation,
\[
0\geq {{E}}'(L)-{{E}}'(\Lambda)\geq \Per(L, \Gamma_k) -(\cos\theta_Y)\mathcal{H}^d(L^C\cap\partial S \cap \Box_{r_k^*}) - \mathcal{H}^d (L^C \cap S^C\cap \Box_{r_k^*})  -\mathcal{H}^d(L^C \cap \dside \Gamma_k)
\]
and thus by  \eqref{eqn: 2k vk} we have

\begin{equation}\label{energy_diff}
  \Per(L,\Gamma_k) - (\cos\theta_Y)\mathcal{H}^d(L^C\cap\partial S \cap \Box_{r_k^*})- \mathcal{H}^d (L^C \cap S^C\cap \Box_{r_k^*}) \leq 2^kV_k.
  \end{equation}

This estimate give us different local information depending on how much $L$ touches $\{z = 0\}$ in each $\e$ sized square.  For this purpose we define two classes of squares, $\F_{C}$ and $\F_{NC}$ standing for the {\it contact} cells and the {\it non-contact} cells as follows. Since $r_k^* \geq r/2 \geq 2\e$ we can choose $R_k \in [1,2]$ so that $r_k^*/R_k \e$ is an integer. Let $w$ be given in \eqref{eqn: w defn},  $s:= \frac{1}{2}\alpha(1- \cos\theta_Y)$ and define
\begin{equation}\label{contact}
 \F_C:= \{\Box = \Box_{R_k\e}(j), j \in R_k\e\integer^d \cap \Box_{r_k^*}: |\{ w = 0\} \cap \Box | \geq s|\Box|\}
 \end{equation}
 and
 
\begin{equation}\label{noncontact}
\F_{NC}:=\{\Box = \Box_{R_k\e}(j), j \in R_k\e\integer^d \cap \Box_{r_k^*}: |\{ w = 0\} \cap \Box | < s|\Box|\}.
\end{equation}
   We define $\mathcal{F}_{C}^\delta$ (resp. $\mathcal{F}_{NC}^\delta$) to be the collection of $\Box^{0,\delta}$ with $\Box \in \F_C$ (resp. $\mathcal{F}_{NC}$).  We will abuse notation and use $\mathcal{F}_{C}$ et. al. to refer both to the collection of squares/cubes and to set which is the union of those squares/cubes.  
 
 \medskip

Now we proceed to measure $V_{k+1}$ by using the fact that 
\[
V_{k+1} \leq |L^C\cap \Gamma_k| = |L^C\cap \mathcal{F}_C^\delta| + |L^C\cap \mathcal{F}_{NC}^\delta |.
\]
The volume of $L^C$ over the contact cells will be measured first, where we replace the isoperimetric inequality by a Poincare-type inequality. Over the non-contact cells we bound the volume by the perimeter using the isoperimetric inequality.  The perimeter above the non-contact cells can be bounded using the energy estimate \eqref{energy_diff}, this is where we can really take advantage of $\cos\overline{\theta}_Y <1$.  Finally we will need to estimate the perimeter of $L$ along the boundary cells between $\mathcal{F}_C$ and $\mathcal{F}_{NC}$.  An additional geometric argument is needed for this.

\medskip

\medskip

{\bf $\circ$ Volume bound over the contact cells. }  We will make use of the following Poincare-type inequality, which can be found for example in the book \cite[Section 5.6.1 Theorem 1]{EvansGariepy}:
For every $ s\in (0,1)$ there exists $C(s,d)$ finite so that for any square $\Box \subset \real^d$ and any $f : \Box \to \real$ which is in $BV(\Box)$ and satisfies that $|\{ f = 0 \}| \geq s |\Box|$,
\begin{equation}\label{moser}
 \left( \int_\Box |f|^{\frac{d}{d-1}} \ dx \right)^{\frac{d-1}{d}}  \leq C\frac{1}{1-(1-s)^{1/d}} \int_\Box |Df| \ dx.
 \end{equation}
 The same result holds in $d=1$ with $\|f\|_{L^\infty(\Box)}$ on the left hand side.  We emphasize that these are not integral averages, the inequality is already scale invariant.

 \begin{lemma}[Estimates over the contact cells]\label{lem:contact}
\begin{equation}\label{est_contact_0}
|L^C\cap \F_C^\delta| \leq CR_k\e2^{k/2}V_k^{1/2}.
\end{equation}
\end{lemma} 

\begin{proof}
  We will make use of the function $w$ defined in \eqref{eqn: w defn} which rearranges $L^C \cap \Box_{r_k*}$ into a subgraph with the same volume and decreased perimeter via \eqref{eqn: w prop}.   From \eqref{eqn: w prop}, taking a sum over the squares constituting $\F_C$, we have,
\[
 |L^C \cap \F_C^\delta| = \int_{\F_C} w(x) \ dx.
 \]

Using H\"{o}lder inequality as well as the Poincare-type inequality \eqref{moser} in every $\e$ square $\Box$ of $\F_C$ using the definition of contact cells, we have for $d \geq 2$,
\[ \int_{\F_C} w \ dx = \sum_{\Box \in \F_C} \int_\Box w \ dx \leq  \sum_{\Box \in \F_C} |\Box|^{1/d}\left(\int_\Box w^{\frac{d}{d-1}}  \ dx\right)^{\frac{d-1}{d}}\leq  CR_k\e \int_{\F_C} |Dw|  \ dx.\]
When $d=1$ the computation is essentially the same,
\[ \int_{\F_C} w \ dx = \sum_{\Box \in \F_C} \int_\Box w \ dx \leq  \sum_{\Box \in \F_C} |\Box| \| w \|_{L^\infty(\Box)} \leq  CR_k\e \int_{\F_C} |Dw|  \ dx.\]
Note that 
\begin{align}
 \Per(L^C, \F^{\delta}_C) -\mathcal{H}^d(L^C\cap \F_C)&\geq \int_{\F_C \cap \{w >0\}} (1+|Dw(x)|^2)^{1/2} \  dx - |\{w >0\} \cap \mathcal{F}_C| \notag \\
 &=  \int_{\F_C \cap \{w >0\}} (1+|Dw(x)|^2)^{1/2}-1 \  dx \notag \\
 &\geq   a\int_{\mathcal{F}_C } |Dw(x)| \ dx -a^2|\mathcal{F}_C|. \label{eqn: contact est 2}
 \end{align}

for any $0 < a \leq 1$, using Lemma~\ref{lem: surface area to BV} for the last line. Using \eqref{energy_diff}, replacing $\cos \theta_Y$ by $1$, as well as the fact $\Per(L,\F_{NC}^\delta) - \mathcal{H}^d(\F_{NC} \cap L^C) \geq 0$ yields
\begin{equation}\label{eqn:est_contact_3}
a \int_{\F_C } |Dw| \ dx \leq C2^kV_k+a^2|\mathcal{F}_C| \leq C2^kV_k+a^2
 \end{equation}
since $\mathcal{H}^d(\mathcal{F}_C) \leq \mathcal{H}^d(\Box_{r_k^*}) \leq 1$.  Combining all the above we get
$$
 |L^C \cap \F_C^\delta| \leq CR_k\e \int_{\F_C} |Dw|  \ dx \leq C(Ca^{-1}2^kV_k+a)R_k\e \leq CR_k\e2^{k/2}V_k^{1/2},
$$
when we choose $a = 2^{k/2}V_k^{1/2}$, as long as that quantity is $\leq 1$. 
\end{proof}

{\bf $\circ$ Perimeter bound over the non-contact cells. }  Above the non-contact cells the perimeter has a better bound since we only see $\cos\theta_Y$ averaged over entire unit cells and so we can effectively argue as in the flat surface $|\cos\theta_Y| <1$ case.  
\begin{lemma}\label{contact}
There is $C \geq1$ depending on $1-\cos\theta_Y$ and $\alpha$ so that,
\begin{equation}\label{main1}
 \Per(L,\F_{NC}^\delta) \leq C2^kV_k.
\end{equation}
\end{lemma}
\begin{proof}

Here we use the energy difference more carefully. From  \eqref{energy_diff} we have 
\begin{align*} 
 2^kV_k  &\geq \Per(L, \mathcal{F}^{\delta}_{NC})+\Per(L, \mathcal{F}^{\delta}_{C}) - \mathcal{H}^d(\mathcal{F}_C\cap L^C) - ((1-\alpha)+\alpha \cos\theta_Y) \mathcal{H}^d(\mathcal{F}_{NC})  \\ 
&\geq  \Per(L, \mathcal{F}^{\delta}_{NC})- ((1-\alpha)+\alpha\cos\theta_Y)\mathcal{H}^d(\mathcal{F}_{NC}) 
\end{align*}

where the first inequality uses assumption \eqref{flat_top}, and the second inequality uses the fact $\Per(L,\F_C^\delta) - \mathcal{H}^d(\F_C \cap L^C) \geq 0$.  As mentioned above, with a more careful set modification in the mode of Proposition~\ref{prop: contact angle quantity non-degen} we could also use the more general condition \eqref{eqn: condition at the max} here. 

\medskip

For $\Box \in \F_{NC}$ we have,
\begin{equation}\label{est_NC_0}
 (1-s)|\Box|\leq |\Box\cap \{ w>0\}|\leq \int_{\Box \cap \{ w>0\}}  \sqrt{1+|Dw|^2}  \ dx \leq \Per(L, \Box^{0,\delta}).
 \end{equation}
As a result, using our choice of $s = \frac{1}{2}\alpha(1- \cos\theta_Y)$, we can control the change in energy above the non-contact set,
\begin{align*}
\Per(L, \F_{NC}^\delta) - (1-\alpha(1- \cos\theta_Y)) \mathcal{H}^d(\mathcal{F}_{NC}) &\geq \Per(L, \F_{NC}^\delta) - \frac{1}{1-s}(1-\alpha(1- \cos\theta_Y))\sum_{\Box \in \mathcal{F}_{NC}}\Per(L, \Box^{0,\delta}) \\
&\geq \frac{\alpha(1- \cos\theta_Y) -s}{1-s}\Per(L, \F_{NC}^\delta) \\
&=\frac{s}{1-s}\Per(L, \F_{NC}^\delta)
\end{align*}
Combining above estimates with \eqref{eqn: 2k vk} and \eqref{energy}, we have
\begin{equation}\label{perimeter}
 c \Per(L,\F_{NC}^\delta)
 \leq {{E}}(L)-{{E}}(F)+\mathcal{H}^d(L^C \cap \partial_{\textup{side}} \Gamma_k) \leq  0+2^kV_k,
 \end{equation}
 where $c =\frac{s}{1-s} = \frac{\frac{1}{2}\alpha(1- \cos\theta_Y)}{1-\frac{1}{2}\alpha(1- \cos\theta_Y)}>0$, and we have used \eqref{eqn: 2k vk} to conclude.
\end{proof}

\vspace{20pt}

Recall that to achieve \eqref{discrete} we need to bound $|L^C\cap \mathcal{F}_{NC}^\delta|$ using its perimeter. Note that
$$
\Per(L^C\cap \F_{NC}^\delta) = \Per(L,\F_{NC}^\delta) + \mathcal{H}^d(L^C\cap \F_{NC})+\mathcal{H}^d(L^C \cap \partial_{side} \Gamma_k) + \mathcal{H}^d(\partial\F_{NC}^\delta\cap L^C\cap \Gamma_k).
$$

For the first three terms on the right hand side,
$$
\mathcal{H}^d(L^C \cap \F_{NC})  \leq  \Per(L,\F_{NC}^\delta) \leq C 2^kV_k \ \hbox{ and } \mathcal{H}^d(L^C \cap \partial_{side} \Gamma_k) \leq 2^kV_k.$$

\medskip

{\bf $\circ$ Estimates along the microscopic contact line.} What is left is to bound the term $\mathcal{H}^d(\partial \F_{NC}^\delta\cap L^C\cap \Gamma_k)$.  It turns out that we can bound the surface area of $L^C$ along $\partial \F_{NC}^\delta\cap \Gamma_k$ in terms of the volume and perimeter of of $L^C$ above the boundary squares of $\F_{NC}$.  We refer to the collection of boundary squares,
\begin{equation}
\delta \F_{NC} = \{ \Box \in \F_{NC} : \ \Box \ \hbox{ is neighboring an element of $\mathcal{F}_{C}$}\}.
\end{equation}
We also define similarly $\delta \mathcal{F}_C$ to be the boundary squares of $\F_C$. In analogy to our previous definitions we call $\delta\mathcal{F}_{NC}^\delta$ to be the region above $\delta \F_{NC}$ in $0 < z< \delta$, more precisely it is the the interior of the union of the closures of $\Box^{0,\delta}$ for $\Box \in \delta \F_{NC}$.

\begin{lemma}\label{boundary}
\begin{equation}\label{main2}
\mathcal{H}^d(\partial \F_{NC}^\delta \cap L^C \cap  \Gamma_k) \leq 2d\left[\frac{1}{\e}|L^C \cap \delta\F_{NC}^\delta|+ \Per(L^C,\delta\mathcal{F}_{NC}^\delta)\right],
\end{equation}
\end{lemma}
We will return to the proof of Lemma~\ref{boundary} later, first we explain that $|L^C \cap \delta\F_{NC}^\delta|$ can be bounded in a similar way to $|L^C \cap \F_C|$ with even a stronger result.  
\begin{lemma}\label{boundary_cell}
\begin{equation}\label{main3}
|L^C\cap \delta\F_{NC}^{\delta}| \leq C\e2^kV_k.
\end{equation}
\end{lemma}

\begin{proof}
Each square $\Box \in \delta\F_{NC}$ has a neighbor $\blacksquare \in \mathcal{F}_C$ and $\Box \subseteq 3\blacksquare$.  Since $\blacksquare$ is a contact square it holds
$$|\{w = 0 \} \cap 3 \blacksquare| \geq s|\blacksquare| \geq 3^{-d}s|3\blacksquare|.$$
We remark that this condition means that $3 \blacksquare$ is a contact square for the parameter $s \mapsto 3^{-d}s$. Each $\blacksquare \in \delta \F_C$ serves as the neighbor for at most $2d$ squares $\Box \in \delta\F_{NC}$.  Proceeding as in the proof of Lemma~\ref{contact}, and using \eqref{eqn:est_contact_3}, we have

$$
\begin{array}{lll}
|L^C\cap\delta\F_{NC}^{\delta}| = \int_{\delta\F_{NC}} w \ dx & \leq & 2d\sum_{\blacksquare \in \delta \F_{C}} \int_{3\blacksquare} w \ dx \\ \\
& \leq & \sum_{\blacksquare \in \delta \F_{C}}C(s,d)\e\int_{3\blacksquare} |Dw| \ dx \\ \\
   &\leq & C\e (2^kV_k + \mathcal{H}^d(\delta\F_{C}))\\ \\
   &\leq & C\e(2^kV_k + 2 d \mathcal{H}^d(\F_{NC})) \\ \\
   &\leq & C\e2^kV_k,
\end{array}
$$
where for the last inequality is we observe, due to \eqref{main1}, that 
\[\mathcal{H}^d(\delta\mathcal{F}_{NC}) \leq \mathcal{H}^d(\mathcal{F}_{NC}) \leq (1-s)^{-1} H^d(\{w>0\}\cap \mathcal{F}_{NC}) \leq (1-s)^{-1}Per(L^C\cap \mathcal{F}_{NC}^{\delta})  \leq C2^kV_k.\]
\end{proof}

 \medskip

Now we return to the proof of Lemma~\ref{boundary}:
\begin{proof}[Proof of Lemma~\ref{boundary}]

  We begin with some definitions.  For a square $\Box$ of $\real^d$ let $\partial^e \Box$ be the face of $\Box$ with inward normal $e$ ranging over the $2d$ lattice directions.  Note that $\partial \F_{NC}^\delta$ is a union of $\partial^e \Box^{0,\delta} :=\partial^e \Box \times (0,\delta)$ where $\Box \in \F_{NC}$ and $e$ is some lattice direction.  Each cylinder of $\F_{NC}^\delta$ has at most $2d$ faces of $\partial \F_{NC}^\delta$ which are directly neighboring it.

  \medskip
  
  Now we aim to prove the following.  Fix a single face $\partial^e \Box^{0,\delta}$ of $\partial \F_{NC}^\delta$ with $\Box \in \delta\F_{NC}$ we aim to bound,
  \begin{equation}\label{eqn: single bdry cube bound}
  \mathcal{H}^d(L^C \cap \partial^e \Box^{0,\delta}) \leq \frac{1}{\e}|L^C \cap \Box^{0,\delta}|+\Per(L^C,\Box^{0,\delta}).
  \end{equation}
  Then summing over the faces $\partial^e \Box^{0,\delta}$ making up $\partial \F_{NC}^\delta$ and using that each $\Box \in \delta\F_{NC}$ is associated with at most $2d$ faces, 
    \begin{equation*}
  \mathcal{H}^d(L^C \cap \partial \F_{NC}^\delta \cap \Gamma_k) \leq \frac{2d}{\e}|L^C \cap \delta\F_{NC}^\delta|+2d\Per(L^C,\delta\F_{NC}^\delta).
  \end{equation*}
  This would complete the proof, thus we are left to show \eqref{eqn: single bdry cube bound}.

  \medskip
  
  For the rest of the proof we will fix a single face $\partial^e \Box^{0,\delta}$ of $\partial \F_{NC}^\delta$ with $\Box \in \delta\F_{NC}$.  By using translation and rotation symmetry there is no loss in supposing that $e = e_1$ and 
  \[\partial^e\Box^{0,\delta} = \{ 0 \} \times (-R_k\e/2,R_k\e/2)^{d-1}\times (0,\delta),\]
  which we naturally identify with a subset of $\real^d$ which we call by the same name.  
  
  \medskip
  
  Now, similar to what we did in \eqref{eqn: w defn}, we want to rearrange $L^C \cap \Box^{0,\delta}$ into a subgraph, but now we would like the graph to be over $\partial_{e_1}\Box^{0,\delta}$ instead of over $\{ z = 0\}$.  We define,
    \begin{equation}\label{eqn: g defn}
 g(x',z) := \int_0^{R_k\e} \varphi_{L^C}(t,x',z) \ dt \ \hbox{ defined for } \ (x',z) \in \partial^e\Box^{0,\delta}.
 \end{equation}
 Similar with the properties for $w$ \eqref{eqn: w prop} we have
 \begin{equation}\label{eqn: g prop}
 \int_{ \partial^e\Box^{0,\delta} } g(x',z) \ dx'dz = |L^C \cap \Box^{0,\delta}| \ \hbox{ and } \ \int_{\partial^e\Box^{0,\delta} \cap \{ g <R_k\e\}} \sqrt{1+|Dg|^2} \ dx'dz \leq \Per(L^C,\Box^{0,\delta}),
 \end{equation}
and finally also $\mathcal{H}^d(\{g >0\}) \geq \mathcal{H}^d(L^C \cap \partial^e\Box^{0,\delta})$. 
\medskip

Now we estimate the places where $0<g<R_k\e$ and $g = R_k\e$ separately.  When $0<g <R_k\e$ we can estimate using the second part of \eqref{eqn: g prop},
\[\mathcal{H}^d(\{0 < g < R_k\e\} \cap \partial^e\Box^{0,\delta} ) \leq \int_{\partial^e\Box^{0,\delta} \cap \{ g <R_k\e\}} \sqrt{1+|Dg|^2} \ dx'dz \leq \Per(L^C,\Box^{0,\delta}).\]
For the set $g = R_k\e$ we use the first part of \eqref{eqn: g prop} to get,
\[\mathcal{H}^d(\{g= R_k\e\} \cap \partial^e\Box^{0,\delta} ) = \int_{\{g= R_k\e\} \cap \partial^e\Box^{0,\delta} } \frac{1}{R_k\e}g(x',z)dx'dz \leq \frac{1}{R_k\e} |L^C \cap \Box^{0,\delta}|.\]
Summing the two estimates together we get,
\[ \mathcal{H}^d(L^C \cap \partial^e\Box^{0,\delta}) \leq \mathcal{H}^d( \{g >0\} \cap \partial^e\Box^{0,\delta} ) \leq \frac{1}{R_k\e} |L^C \cap \Box^{0,\delta}|+\Per(L^C,\Box^{0,\delta}).\]
This completes the proof of \eqref{eqn: single bdry cube bound}.
\end{proof}

\bigskip

{\bf $\circ$ DeGiorgi-type iteration.} Now we collect the estimates we obtained above to show \eqref{discrete}. 
First  \eqref{est_contact_0} would get us
\[
V_{k+1} \leq |L^C\cap \Gamma_k| \leq C\e2^{k/2}V_k^{1/2} + |L^C\cap \mathcal{F}_{NC}^\delta|.
\]

Combining \eqref{main1}, \eqref{main2} and \eqref{main3}  yields \eqref{discrete}.
In particular for all $k$ such that 
\[\e \leq (C2^kV_k)^{\frac{1}{2}+\frac{1}{d}}\]
 it holds that,
\[V_{k+1} \leq (C(d,\alpha,\cos\theta_Y)2^kV_k)^{1+\frac{1}{d}}.\]
It is well known that for iterations as above, when $V_0 \leq C(d)\delta_0$ is sufficiently small depending on $C(d,\alpha,\cos\theta_Y)$, $V_k \to 0$ as $k \to \infty$.  Thus for $k$ sufficiently large it must hold that $(C2^kV_k)^{\frac{1}{2}+\frac{1}{d}} \leq \e$, this is the desired result.  Actually we can show something slightly stronger than is claimed by the statement of the Proposition since we claim that under this iteration $2^kV_k \leq C\e^{\frac{2d}{d+2}}$ for $k \geq C\log(\log\frac{1}{\e}+C)$ so that $2^{-k}\leq c(\log\frac{1}{\e}+C)^{-1}$ as $\e \to 0$.  This improves the statement by a logarithmic factor.




\section{Homogenization and the size of the boundary layer}\label{sec: hom}

In this final section we are able to combine the results of the previous sections to obtain the main quantitative homogenization result Theorem~\ref{thm: main}. 

\medskip

Fix a bounded domain $U \subset \real^d$ with smooth or piecewise smooth boundary and call $\Omega = \overline{U} \times \real$.  First let us recall that the homogenized energy density is given by, for $S_0= \{z\leq 0\}$,
\[ d\overline{E}(L) = \sigma_{\textup{LV}}\left.d\mathcal{H}^{d}\right|_{\partial L \cap \partial V}+\overline{\sigma}_{\textup{SL}}\left.d\mathcal{H}^{d}\right|_{\partial L \cap \partial S_0} + \overline{\sigma}_{\textup{SV}}\left.d\mathcal{H}^{d}\right|_{\partial V \cap \partial S_0},  \]
and the associated volume constrained minimization problem,
\begin{equation}\label{eqn: L0 min}
L_0  = \argmin \bigg\{ \overline{{E}}(\Lambda,\Omega): \Lambda \in \mathcal{A}_0(\textup{Vol},\Omega) \bigg\} 
\end{equation}
\[
\hbox{ in the admissible class } \ \mathcal{A}_0(\textup{Vol},\Omega) = \bigg \{ \Lambda \in BV(\Omega \setminus S_0) \ \hbox{ with } \ |\Lambda| = \textup{Vol} \bigg\}
\]
From the result of Gonzalez \cite{Gonzalez}, if $U\times \real$ is large enough to contain a ball of radius $\rho_0$, $L_0$ is a spherical cap $B_{\rho_0}^+(x,z_0)$ with the parameters $\rho_0$ and $z_0$ fixed by the volume constraint and the contact angle,
\begin{equation}\label{eqn: rho z}
z_0 = \rho_0 \cos \bar{\theta}_Y \ \hbox{ and } \ |B^+_{\rho_0}(x,z_0)| = \textup{Vol}.
 \end{equation}
From now on assume that $U$ satisfies said property.  The $L^1$-stability of the global minimizer $L_0$, which we will use below, was proven in \cite{CM},
\begin{thm}[Caffarelli, Mellet \cite{CM}]\label{stability} 
There are constants  $0<\alpha(d) \leq 1/2$ and $C(d,1-|\cos\overline{\theta}_Y|)>0$ so that for any $\Lambda \in BV(\real^{d+1} \setminus S_0)$ with $|\Lambda| = \textup{Vol}$,
 \begin{equation}\label{rate}
 \min_{x \in \real^d} \frac{|\Lambda \Delta (L_0+x) |}{ \textup{Vol}} \leq C\left(\frac{\overline{{E}}(\Lambda) - \overline{E}(L_0)}{\sigma_{\textup{LV}}\textup{Vol}^{\frac{d}{d+1}}}\right)^\alpha.
  \end{equation}
\end{thm}

 Before we state the main Theorem of the section let us first recall the length scale from Proposition~\ref{prop: perim 0 t},
 \begin{equation}\label{eqn: r length scale}
  r_0(\e):=R_0\e\exp(C_1 |\log (\textup{Vol}^{-\frac{1}{d+1}}\e)|^{1/2}),
  \end{equation}
and in order to avoid complicated formulas involving $\textup{Vol}$, $|U|$ and $|\partial U|$ in the statement of the Theorem we define the error rate, which is a non-dimensional quantity,
\[ \textup{err}(\e) := \left(\textup{Vol}^{-1}|U| + \textup{Vol}^{-\frac{d}{d+1}}|\partial U|\right) r_0(\e).\]
\begin{thm}\label{thm: main periodic}
Let $\rho_0, z_0$ as above in \eqref{eqn: rho z} so that $L_0=B^+_{\rho_0}(x, z_0)$ minimizes $\overline{E}$ with volume constraint.  Let $L_\e$ with $|L_\e| = \textup{Vol}$ be the volume constrained minimizer associated with the energy $E_\e$ 
\begin{enumerate}[(i)]
\item We have a rate of convergence in $L^1$,
\begin{equation}
\min_{x \in \real^d} \frac{|L_\e \Delta L_0|}{\textup{Vol}} \lesssim \textup{err}(\e)^\alpha,
\end{equation}
where $\alpha$ is the exponent given in \eqref{rate} above.
\item Call $\beta = \frac{2d}{(d+1)(d+2)}$ and $h_0(\e) = C\textup{Vol}^{\frac{1-\beta}{d+1}}\e^{\beta}\textup{err}(\e)^{\frac{(1-\beta)\alpha}{d+1}}$ then,
 \[ \min_{x \in \real^d} \tfrac{1}{\textup{Vol}^{\frac{1}{d+1}}}d_H\big(L_\e \cap \{ z \geq h_0(\e)\},  L_0 \cap \{ z \geq h_0(\e)\} \big) \lesssim\textup{err}(\e)^{\frac{\alpha}{d+1}}.   \] 
 In words, $L_\e$ converges {uniformly}, outside of a boundary layer of size $h_0(\e)$, to some global minimizer of the homogenized problem. In $d+1=2$ this is $h_0(\e) \sim \e^{(1+\alpha)/3}$ and in $d+1 = 3$ this is $h_0(\e) \sim \e^{(3+2\alpha)/9 - o(1)}$.
 \end{enumerate}
\end{thm}

\begin{remark}
The terms $\textup{Vol}^{-1}|U|r_0(\e)$ and $\textup{Vol}^{-\frac{d}{d+1}}|\partial U|r_0(\e)$ in the error rate $\textup{err}(\e)$ correspond to parts of the error which do not appear in the flat surface case. There is one additional term, $\textup{Vol}^{-\frac{1}{d+1}}r_0(\e)$, that does appear in the flat surface case but it is hidden here inside the first error term $\textup{Vol}^{-1}|U|r_0(\e)$.

\medskip

The term $\textup{Vol}^{-1}|U|r_0(\e)$ comes from the part of the volume of the liquid droplet which fills in the rough surface below the level $z=0$, this effectively changes the volume of the droplet and contributes an error.  The term $\textup{Vol}^{-\frac{d}{d+1}}|\partial U|r_0(\e)$ comes from the boundary layer along $\partial U$ where the periodic cell problem solution is no longer necessarily optimal, this contributes an error of order $\sim \sigma_{\textup{LV}}\e|\partial U|$ in the estimate of the energies.  The last (hidden) term $\textup{Vol}^{-\frac{1}{d+1}}r_0(\e)$ is the ``true" homogenization error which is coming from the $d-1$ dimensional contact line where again the cell problems no longer give good estimates for the energy per unit cell.

\end{remark}

 Let us give an outline of the arguments in the proof of Theorem~\ref{thm: main periodic}.  Using the regularity theory we have developed in the previous sections we will prove that given $L$ which is the volume constrained global minimizer of $E$ over a domain $\Omega$ with solid surface $S_\e$ we can find a finite perimeter set $\Lambda \subset \real^{d+1}_+$ such that,
\[ |L \Delta \Lambda | \lesssim \textup{err}(\e) \ \hbox{ and } \ |{{E}}_\e(L) - \overline{{E}}(\Lambda)| \lesssim  \textup{err}(\e).\]
 Correspondingly, given the spherical cap $L_0$ which is the global minimizer of $\overline{{E}}$ with solid surface $S = \real^{d+1}_-$, we can modify $L_0$ to $L_{0,\e}$ by unioning on the solid-liquid and solid-vapor cell problem solutions in the respective solid-liquid and solid-vapor contact regions so that,
\[ |(L_0 \Delta L_{0,\e}) \cap \{z \geq 0\}| =0 \ \hbox{ and } \ |\overline{{E}} (L_0) - {{{E}}}_\e(L_{0,\e})| \lesssim \e.\]
So we have
\[
\overline{{E}}(L_0) \leq \overline{{E}}(\Lambda)\leq {{E}}_\e(L) + C\textup{err}(\e) \leq {{E}}_\e(L_{0,\e})+C\textup{err}(\e) \leq \overline{{E}}(L_0) +C\textup{err}(\e),
 \]
from which we conclude that $|\bar{E}(L_0)-\bar{E}(\Lambda)| \leq |\overline{{E}}(L_0) - {{E}}(L)| \lesssim \textup{err}(\e)$.  The  $L^1$ convergence of $L$ to $L_0$ then follows from the stability estimate in \cite{CM} for the homogenized energy (Theorem~\ref{stability}). Convergence in Hausdorff distance is a consequence of the $L^1$ convergence combined with the nondegeneracy estimate  (Proposition~\ref{prop: boundary density estimates}).

\subsection{The proof of homogenization theorem}

As outlined above, the proof of Theorem~\ref{thm: main periodic} consists of three, essentially independent, main ingredients.  The first part is where we use the correctors and the Hausdorff estimate of the contact line (Proposition~\ref{prop: non-degen}) to establish the convergence rate of the energies. The second is the stability result for the homogenized energy given in Theorem~\ref{stability}.  The third part of the argument uses the contact angle non-degeneracy result Proposition~\ref{prop: boundary density estimates}
 to upgrade convergence in measure to convergence in Hausdorff distance outside of a certain boundary layer.

\medskip

Now we state more precisely the first and third part.
 
 \medskip

  The following result does \emph{not} depend on setting of the homogenized problem in the upper half space and, at least at a conceptual level, could be extended to surfaces with $\e$-scale roughness which are smoothly varying (non-flat) at scale $1$. 
  
   \begin{prop}[Convergence of the energy]\label{prop: energy convergence}
Let $U \subset \real$ a fixed open bounded region of $\real^d$ with smooth or piecewise smooth boundary and containing some ball of radius $\rho_0$.  Call $\Omega= \overline{U} \times \real$.   Let $L_\e \subset \Omega \setminus S_\e$ be the volume constrained global minimizer of the energy ${{E}}_\e(\cdot,\Omega)$ and $L_0 \subset \Omega \cap \{ z >0\}$ be the volume constrained global minimizer of the energy $\overline{E}(\cdot,\Omega)$.  There exists $\Lambda_\e \subset \Omega \cap\{z>0\}$ with,
\[ |L_\e \Delta \Lambda_\e| \leq C|U|r_0(\e) \ \hbox{ and } \ |\overline{E}(\Lambda_\e,\Omega) - \overline{E}(L_0,\Omega)| \leq C\sigma_{\textup{LV}}\left(\textup{Vol}^{-\frac{1}{d+1}}|U| + |\partial U|\right)r_0(\e). \]
Recall that $r_0(\e)$ here is the length scale from Proposition~\ref{prop: perim 0 t}, see \eqref{eqn: r length scale}.
\end{prop}

\medskip

For the last part of Theorem~\ref{thm: main periodic}, we give a result which upgrades $L^1$ estimates to uniform estimates for sets satisfying the boundary density estimates Proposition~\ref{prop: boundary density estimates}:

\begin{lem}[$L^1$ to uniform]\label{lem: uniform convergence}
Let $L_1$ and $L_2$ be any two sets of $\Omega \setminus S_\e$ which both satisfy the results of Proposition~\ref{prop: boundary density estimates}.  Then for  $\beta := \frac{2d}{(d+1)(d+2)}$,
\[ d_H\bigg( L_1 \cap \{ z \geq h\}, L_2 \cap \{ z \geq h\}\bigg) \leq C |L_1 \Delta L_2|^{\frac{1}{d+1}} \ \hbox{ for any } \ h \geq C\e^\beta|L_1 \Delta L_2|^{\frac{1-\beta}{d+1}}.\] 
\end{lem}

 Before proceeding with the proofs of Proposition~\ref{prop: energy convergence} and Lemma~\ref{lem: uniform convergence} we explain how to derive Theorem~\ref{thm: main periodic} from the above three results.
 \begin{proof}[Proof of Theorem~\ref{thm: main periodic}] Take $\Lambda_\e$ from Proposition~\ref{prop: energy convergence} which has,
 \begin{equation}\label{eqn: lambd aprop}
  |L_\e \Delta \Lambda_\e| \leq C|U|r_0(\e) \ \hbox{ and } \ |\overline{E}(\Lambda_\e,\Omega) - \overline{E}(L_0,\Omega)| \leq C\sigma_{\textup{LV}}\left(\textup{Vol}^{-\frac{1}{d+1}}|U| + |\partial U|\right)r_0(\e). 
  \end{equation}
 Applying the $L^1$-stability estimate (Theorem~\ref{stability}) with volume $|\Lambda_\e|$ and then using that,
 \[ |L_0(\textup{Vol}) \Delta L_0(\textup{Vol}\pm \delta)| \leq C\delta,\]
 we obtain,
\begin{align*}\notag
 \min_{x \in \real^d} \frac{|\Lambda_\e \Delta (L_0+x) |}{ \textup{Vol}} &\leq C\left(\frac{\overline{{E}}(\Lambda_\e) - \overline{E}(L_0)}{\sigma_{\textup{LV}}\textup{Vol}^{\frac{d}{d+1}}}\right)^\alpha +C\textup{Vol}^{-1}|U|r_0(\e) \\
 &\leq C\left(\textup{Vol}^{-1}|U| + \textup{Vol}^{-\frac{d}{d+1}}|\partial U|\right)^\alpha r_0(\e)^\alpha
 \end{align*}
 using for the last inequality \eqref{eqn: lambd aprop} and that $\e \leq 1$.  Using $|L_\e \Delta \Lambda_\e| \leq C|U|r_0(\e)$ gets the same estimate (up to constants) for  $\min_{x \in \real^d} |L_\e \Delta (L_0+x) |$ and that is the result of part $(i)$.
 
 \medskip
 
 Now to achieve part $(ii)$ we make use of the boundary non-degeneracy in the form of Lemma~\ref{lem: uniform convergence}.  Both $L_\e$ and $L_0$ satisfy the interior density estimates Lemma~\ref{lem: meas reg} and the boundary density estimates in the form of Proposition~\ref{prop: non-degen} so Lemma~\ref{lem: uniform convergence} applies and gives the result,
\[ \min_{x \in \real^d} d_H\bigg( L_\e \cap \{ z \geq h\}, (L_0+x) \cap \{ z \geq h\}\bigg) \leq A r_0(\e)^{\frac{\alpha}{d+1}}
\ \hbox{ for any } \ h \geq A^{1-\beta}\e^{\beta}r_0(\e)^{\frac{\alpha(1-\beta)}{d+1}},\] 
 with $\beta = \frac{2d}{(d+1)(d+2)}$ and $A = C\textup{Vol}^{\frac{1}{d+1}}\left(\textup{Vol}^{-1}|U| + \textup{Vol}^{-\frac{d}{d+1}}|\partial U|\right)^{\frac{\alpha}{d+1}}$ with a universal $C$.
 
 \end{proof}

\medskip

\begin{proof}[Proof of Proposition~\ref{prop: energy convergence}]
Let $R$ a positive integer to be chosen large depending on $\e$ and call $r = R\e$ to be the intermediate length scale $\e \ll r \ll 1$.  Partition $\real^d$ by squares $\Box_{k} = rk + [0,r)^d$ of side length $r$ centered at the points of $r\integer^d$.  Corresponding to each square $\Box_k$ is a $d+1$-dimensional cube $Q_k$ with side lengths $r$ centered at the same point $k \in r \integer^d$. Call $3\Box_{k}$ to be the square of side length $3r$ centered at $r k$.

\medskip

Since we are computing the energy only in the region $\Omega = \overline{U} \times \real$ we can restrict to the squares $\Box_k$ with $\Box_k \cap U \neq \emptyset$.  First we remove from consideration the {\it boundary squares} which are located too near to $\partial U$, we call $J_{B}$ to be the collection of $k$ so that $3\Box_k \cap \partial  U \neq \emptyset$. Since $\partial U$ is at least piecewise smooth,
\[ \#(J_{B}) \leq C|\partial U|r^{1-d} . \]
Then for the remaining squares we divide based on whether they are in the wetted set, in the complement of the wetted set, or near the contact line. Based on Proposition~\ref{prop: perim 0 t} we will be able to guarantee that there are not too many squares near the contact line.  

\medskip

 We say that $\Box_k$ is \emph{wetted} if $2\Box_k^{r} \subset L_\e$, we say that $\Box_k$ is \emph{non-wetted} if $2\Box_k^{r} \subset L^C$ and otherwise we say that $\Box_k$ is \emph{near the contact line}.  Recall that,
 \[ \Box_k^r = \Box_k \times \{ z = r\}.\]
 Call $J_{\textup{SL}}$ to be the $k \in \integer^d$ so that $\Box_k$ is wetted, $J_{\textup{SV}}$ to be the $k \in \integer^d$ so that $\Box_k$ is non-wetted, and call $J_{\textup{SLV}}$ to be the $k$ so that $\Box_k$ is near the contact line.  We will refer to the contact sets (discretized at scale $r$) as,
 \[ K_{\textup{SL}} = \cup_{k \in J_{\textup{SL}}} \Box_k \ \hbox{ and } \ K_{\textup{SV}} = \cup_{k \in J_{\textup{SV}}} \Box_k.\]
Every boundary face of $K_{\textup{SL}}$ or $K_{\textup{SV}}$ must border a box of $J_{\textup{SLV}}$ or $J_B$ and each such square may only have at most $2d$ neighbors therefore 
\[ |\partial K_{\textup{SL}}| \vee  |\partial K_{\textup{SV}}| \leq 2d \#(J_{\textup{SLV}} \cup J_B) r^{d-1} .\]

\medskip

Consider now $\#(J_{\textup{SLV}})$, the number of $\Box_k$ which are near the contact line.  If $\Box_k$ is near the contact line then, from the definition, there is a point $(x,r) \in \partial_* L _\e\cap 2\Box_k^{r}$ and so, by the interior density estimates Lemma~\ref{lem: meas reg} we have
\[
\Per(L_\e, 3Q_k \cap \{z \geq \tfrac{1}{2} r  \}) \geq c (\tfrac{1}{2} r)^d.
\]
On the other hand, since $\cup_{k \in J_{\textup{SLV}}} 3Q_k$ is finite overlapping, by Proposition~\ref{prop: perim 0 t}
\[ c\#(J_{\textup{SLV}}) r^d \leq\sum_{k \in J_{\textup{SLV}}} \Per(L_\e, 3Q_k \cap \{z \geq \tfrac{1}{2}  r  \}) \leq C \Per(L_\e, \{ \tfrac{1}{2}  r  \leq z \leq \tfrac{3}{2}r\}) \leq C\textup{Vol}^{-\frac{1}{d+1}}r |U|,\]
as long as $r \geq r_0(\e) $, i.e. $R = r/\e \geq r_0(\e)/\e$, rearranging this we get,
\begin{equation}\label{eqn: hom bdry est}
 \#(J_{\textup{SLV}}) \leq C \textup{Vol}^{-\frac{1}{d+1}}r^{1-d}|U| \ \hbox{ and also } \ |\partial K_{\textup{SL}}| \vee  |\partial K_{\textup{SV}}| \leq C(\textup{Vol}^{-\frac{1}{d+1}}|U|+|\partial U|)
 \end{equation}
This is simply a statement that the contact line is $d-1$ dimensional if we look at a scale $r \geq r_0(\e)$.

\medskip

Now we modify $L_\e$ by replacing the profile in the wetted set by the optimal profile $L_{\textup{SL}}$ from the solid-liquid cell problem and in the non-wetted set by the optimal profile $L_{\textup{SV}}$ from the solid-vapor cell problem.  Recall that $L_{\textup{SL}}$ and $L_{\textup{SV}}$ are the $1$-periodic minimizers obtained in Lemma~\ref{maximal} which satisfy,
\[ {{E}}_{1\textup{-}{per}}(L_{\textup{SL}}) = \overline{\sigma}_{\textup{SL}} \ \hbox{ and } \ {{E}}_{1\textup{-}{per}}(L_{\textup{SV}}) = \overline{\sigma}_{\textup{SV}}.\]
We define,
\[
 \Lambda := \left[L_\e \setminus \bigg((K_{\textup{SL}} \cup K_{\textup{SV}}) \times (-\infty,r) \bigg)\right] \cup \left[\e L_{\textup{SL}} \cap \bigg(K_{\textup{SL}} \times (-\infty,r)\bigg)\right ]\cup\left [\e L_{\textup{SV}} \bigg(K_{\textup{SV}} \times (-\infty,r)\bigg)\right].
\]
We aim to estimate the energy difference ${{E}}(\Lambda,\Omega)-{{E}}(L_\e,\Omega)$.  As usual we use the minimization property of $L_\e$ to get the estimate in one direction, because we are altering the volume we need to have an estimate of the volume change $|L_\e \Delta \Lambda|$.  We can easily estimate, 
\[ |L_\e \Delta \Lambda| \leq |\{(x,z) \in U \times \real: \e \phi(\tfrac{x}{\e}) \leq z \leq r\}| \leq (r+M \e) |U|.\]
 Now we can use that $L_\e$ is the constrained global minimizer with volume $\textup{Vol}$ with the volume change monotonicity formula \eqref{eqn: volume change} to obtain,
 \[ {{E}}_\e(L_\e,\Omega) \leq {{E}}_\e(\Lambda,\Omega) + C\sigma_{\textup{LV}}\textup{Vol}^{-\frac{1}{d+1}}r|U|.\]

We compare the $L_{\textup{SL}}$ with $L_\e\cap \{z\leq r\}$ in $J_{\textup{SL}}$, and $L_{\textup{SV}}$ with $L_\e\cap \{z\leq r\}$ in $J_{\textup{SV}}$, and we conclude 
\begin{equation}\label{eqn: energy change care}
\begin{array}{lll}
{{E}}_\e(\Lambda,\Omega)  &\leq& E_\e(\Lambda, K_{\textup{SL}} \times (-\infty,r)) + E_\e(\Lambda, K_{\textup{SL}} \times (-\infty,r)) +E_\e(\Lambda, \Omega \setminus [(K_{\textup{SL}} \cup K_{\textup{SV}}) \times (-\infty,r)]) \vspace{1.5mm}\\
& \leq& {{E}}_\e(L_\e,\Omega) + C\sigma_{\textup{LV}}\left(\textup{Vol}^{-\frac{1}{d+1}}|U| + |\partial U|\right)r
\end{array}
\end{equation}
More precisely we are using that, via Lemma~\ref{reduction} and Theorem~\ref{thm: cell gen periodic},
\[  
\begin{array}{lll}
E_\e(L_\e, K_{\textup{SL}} \times (-\infty,r)) &\geq& \e^d\Sigma_{\textup{SL}}(\tfrac{1}{\e}K_{\textup{SL}}) - C\sigma_{\textup{LV}}|\partial K_{\textup{SL}}|\e \vspace{1.5mm}\\
&\geq& \overline{\sigma}_{\textup{SL}}|K_{\textup{SL}}| - C\sigma_{\textup{LV}}|\partial K_{\textup{SL}}| \e(1+\log\frac{r}{\e}) \vspace{1.5mm}\\
&=& E_\e(\Lambda,K_{\textup{SL}} \times (-\infty,r))- C\sigma_{\textup{LV}}|\partial K_{\textup{SL}}|\e(1+\log\frac{r}{\e}).
\end{array}
    \]
The analogous result holds also over the non-wetted set $K_{\textup{SV}}$.  Combining those two estimates and using \eqref{eqn: hom bdry est} to estimate $|\partial K_{\textup{SL}}|$ and $|\partial K_{\textup{SV}}|$ we get \eqref{eqn: energy change care}.

\medskip

Finally we have perturbed $L_\e$ to a set $\Lambda$ with,
\begin{equation}\label{eqn: est with lambda L}
 |L_\e \Delta \Lambda| \leq Cr|U| \ \hbox{ and } \ |{{E}}_\e(L_\e,\Omega)- {{E}}_\e(\Lambda,\Omega)| \leq C\sigma_{\textup{LV}}\left(\textup{Vol}^{-\frac{1}{d+1}}|U| + |\partial U|\right)r,
 \end{equation}
and ${{E}}_\e(\Lambda,\Omega)$ is close to $\overline{{E}}(\Lambda^+,\Omega)$ where $\Lambda^+:= \Lambda\cap\{z \geq 0\}$,
\begin{equation}\label{eqn: cant be helped}
|{{E}}_\e(\Lambda,\Omega) - \overline{{E}}(\Lambda^+,\Omega)| \leq \sigma_{\textup{LV}}(\#(J_{\textup{SLV}})+\#(J_B))r^d \leq C\sigma_{\textup{LV}}\left(\textup{Vol}^{-\frac{1}{d+1}}|U| + |\partial U|\right)r
\end{equation}

\medskip

It remains to show that $\overline{{E}}(\Lambda^+,\Omega)$ is close to $\overline{{E}}(L_0)$. One direction is easy, since $\Lambda^+$ is admissible for the volume constrained minimization problem for $\overline{E}$ with surface $\{z\leq 0\}$ (up to an error in the volume):
\begin{equation}\label{eqn: lambda L0 1 dir}
 \overline{E}(L_0,\Omega) \leq \overline{E}(\Lambda^+,\Omega) +C\textup{Vol}^{-\frac{1}{d+1}}r|U|. 
 \end{equation}
 
 \medskip
 
For the other direction we modify $L_0$ to construct $L_{0,\e} \subseteq \Omega \setminus S_\e$ such that
\[ L_{0,\e} \cap \{ z \geq 0 \} = L_0 \ \hbox{ and } \ \ |{{E}}(L_{0,\e},\Omega) - \overline{{E}}(L_0,\Omega)| \leq C|U|\e.\]
 By assumption $U$ contains some ball of radius $\rho_0$, so $L_0$ does not intersect with $\partial U$.  As we did for $L_\e$ we define the contact area of $L_0$ at scale $r$ in terms of,
\[ J^0_{\textup{SL}} :=  \{ k: \Box_k^r \subset L_0\}, \ J^0_{\textup{SV}} :=  \{ k: \Box_k^r \subset \Omega \setminus L_0\}, \ J^0_{\textup{B}} := \{ k: 3\Box_k  \cap \partial U \neq \emptyset\}.\]  
As before the contact line cells $J_{\textup{SLV}}^0$ are all the remaining indices $k$ so that $\Box_k \cap U \neq \emptyset$. From the perimeter estimate Proposition~\ref{prop: perim 0 t} for the homogenized problem and the finite perimeter of $U$,
\[ \#(J_{\textup{SLV}}^0) \leq C\textup{Vol}^{-\frac{1}{d+1}}r^{1-d} |U| \ \hbox{ and } \  \#(J_{\textup{B}}^0) \leq Cr^{1-d} |\partial U|.\]
Then we modify $L_0$ to be admissible for the $\e$-problem as before using the cell problem solutions,
\[
L_{0,\e}:= L_0\cup\left (\e L_{SV} \cap\bigcup_{k\in J^0_{\textup{SV}}} Q_k\right) \cup \left (\e L_{SL} \cap\bigcup_{k\in J_{\textup{SL}}^0} Q_k\right).
\]
Note that the energies ${{E}}_\e(L_{0,\e},\Omega)$ and $\overline{{E}}(L_0,\Omega)$ are close,
\[ {{E}}_\e(L_{0,\e},\Omega) - \overline{{E}}(L_0,\Omega) \leq C\sigma_{\textup{LV}}\left(\textup{Vol}^{-\frac{1}{d+1}}|U| + |\partial U|\right)r . \]
On the other hand $L_{0,\e}$ is admissible for the ${{E}}_\e$ problem with volume $|L_0| \leq |L_{0,\e}| \leq |L_0| + M\e|U|$. Using the volume change monotonicity formula \eqref{eqn: volume change} as usual,
\[
{{E}}_\e(L_\e,\Omega) \leq {{E}}_\e(L_{0,\e},\Omega) + M\sigma_{\textup{LV}}\textup{Vol}^{-\frac{1}{d+1}}r|U|\leq \overline{{E}}(L_0,\Omega)+C\sigma_{\textup{LV}}\left(\textup{Vol}^{-\frac{1}{d+1}}|U| + |\partial U|\right)r.
\]
Then combining this last estimate with \eqref{eqn: est with lambda L} and \eqref{eqn: cant be helped},
\[
\overline{{E}}(\Lambda^+,\Omega) \leq \overline{{E}}(L_0,\Omega)+C\sigma_{\textup{LV}}\left(\textup{Vol}^{-\frac{1}{d+1}}|U| + |\partial U|\right)r.
\]
Combining this estimate with \eqref{eqn: lambda L0 1 dir} gives the claimed estimate of $|\overline{{E}}(\Lambda^+,\Omega) - \overline{{E}}(L_0,\Omega)|$.

\end{proof}

\medskip

\begin{proof}[Proof of Lemma~\ref{lem: uniform convergence}]
Let $C_0,C_1>0$ to be chosen large enough (universal) and call $r = C_0|L_1 \Delta L_2|^{\frac{1}{d+1}}$ and let ${h} \geq C_1\e^\beta|L_1 \Delta L_2|^{\frac{1-\beta}{d+1}}$. Suppose that $(x_1,z_1) \in L_1 \cap \{ z \geq {h}\}$ with $d((x_1,z_1), L_2 \cap \{ z \geq {h}\}) \geq r$.  From the choice of $h$, 
\[ \frac{h}{r}  = C_1 (\frac{\e}{r})^{\beta}\]
which is exactly chosen so that (when $C_1$ is large enough) Proposition~\ref{prop: boundary density estimates} implies,
\[ |L_1 \Delta L_2 | \geq c r^{d+1} \]
for a (possibly smaller than before) universal $c$ from Proposition~\ref{prop: boundary density estimates} so we have a contradiction plugging in the choice of $r$ making $C_0$ larger if necessary so that $C_0^{d+1} > 1/c$.

\medskip

The same argument applies for $(x_2,z_2) \in L_2 \cap \{ z \geq {h}\}$ with $d((x_2,z_2), L_1 \cap \{ z \geq {h}\}) \geq r$.

\end{proof}

\appendix

\section{}

\addtocontents{toc}{\SkipTocEntry}
\subsection{Proof of Lemma~\ref{lem: meas reg}}\label{proof: meas reg}
This proof appears in \cite{CM}, we repeat it here for completeness and because of minor technical differences.  We write $B_r$ for $B_r(x,z) \cap \Omega$ to simplify the notation.  Define the following quantities,
\begin{equation}
V_1(r) = | L \cap B_r|, \  S_1(r) = \mathcal{H}^d(L \cap \partial B_r) \ \hbox{ and } \ V_2(r) = | B_r \setminus L|, \  S_2(r) = \mathcal{H}^d( \partial B_r \setminus L ).
\end{equation}
 From the co-area formula,
\[ V_j'(r) = S_j(r).\]
If we consider the change in energy by removing $L \cap B_r$ from the minimizer $L$ with volume $\textup{Vol}$ we obtain
\[ {{E}}'(L) - \textup{Vol}^{-\frac{1}{d+1}}V(r) \leq {{E}}'(L \setminus B_r) = {{E}}'(L) - \Per(L,B_r)+S_1(r) \]
This estimate of course relies on $r < z$.  Rearranging and using $V_1(r) \leq \textup{Vol}$ we obtain,
\begin{equation}\label{eqn: reg energy est}
 \Per(L,B_r) \leq S_1(r) + \textup{Vol}^{-\frac{1}{d+1}}V_1(r) \leq S_1(r) + V_1(r)^{\frac{d}{d+1}}.
 \end{equation}
Using \eqref{eqn: reg energy est} in combination with the isoperimetric inequality we obtain differential inequalities for $V_j$,
\begin{equation}\label{eqn: ode V1}
V_1'(r) \geq \frac{1}{2}(S_1(r)+\Per(L,B_r))- \frac{1}{2}V_1(r)^{\frac{d}{d+1}} \geq \frac{1}{2}(\mu_{d+1}-1)V_1(r)^{\frac{d}{d+1}} \geq cV_1(r)^{\frac{d}{d+1}}
\end{equation}
since $\mu_{d+1}-1 = c(d) >0$.  Now since $V_1(r) >0$ for all $r>0$ from the definition $(x,z)$ being in the essential boundary $\partial_eL$, the differential inequality implies that,
\[V_1(r) \geq c(d)r^{d+1},\]
which is part of the desired result.

\medskip

We need to be a bit more careful with unioning on $B_r$ since that may not preserve the spatial constraint.  If we consider the change in energy by adding $B_r \cap \Omega$ to $L$,
\begin{equation}\notag
{{E}}'(L)  \leq {{E}}'(L \cup B_r \cap \Omega) =
\left\{\begin{array}{lll}
{{E}}'(L) - \Per(L,B_r)+S_2(r) & \hbox{ when } & d(x,\partial U) \leq r \\
 {{E}}'(L) - \Per(L,B_r)+\mathcal{H}^d(\partial (B_r \cap \Omega) \setminus L)  & \hbox{ when } & d(x,\partial U) > r
\end{array}\right.
\end{equation}
  Rearranging we get,
\[ \Per(L,B_r) \leq \max\{S_2(r),\mathcal{H}^d(\partial (B_r \cap \Omega)\} \leq \max\{ \mathcal{H}^d(\partial B_r), \mathcal{H}^d(\partial (B_r \cap \Omega))\}.\]
Either way implies, by the smoothness of $\partial \Omega$, the desired upper bound on the perimeter,
\[ \Per(L,B_r) \leq C r^d.\]
We need also a slightly different argument to obtain $V_2(r) \geq c(d)r^{d+1}$ -- if $d(x,\partial \Omega) < r/2$ then from the regularity of $\partial \Omega$,
\[ V_2(r) \geq |B_r \setminus \Omega| \geq c r^{d+1}.\]
If $d(x,\partial \Omega) \geq r/2$ then we can argue as before in \eqref{eqn: ode V1},
\begin{equation}
\hbox{for } \ 0 \leq t \leq r/2 \ \hbox{ we have } V_2'(t) \geq \frac{1}{2}(S_2(t)+\Per(L,B_t)) \geq \frac{1}{2}\mu_{d+1}V_2(t)^{\frac{d}{d+1}} \geq cV_2(t)^{\frac{d}{d+1}}
\end{equation}
Now since $V_2(t) >0$ for all $t>0$ from the definition $(x,z)$ being in the essential boundary $\partial_eL$, the differential inequality implies that,
\[V_2(r) \geq V_2(r/2) \geq cr^{d+1},\]
which is the upper bound on the volume of $L$ in $B_r$.

\medskip

Finally we aim for the lower bound on $\Per(L,B_r)$, 
$$(V_1(r) + V_2(r))^{\frac{d}{d+1}} = |B_r|^{\frac{d}{d+1}} = \mu_{d+1}|\partial B_r| = \mu_{d+1}(S_1(r)+S_2(r))$$
and by isoperimetric inequality
$$ V_1(r)^{\frac{d}{d+1}}+V_2(r)^{\frac{d}{d+1}} \leq \mu_{d+1}(S_1(r)+S_2(r) + 2\Per(L,B_r)).$$
Subtracting the the first inequality from the second,
$$ 2\mu_{d+1}\Per(L,B_r) \geq V_1(r)^{\frac{d}{d+1}}+V_2(r)^{\frac{d}{d+1}} - (V_1(r) + V_2(r))^{\frac{d}{d+1}} \geq c_d \min\{ V_1(r),V_2(r) \}^{\frac{d}{d+1}} \geq c_d r^d,$$
here we have used the simple inequality $a^{\frac{d}{d+1}}+b^{\frac{d}{d+1}} - (a+b)^{\frac{d}{d+1}} \geq (2-2^{\frac{d}{d+1}}) \min\{a,b\}$ for all $a,b>0$.

\addtocontents{toc}{\SkipTocEntry}
\subsection{Proof of Lemma~\ref{lem: vertical horizontal}}\label{proof: vertical horizontal}

  We work with a specific $\Gamma_j$ and therefore drop the $j$ for the remainder of the proof, furthermore there is no loss in considering that $(x_j,z_j) = (0,0)$.  

\medskip

From the definition of $\partial E$ being vertical we have that $n = (n',n_{d+1})$ with $|n_{d+1}| < 2/\gamma \leq \frac{1}{2}$, in particular $n \cdot n' \geq \frac{1}{2}$.  We consider the half-spaces 
\[ H_\pm = \{ (x,z) \cdot n  \leq \pm \delta \}.\]
Let $\psi$ be a smooth function with $\psi = 1$ on $\frac{1}{2} \Gamma$ and $\psi = 0$ on $\real^d \setminus \Gamma$ with $|\grad \psi| \leq \frac{C(d)}{r}$.  Now we compute,
\begin{equation*}
\begin{array}{lll}
 \int_{E \setminus H_-} \grad \cdot (n' \psi) \ dxdz &=& \int_{\partial E \cap \Gamma} \psi n' \cdot \nu_E \ d\mathcal{H}^{d} - \int_{\partial H_- \cap \Gamma} \psi n' \cdot n \ d\mathcal{H}^{d}  \vspace{1.5mm}\\
 &\leq& \Per(E,\Gamma) - \tfrac{1}{2} \mathcal{H}^d(\partial H_- \cap \frac{1}{2}\Gamma) 
 \end{array}
 \end{equation*}
 Since $|\grad \cdot (\psi n' )|\leq C(d)/r$ and 
 \[|(E \setminus H_-) \cap \Gamma| \leq |(H_+ \setminus H_-) \cap \Gamma| \leq C(d)\delta \gamma r^{d+1}\]
  we rearrange to obtain,
 \[ \frac{1}{2} \mathcal{H}^d(\partial H_- \cap \tfrac{1}{2}\Gamma) - C(d)\delta \gamma r^{d} 
\leq \Per(E,\Gamma) . \]
Now we show that for $\delta \leq 1/4$,
 \[ \mathcal{H}^d(\partial H_- \cap \tfrac{1}{2}\Gamma) \geq c(d)\gamma r^d . \]
 So, as long as $\delta \leq c_1$ for some dimensional constant $c_1(d)$, we obtain for some other constant $c_0(d)$,
 \[ \Per(E,\Gamma) \geq c_0(d)\gamma r^d.\]
 That completes the proof of the first part of the Lemma.

\medskip

 Now we consider the case $\partial E$ is horizontal in $\Gamma$ then 
 \[ |n \cdot e_{d+1}| \geq 2\gamma^{-1} \ \hbox{ and } \ |n'| \leq 1-2\gamma^{-1}\]
 We just need to show that $\partial E$ cannot intersect the top or bottom of $\Gamma$ or $3\Gamma$. If $(x,z) \in \partial E \cap \Gamma$ then we have:
 \[ 2\gamma^{-1}|z| - (1-2\gamma^{-1})r\leq | n \cdot (x,z)  | \leq   3\delta r. \]
Rearranging the above we obtain,
\[ |z| \leq (\tfrac{3}{2}\delta \gamma + (\tfrac{1}{2}\gamma-1))r < \tfrac{1}{2} \gamma r \ \hbox{ if } \ \delta < \tfrac{2}{3}\gamma^{-1}.\]
The same argument shows that $\partial E$ does not intersect the top or bottom boundary caps of $3\Gamma$. \qed

\addtocontents{toc}{\SkipTocEntry}
\subsection{Proof of Proposition~\ref{prop: boundary density estimates}}\label{proof: boundary density estimates}
Let's prove $(i)$.  Assume that $x+te_{d+1} \in L$ with $t \geq C\e^\beta r^{1-\beta}$, then for $C$ a sufficiently large universal constant the converse of Corollary~\ref{cor: uniform non-degen} implies that,
\[ L \cap Q_{r/2}^+ \cap \{ z = \delta_0 r/2\} \neq \emptyset.\]
In particular there is some point $(x_0,\delta_0 r/2)$ in that set.  Then using the interior density estimates Lemma~\ref{lem: meas reg} in $Q_{\delta_0 r/2}(x_0,\delta_0 r/2)$ we obtain,
\[ |L \cap Q_r^+(x,t)| \geq |L \cap Q_{\delta_0 r/2}(x_0,\delta_0 r/2)| \gtrsim r^{d+1}.\]
A similar argument applies to $(ii)$, and for part $(iii)$ we use the perimeter lower bound from Lemma~\ref{lem: meas reg} in a similar way. \qed

\bibliographystyle{plain}
\bibliography{capillary_articles.bib}

\end{document}